\documentclass[reqno,11pt]{article}

\usepackage{appendix}
\usepackage{amsfonts}
\usepackage{mathrsfs}
\usepackage{amsmath}
\usepackage{empheq}
\usepackage{mathtools}
\usepackage{graphicx}
\usepackage{graphicx}
\usepackage{booktabs}
\usepackage{amssymb}
\usepackage{cite}
\usepackage{epsfig}
\usepackage{epstopdf}
\usepackage[dvipsnames]{xcolor}
\usepackage{bm}
\usepackage{hyperref}
\usepackage{enumitem}
\usepackage{xspace}
\usepackage{microtype}
\usepackage{latexsym,amsmath,amssymb,amsthm,graphics,amscd,verbatim,mathrsfs}

\usepackage{algorithm}
\usepackage{algorithmicx}
\usepackage{algpseudocode}

\usepackage{amsmath}

\usepackage[margin=1.3in]{geometry}

\usepackage{array, tabularx, caption, boldline}
\usepackage{graphicx}
\usepackage{cellspace}

\DeclareMathOperator*{\argmax}{arg\;max}
\DeclareMathOperator*{\argmin}{arg\;min}

\DeclareMathOperator{\dist}{dist}

\DeclareMathOperator{\R}{\mathbb R}

\DeclareMathOperator{\Fix}{Fix}

\newcommand{\norm}[1]{\|#1\|}

\newcommand{\Exp}{\mathbb{E}}

\newcommand{\parti}[1]{\frac{\partial f}{\partial x_{#1}}}
\newcommand{\eg}{{\it e.g.}}
\newcommand{\ie}{{\it i.e.}}
\newcommand{\as}{\text{a.s.}}

\newcommand{\cF}{\mathcal{F}}
\newcommand{\prt}[1]{\left(#1\right)}
\newcommand{\srt}[1]{\left[#1\right]}
\def\B{S}

\newcommand{\cona}{c_1}
\newcommand{\conb}{c_2}
\newcommand{\conc}{c_3}
\newcommand{\range}{\mathrm{range}}

\theoremstyle{definition}

\newtheorem{assumption}{Assumption}

\newtheorem{theorem}{Theorem}
\newtheorem{lemma}{Lemma}
\newtheorem{proposition}{Proposition}

 \newtheorem{corollary}{Corollary}
\newtheorem{definition}{Definition}

\newtheorem{question}{Question}
\theoremstyle{definition}
\newtheorem{example}{Example}

\usepackage[labelformat=simple]{subcaption}

\definecolor{orangep}{RGB}{255,128,0}
\newcommand{\revise}[1]{#1}
\definecolor{redp}{RGB}{255,0,0}

\makeatletter
\def\@biblabel#1{#1.}
\renewcommand{\@seccntformat}[1]{\csname the#1\endcsname.\hspace{0.5em}}
\makeatother
\begin{document}
\title{An error bound-based convergence analysis framework\\ for a class of randomized algorithms}
\author{
Zhichun Yang\thanks{The Hong Kong Polytechnic University ({\tt zhichun.yang@polyu.edu.hk}).}
\and
Li Jiang\thanks{The Hong Kong Polytechnic University ({\tt jianglimath96@gmail.com}).}
\and
Tianxiang Liu\thanks{University of Tsukuba ({\tt liutx@sk.tsukuba.ac.jp})}
\and
Man-Chung Yue\thanks{The Hong Kong Polytechnic University ({\tt manchung.yue@polyu.edu.hk}).}
}

\date{}
\maketitle


\begin{abstract}
    \revise{
Classical analyses show that randomized coordinate descent (RCD) and gradient descent (GD) share the same convergence rates in terms of objective gap under convexity and specific global EB-type assumptions.
However, this rate preservation phenomenon does not extend to iterate rates, almost-sure rates, or general local error bound conditions. In this paper, we study this phenomenon systematically, not only for RCD but also for a much broader class of algorithms that share the same structures. Specifically, we introduce an abstract class of algorithms, called monotone randomized algorithms (MRAs), and new unified error bound (UEB) conditions that are independent of problem class and accommodate general EB moduli.
These UEB conditions subsume many EB- and Kurdyka--\L{}ojasiewicz-type assumptions used in the analysis of algorithms for optimization, convex feasibility, and common fixed point problems. We then develop a new analysis framework and a set of technical tools that yield nonasymptotic in-expectation rates and asymptotic almost-sure rates for both the optimality measure and the iterates under the global UEB condition. 
Under the local UEB condition, we further establish asymptotic almost-sure rates for both quantities. For comparison, we introduce the monotone deterministic algorithms (MDAs), as the corresponding deterministic counterpart of MRAs. We illustrate the tightness of our MRA analysis by showing that  the lower bounds of MDAs match the upper bounds of MRAs qualitatively. We also demonstrate the strength and versatility of our framework through three applications, namely generalized randomized subspace descent for unconstrained minimization, randomized block proximal gradient methods for composite optimization, and the randomized Krasnosel'ski\u{\i}--Mann method for common fixed point problems.
}
\end{abstract}

\section{Introduction}


In modern machine learning, randomized algorithms have emerged as the methods of choice for many real-world applications with high-dimensional decision variables, massive datasets, and/or a large number of constraints. Gaining a rigorous understanding of the convergence behavior of such algorithms is therefore a fundamental question with significant practical implications. The central goal of this paper is to systematically analyze a class of randomized algorithms satisfying certain structural properties under general notions of error bound conditions.


To motivate our study and highlight the challenges involved, let us begin by focusing on the randomized coordinate descent (RCD), a classical and widely used randomized algorithm for solving unconstrained smooth minimization problems. 
Owing to its simplicity and effectiveness in tackling high-dimensional problems, RCD and its variants have been extensively studied in the literature; see, for example,~\cite{Wright2015CoordinateDA, Cai2022CyclicBC, Xu2013ABC, Ptracu2013EfficientRC, Xu2017Globally, Nesterov2012Efficiency, Richtarik2014Iteration, Lu2015Complexity}.
\revise{For convex objectives, RCD is known to enjoy strong convergence guarantees~\cite{Nesterov2012Efficiency, beck2013convergence, Richtarik2014Iteration, Lu2015Complexity} in terms of the expected objective value comparable to those of deterministic gradient descent (GD): an $O(1/k)$ rate in the general convex case and a linear rate in the strongly convex case. Beyond convexity, linear convergence in expectation can still be obtained under a global Polyak--{\L}ojasiewicz (PL) condition~\cite{Ptracu2013EfficientRC,karimi2016linear}. These results suggest a rate-preservation phenomenon: at least in several special regimes, RCD preserves the qualitative convergence-rate behavior of its deterministic counterpart GD. However, the generality of this phenomenon remains unclear. Existing guarantees typically apply to expected objective values under global, specialized EB-type assumptions, leaving open whether comparable rates also hold for the iterates, along almost every trajectory, and under local EB geometry with general EB moduli. Almost-sure rates are practically relevant because they describe realized trajectories rather than average behaviour, while local error bounds with general moduli substantially broaden the theory beyond
specialized global geometries. This motivates the central question of our work: does the same rate-preservation principle persist in these broader settings? }

\revise{Compared with global error-bound conditions, analyzing RCD under local error-bound conditions is substantially more challenging. Such local assumptions are nevertheless important, since they hold for many more functions than their global counterparts; for example, local KL properties are satisfied by large classes of definable functions~\cite{Kurdyka1998OnGO,Bolte2007Clarke}. This direction has therefore attracted increasing attention in recent years.} For instance,~\cite{chouzenoux2016block, Xu2017Globally, latafat2022block, latafat2022bregman} establish convergence results based on local KL conditions under an essential cyclicity assumption, which requires that all coordinates be selected at least once within any fixed number of consecutive iterations. \revise{Nonetheless, as we show in Appendix~\ref{app:ess-cyclic}, this assumption fails almost surely under i.i.d. coordinate sampling, one of the most common sampling schemes in both practice and theory. Consequently, for genuinely randomized schemes such as i.i.d. sampling, sharp convergence rates for RCD under local KL geometry---and, in particular, whether they can match those of GD---remain largely open. 
}
At a deeper level, the convergence behaviour of RCD is tied to a small set of structural properties, namely, a decrease with respect to an optimality measure and certain inequalities relating the distance between consecutive iterates to a residual function.
These features are not unique to RCD but are shared by various randomized algorithms.
\revise{Notable examples include the randomized fast subspace descent (RFASD) methods \cite{chen2020randomized} and stochastic subspace descent \cite{kozak2021stochastic} for optimization problems, random projection method (RPM)~\cite{nedic2010} for convex feasibility problems and the random Krasnoselskii–Mann method (RKMM) for common fixed point problems.} With this observation, we endeavor to analyze a family of randomized algorithms in a unified manner under general EB or KL conditions. 

\revise{We begin by developing an abstract framework that captures the main structural properties of the algorithms of interest while remaining concrete enough for rate analysis. In particular, we will introduce an abstract class of stochastic algorithms, called monotone randomized algorithms (MRAs), which contains the examples discussed above as special cases but is not tied to any particular problem class. In parallel, we formulate unified global and local error-bound conditions with general EB moduli. These conditions are designed to be context-independent and flexible enough to encompass many commonly used EB- and KL-type assumptions arising in optimization, convex feasibility, and common fixed point problems. As a deterministic benchmark, we also introduce the corresponding monotone deterministic algorithms (MDAs). This allows the motivating question for RCD and GD to be lifted to an abstract level: under a comparable unified global or local EB geometry, does an MRA achieve the same convergence rate as its deterministic MDA counterpart? The main objective of this work is to answer this question systematically.}


We now summarize our main contributions.

\begin{itemize}
    \item \revise{We formulate an abstract class of stochastic algorithms defined by three inequalities relating two progress measures $h$ and $g$, called monotone randomized algorithms (MRAs), that} capture the key structural properties of several important methods, including RCD and its variants in optimization, RPM for convex feasibility problems, and RKMM for common fixed point problems. We introduce the unified global and local EB conditions for analyzing these algorithms. The proposed error bounds subsume many of EB- or KL-type conditions studied in the literature, including classical EB and KL conditions in unconstrained optimization\revise{~\cite{Kurdyka1998OnGO,Bolte2007Clarke}} and consistent error bounds for fixed point problems~\cite{liu2024convergence, Liu2024ConcreteCR}.

    \item \revise{We develop a new convergence rate analysis framework tailored to MRAs. Under global and local unified error-bound (UEB) conditions, this framework yields a unified set of convergence rates for both the progress measure $h$ and the iterates $x^k$. More precisely, under the global UEB condition, we establish nonasymptotic rates in expectation and asymptotic almost-sure rates; under the local UEB condition, we also obtain asymptotic almost-sure rates. Remarkably, in all cases covered by our theory, the rates of MRAs qualitatively match the worst-case lower bounds of the corresponding deterministic benchmark, namely the monotone deterministic algorithms (MDAs). This demonstrates both the sharpness of our analysis and the efficiency of MRA in a broad range of settings.}
    

    \item \revise{To demonstrate the versatility and strength of our framework, we apply our general results to generalized randomized subspace descent (GRSD) for unconstrained smooth minimization, randomized block proximal-gradient methods for smooth composite optimization, and RKMM for common fixed point problems. In particular, GRSD subsumes randomized subspace descent, block-coordinate descent, and coordinate descent. Although these results are obtained as special cases of our abstract analysis, they provide new convergence guarantees for all three classes of methods. }
    

    \item As a side result, Theorem~\ref{Th_KL_general} establishes that for any smooth definable function $f$ with a locally Lipschitz \revise{continuous} gradient, \revise{the modulus $\phi_{\bar{x}}=1/(\varphi')^2$ associated with a KL desingularization function $\varphi$ at any critical point $\bar{x}$ that is not a local maximizer satisfies $\phi_{\bar{x}}(t)=O(t)$ as $t\downarrow0$. In particular, this estimate implies that} $(\varphi')^2$ \revise{is not integrable} near the origin. An immediate consequence is that $f$ cannot satisfy the KL property with an exponent $\kappa \in [0,1/2)$ at any such point. This substantially generalizes existing results (\eg, \cite[Theorem~7]{bento2025convergence}), which are restricted to local minimizers and power-type desingularization functions.

\end{itemize}

\subsection{Notation}
The sets of real and non-negative real numbers are denoted by $\mathbb{R}$ and $\mathbb{R}_+$, respectively. The Euclidean norm is denoted by $\|\cdot\|$. For any integer $n>0$, we denote $[n] = \{1,\dots, n\}$. For any $\bar{x}\in \mathbb{R}^d$ and $\delta > 0$, $\mathbb{B}(\bar{x},\delta)$ denotes the open ball centered at $\bar{x}$ with radius $\delta$. For any set $B$, we denote by $\dist (x,B)$ the Euclidean distance from $x$ to the set $B$ and by $\mathbf{1}_B$ its indicator function, \ie, $\mathbf{1}_B (x) = 1$ if $x\in B$; and $\mathbf{1}_B (x) = 0$ otherwise. We also abbreviate ``almost surely'' as ``a.s.''.

\section{Problem Setup and Preliminaries}

\subsection{\revise{Monotone randomized and deterministic algorithms}}
\revise{Let $(\Omega,\cF,\{\mathcal{F}_k\}_{k\geq0},\mathbb{P})$ be a filtered probability space.} Let $\{x^k\}_{k\ge 0}\subseteq \mathbb{R}^d$ be a sequence of random
vectors \revise{adapted to $\{\mathcal{F}_k\}_{k\geq0}$}, with a
deterministic initial value $x^0$. For any random variable $X$ defined on the same probability space, we denote by $\mathbb{P}(X \mid \mathcal{F}_k)$ the regular version of the conditional probability distribution, \ie, $\mathbb{P}(X \mid \mathcal{F}_k)$ satisfies the definition of a conditional distribution, and for each $\omega \in \Omega$, the evaluation $\mathbb{P}(X \mid \mathcal{F}_k)(\omega)$ defines a probability distribution. The existence of a regular conditional distribution is guaranteed, \eg, by \cite[Page 269, Theorem 4]{shiryaev2016probability}.

This paper focuses on the following class of randomized algorithms.

\begin{assumption}[Monotone randomized algorithms (MRAs)]\label{assmp:basic}
There exist a closed set $\B \subseteq \mathbb{R}^d$, continuous functions $g,h : \B\to \R$ with $ g \ge 0$ and $ h^* := \inf_{x \in \B} h(x) > -\infty $, constants $ \cona, \conb, \conc >0$ with $\conc \revise{\ge} \conb$, and a positive sequence $\{\alpha_k\}_{k\ge 0} $ such that the following hold.
      \begin{enumerate}[label=(\roman*)]
        \item\label{basic_1} For all $k \ge 0$,
     \begin{equation}\label{eq:basic_1}
             h(x^{k+1})\le  h(x^k) - \frac{\cona}{\alpha_k} \|x^{k+1}-x^{k}\|^2,
     \end{equation}
     \begin{equation}
            \label{eq:expbound}
                \Exp[\|x^{k+1}-x^k\|^2|\cF_k]\ge  \conb \,\alpha_k^2\, g(x^k),\quad \text{a.s.},
        \end{equation} 
              \begin{equation}\label{eq:basic_4}
            \|x^{k+1}-x^k\|^2\le   \conc \, \alpha_k^2\,g(x^k).
        \end{equation}

        \item \label{basic_2} $\{x^k\}_{k\ge 0}\subseteq \B$.
    \end{enumerate}
\end{assumption}

In Assumption~\ref{assmp:basic}, the sequence $\{\alpha_k \}_{k\ge 0}$ plays the role of stepsizes \revise{(or rescaled stepsizes; see Sections~\ref{section_matrix_sketched_gd} and~\ref{section:smooth-random-bpg})} of the randomized algorithm, while the functions $g$ and $h$ represent two metrics for quantifying the progress of the algorithm. Our theoretical development will also rely on certain error bound conditions relating the metrics $g$ and $h$. The presence of the set $S$ is to allow the additional flexibility for accommodating the situation where the error bound condition holds only on a subset in $\mathbb{R}^d$. \revise{As a benchmark for assessing the effect of randomization, we introduce the deterministic counterpart of MRAs.}

\revise{\begin{assumption}[Monotone deterministic algorithms (MDAs)]\label{assmp:basic_det}
    The deterministic sequence $\{x^k\}_{k\ge 0}$ satisfies the conditions in
    Assumption~\ref{assmp:basic}, with the expected progress condition
    \eqref{eq:expbound} replaced by its pointwise deterministic counterpart:
    \begin{equation}\label{eq:det-bound}
        \|x^{k+1}-x^k\|^2
        \geq \conb\,\alpha_k^2\, g(x^k),
        \qquad \forall k\ge 0 .
    \end{equation}
\end{assumption}}

Below we illustrate our algorithmic framework with two representative examples.

\begin{example}\label{exam:SUM}
    Consider the unconstrained minimization problem
    \begin{equation}\label{pro_min}
        \min_{x\in \mathbb{R}^d}~f(x),
    \end{equation}
    where the objective function $f$ attains its minimum~$f^\star$ and has a Lipschitz continuous gradient. A popular randomized algorithm for solving problem~\eqref{pro_min} is the randomized coordinate descent (RCD):
    \begin{equation}\label{RCD}
        x^{k+1} = x^k - \revise{\beta_k} \, \parti{i_k} (x^k)\, e_{i_k} \quad\forall k \ge 0,
    \end{equation}
    where $\revise{\beta_k} > 0$ is the stepsize, $\{i_k\}_{k\ge 0}$ are independent and identically distributed (i.i.d.) random indices sampled according to the uniform distribution on $[d]$, and $e_i$ denotes the $i$-th standard basis in $\mathbb{R}^d$.
    Then, as we will see in Section~\ref{section_matrix_sketched_gd}, \revise{if each coordinate derivative $\parti{i}$ is $L_i$ Lipschitz continuous,} RCD satisfies Assumption~\ref{assmp:basic} with $h(x) = f(x)$, $g(x) = \|\nabla f(x)\|^2$ \revise{and $\beta_k = \frac{\alpha_k}{L_{i_k}}$}. \revise{In particular, standard gradient descent serves as the deterministic counterpart to RCD and satisfies Assumption~\ref{assmp:basic_det}.}
\end{example}

\begin{example}\label{exam:CFP}
     Consider the common fixed point problem
    \begin{equation}\label{CF-P}
        \begin{split}
            \text{find}~x\in F =\bigcap_{i=1}^m\Fix T_i, 
        \end{split}
    \end{equation}
    where each $T_i$ is a firmly quasi-nonexpansive operator on $\mathbb{R}^d$, and $\Fix T_i = \{x\in\mathbb{R}^d: T_i(x) = x\}$ is the set of its fixed points. A natural algorithm for solving problem~\eqref{CF-P} is the random Krasnoselskii–Mann method (RKMM): 
    \begin{equation}\label{alg:RKMM}
        x^{k+1} = x^k + \beta_k (T_{i_k} (x^k) - x^k)\quad\forall k\ge 0,
    \end{equation}
    where $\{\beta_k\}_{k\ge 0}$ are the stepsizes and $\{i_k\}_{k\ge 0}$ are independent and identically distributed (i.i.d.) random indices sampled according to the uniform distribution on $[m]$. Then, as we will prove in Section~\ref{section:RKMM}, \revise{if the operators $\{T_i\}_{i=1}^m$ are firmly quasi-nonexpansive}, RKMM satisfies Assumption~\ref{assmp:basic} with $h(x) = \dist^2(x,F)$, $g(x) = \max_{i\in [m]}\|x-T_i(x)\|^2$ \revise{and $\alpha_k = \beta_k$}. \revise{In particular, when $m=1$, RKMM coincides with the standard Krasnoselskii--Mann method (KMM), which serves as the deterministic counterpart and satisfies Assumption~\ref{assmp:basic_det}.}
\end{example}

\revise{
We impose the following condition on the stepsizes, which is satisfied by standard choices, such as constant stepsizes or $\alpha_k = c'/(k+1)^c$ with $c\in (0,1]$ and $c'>0$.
\begin{assumption}\label{assmp:step-size}
    For the sequence $\{\alpha_k\}_{k\geq0}$ defined in Assumption~\ref{assmp:basic}, there exists a constant $\bar{\alpha} > 0$ such that $0<\alpha_k \le \bar{\alpha}$ for all $k \ge 0$, and $\sum_{k=0}^\infty \alpha_k = +\infty$. 
\end{assumption}
}

The next lemma shows that the progress metric $g(x^k)$ vanishes almost surely.

\begin{lemma}\label{lem:g^k_vanishes}
    Suppose that Assumptions~\ref{assmp:basic}\ref{basic_1} and~\ref{assmp:step-size} hold and that $\Gamma\circ g$ is uniformly continuous for some strictly increasing function $\Gamma:\R_+\rightarrow \R$. Then, $g(x^k) \to 0$ almost surely.
\end{lemma}

\begin{proof}
    By inequalities~\eqref{eq:basic_1} and~\eqref{eq:expbound}, we have
    \[
    \sum_{k=0}^\infty \frac{1}{\alpha_k} \norm{x^{k+1}-x^k}^2 < +\infty \quad \text{and} \quad \sum_{k=0}^\infty \alpha_k\, \Exp[g(x^k)] < +\infty,
    \]
    which implies $\sum_{k=0}^\infty \alpha_k g(x^k) < +\infty$ almost surely. Define the event
    \[
    E = \left\{\omega: \sum_{k=0}^\infty \alpha_k g(x^k(\omega)) < +\infty \right\}.
    \]
    Then $\mathbb{P}(E)=1$. Fix any $\omega\in E$. Since $\sum_{k=1}^\infty \alpha_k = \infty$, we have $\liminf_{k\rightarrow\infty} g(x^k(\omega)) = 0$. Suppose that $\limsup_{k\rightarrow\infty} g(x^k(\omega)) > 0$. Then, there exists an $\epsilon>0$ and two index sequences $\{l_j\}_{j\ge 0} $ and $\{u_j\}_{j\ge 0}$ such that
    \[
    l_j < u_j < l_{j+1} \quad\forall j\geq 0,
    \]
    and that the corresponding subsequences $\{x^{l_j}\}_{j\ge 0}$ and $\{x^{u_j}\}_{j\ge 0}$ satisfy
    \[
    g(x^{l_j}(\omega)) \geq \epsilon,\quad   g(x^{u_j}(\omega))\leq \frac{\epsilon}{2}\quad \text{and}\quad g(x^{k}(\omega)) \geq \frac{\epsilon}{2}\quad \forall  l_j < k < u_j.
    \]
    Define $\gamma_j = \sum_{k=l_j}^{u_j-1} \alpha_k$. Since $\sum_{k=0}^{\infty} \alpha_kg(x^k)< +\infty$, $\gamma_j$ is summable. Then, as $j\rightarrow \infty$,
    \[
    \norm{x^{u_j}(\omega)-x^{l_j}(\omega)} \leq \sum_{k= l_j}^{u_j-1}  \norm{x^{k+1}(\omega)-x^{k}(\omega)} \leq
    {\sqrt{\gamma_j} \prt{\sum_{k= l_j}^{u_j-1} \frac{1}{\alpha_k} \norm{x^{k+1}(\omega)-x^{k}(\omega)}^2}^{\frac{1}{2}}\rightarrow 0.}
    \]
    On the other hand, by the uniform continuity of $\Gamma\circ g$, we have
    \[
    0<\Gamma(\epsilon)-\Gamma(\frac{\epsilon}{2})\leq \Gamma(g(x^{l_j}(\omega)))-\Gamma(g(x^{u_j}(\omega)))\rightarrow 0,\quad\text{as $j\rightarrow \infty$,}
    \]
    which yields a contradiction. Thus, $g(x^k(\omega)) \to 0$ for all $\omega \in E$.
\end{proof}

\subsection{Auxiliaries}\label{sec:Auxiliaries}
We next collect some preparatory results for our theoretical development.

\begin{definition}[Inverse smoothing and desingularization functions]\label{def:inverse-smoothing}
    Given any $\tau > 0$ and function $\phi: [0,\tau) \rightarrow \mathbb{R}_+$ that is \revise{non-decreasing and strictly positive on $(0,\tau)$}, we define the following functions.
    \begin{enumerate}[label=\textup{\textrm{(\roman*)}}]
        \item Fix a $t_0\in (0,\tau)$,\footnote{The results in this paper are independent of the choice of $t_0$.} define $\Psi:[0,\tau)\rightarrow \revise{\R\cup \{+\infty\}}$ such that $\Psi(t) = \int_{t}^{t_0} \frac{1}{\phi(s)} ds$.
        
        \item  If $1/\sqrt{\phi}$ is integrable near $0$,\footnote{Given a function $f:[0,\tau)\rightarrow [0,+\infty]$ with $\tau>0$, we say $f$ is integrable near $0$ if there exists a number $\tau' \in (0,\tau)$ such that $f$ is integrable on $[0, \tau']$.}define ${\Phi}:[0,\tau)\rightarrow \R_+$ such that $\Phi(t) = \int_{0}^t \frac{1}{\sqrt{\phi(s)}} ds$.
    \end{enumerate}
\end{definition}

The function $\Psi$ is inspired by the inverse smoothing function~\cite[Equation (4.3)]{liu2024convergence}, while the function $\Phi$ serves a similar purpose to the desingularization function in the Kurdyka–Łojasiewicz (KL) condition.

To reflect the relevant context, the notation $(\phi,\Psi,\Phi)$ will be augmented with subscripts in the following sections. Specifically, it appears as $(\phi_{\B},\Psi_{\B},\Phi_{\B})$ in Section~\ref{sec:globalEB}, and as $(\phi_{\mathcal{A}(\omega)},\Psi_{\mathcal{A}(\omega)},\Phi_{\mathcal{A}(\omega)})$ and $(\phi_{x^*(\omega)},\Psi_{x^*(\omega)},\Phi_{x^*(\omega)})$ in Section~\ref{sec:localEB}.

The lemma below summarizes the properties of $\Psi$ and $\Phi$, which is partly established in~\cite[Proposition 4.3]{liu2024convergence}. We provide a proof for self-containedness.
\begin{lemma}\label{lem:functions}
For the function $\phi$ defined in Definition~\ref{def:inverse-smoothing}, the following properties hold:
    \begin{enumerate}[label=\textup{\textrm{(\roman*)}}]
        \item\label{lem_func_1}
        The function $\Psi$ is convex and decreasing on $[0,\tau)$, and real-valued \revise{on $(0, \tau)$ with 
        \[
        -\frac{1}{\phi(t)}\in\partial \Psi(t) = \left[-\frac{1}{\phi(t-)},-\frac{1}{\phi(t+)}\right],
        \]
        where $\phi(t-) = \lim_{s\uparrow t} \phi(s)$ and $\phi(t+) = \lim_{s\downarrow t} \phi(s)$.} If $1/\phi$ is integrable near 0, then $\Psi(t) \uparrow \Psi(0) < +\infty$ as $t \downarrow 0$; otherwise, $\Psi(t) \uparrow \Psi(0) = +\infty$ as $t \downarrow 0$.
        \item\label{lem_func_2} If $1/\sqrt{\phi}$ is integrable near $0$, then $\Phi$ is continuous, concave, and increasing on $[0,\tau)$, with $\Phi(0) = 0$, and 
        \revise{\[
        -\frac{1}{\sqrt{\phi(t)}}\in \partial [-\Phi(t)] = \left[-\frac{1}{\sqrt{\phi(t-)}},-\frac{1}{\sqrt{\phi(t+)}}\right]
        \]}
        on $(0,\tau)$.
    \end{enumerate}
\end{lemma}

\begin{proof}
    We prove only assertion~\ref{lem_func_1} as assertion~\ref{lem_func_2} can be proved similarly. 
    Since the function $\phi$ is positive and non-decreasing, 
    \revise{the function $-1/\phi$ is negative and non-decreasing. Since
    \[
        \Psi(t)=\int_{t_0}^{t}-\frac{1}{\phi(s)}\,ds,
    \]
    it follows from~\cite[Theorems~24.1--24.2]{Rockafellar1970}
    that $\Psi$ is decreasing and convex, with
    \[
        \Psi'_-(t)=-\frac{1}{\phi(t-)},
        \qquad
        \Psi'_+(t)=-\frac{1}{\phi(t+)},
    \]
    and
    \[
        \partial\Psi(t)
        =
        [\Psi'_-(t),\Psi'_+(t)]
        =
        \left[
            -\frac{1}{\phi(t-)},
            -\frac{1}{\phi(t+)}
        \right].
    \]
    Since $\phi(t-)\leq\phi(t)\leq\phi(t+)$, it follows that
    $-1/\phi(t)\in\partial\Psi(t)$.}
    By the monotone convergence theorem, we have
    $\Psi(t)\uparrow\Psi(0)$ as $t\downarrow0$. The proof is completed by
    noting that $\Psi(0)<+\infty$ if and only if $1/\phi$ is integrable
    near zero.
\end{proof}


Lemma~\ref{lem:functions} implies that when $1/\phi$ is not integrable near zero, the range of $\Psi$, and thus the domain of $\Psi^{-1}$, is exactly $(\lim_{t \uparrow \tau} \Psi(t), +\infty]$. Therefore, for any $t \in (0, \tau)$, the interval $[\Psi(t), +\infty)$ lies entirely within the domain of $\Psi^{-1}$.

We will also use the following elementary result.
\begin{lemma}\label{lem:rate-preliminary}
For the function $\phi$ defined in Definition~\ref{def:inverse-smoothing}, the following properties hold:
    \begin{enumerate}[label=\textup{\textrm{(\roman*)}}]
        \item \label{rate-preliminary1} For any $a,\tilde{a}\in (0,\tau)$ and $b\geq 0$, if $\tilde{a}\leq a - b \phi(a)$, then 
        $\Psi(\tilde{a}) \geq \Psi(a) + b$. If $1/\phi$ is integrable near $0$, then this conclusion also holds for $\tilde{a} = 0$, $a\in (0,\tau)$.
        \item \label{rate-preliminary2} Suppose further that $1/\sqrt{\phi}$ is integrable near $0$. For any $a\in (0,\tau)$, $\tilde{a} \geq 0$ and $\Delta \ge 0$, if $\tilde{a}\leq a - \Delta \sqrt{\phi(a)}$, then $\Delta \leq \Phi(a) - \Phi(\tilde{a})$.
    \end{enumerate}
\end{lemma}

\begin{proof}
    The recursion in~\ref{rate-preliminary1} can be reformulated as
    \[
    -\frac{1}{\phi(a)}(\tilde{a}-a) \geq b.
    \] 
    \revise{By Lemma~\ref{lem:functions}, $\Psi$ is convex and
    $-1/\phi(a)\in\partial\Psi(a)$. We therefore have}
    \[
        \Psi(\tilde{a})-\Psi(a)
        \geq -\frac{1}{\phi(a)}(\tilde{a}-a)
        \geq b.
    \]
    For the case where $1/\phi$ is integrable near $0$ and $\tilde{a}=0$, we use the continuity of $\Psi$ at $0$ established in Lemma~\ref{lem:functions}. Taking the limit as $\tilde{a}\downarrow 0$ in the preceding inequality yields the result. 
    
    As for the recursion in~\ref{rate-preliminary2}, \revise{Lemma~\ref{lem:functions} shows that $-\Phi$ is convex and
    $-1/\sqrt{\phi(a)}\in\partial[-\Phi](a)$. Hence, by the subgradient inequality,}
    \[
        \Delta
        \leq \frac{a-\tilde{a}}{\sqrt{\phi(a)}}
        \leq \Phi(a)-\Phi(\tilde{a}),
    \]
    which completes the proof.
\end{proof}

\revise{\section{Unified Error Bound Conditions}}
\revise{A central objective of this paper is to characterize the convergence behavior of MRAs defined in Assumption~\ref{assmp:basic} across a broad range of geometric regimes. Towards that end, we introduce global and local unified error-bound (UEB) conditions. The qualifier ``unified'' reflects two features. First, the conditions are formulated at the level of the abstract algorithmic framework and are therefore independent of any particular problem class or algorithmic instance. Second, inspired by nonlinear metric regularity~\cite{Ioffe2013Nonlinear} and nonlinear metric subregularity~\cite{kruger2016nonlinear, li2026generalized}, they encode different error-bound geometries through an abstract modulus function $\phi$, thereby allowing diverse convergence regimes to be treated within the same framework.
}
We first focus on the situation where the error bound condition holds globally over the domain $S$.

\begin{assumption}[Global unified error bound (GUEB)]\label{assmp:Global-EB}
    There exists a function $\phi_{\B}: [0,\tau) \rightarrow \mathbb{R}_+$, where \revise{$\tau \in (0,+\infty]$}, such that the following hold.
      \begin{enumerate}[label=(\roman*)]
        \item\label{ass_g_E_1} For any $x\in \B$,
        \begin{equation}\label{eq:globalEB}
        \phi_{\B}(h(x)-h^*) \leq g(x).
        \end{equation}
        \item\label{ass_g_E_2} The function $\phi_{\B}$ is \revise{non-decreasing and strictly positive on $(0,\tau)$ with $\phi_\B(0) = 0$}.
        \item\label{ass_g_E_3} $0\leq h(x)-h^*< \tau$ for all $x\in \B$.
    \end{enumerate}
\end{assumption}

Assumption~\ref{assmp:Global-EB} encompasses many existing EB- or KL-type conditions across various contexts.  
For instance, consider the smooth unconstrained minimization problem~\eqref{pro_min} in Example~\ref{exam:SUM} and recall the choice of $h(x) = f(x)$ and $g(x) = \|\nabla f(x)\|^2$. \revise{In this case, Assumption~\ref{assmp:Global-EB} is equivalent to a global KL condition~\cite{Bolte2017Errora,fontaine2021convergence,fatkhullin2022sharp} with desingularizing function $\Phi(t) = \int_{0}^t \frac{1}{\sqrt{\phi_\B(s)}} ds$. When $\phi_\B$ is a power-type modulus, Assumption~\ref{assmp:Global-EB} recovers global KL or generalized Polyak--{\L}ojasiewicz-type conditions with corresponding exponents~\cite{fatkhullin2022sharp,jiang2026generalizedversionchungslemma}. In particular, if $\phi_\B$ is linear, it further reduces to the standard Polyak--{\L}ojasiewicz condition~\cite{polyak1963gradient,karimi2016linear,NecoaraNesterovGlineur2019Linear,YueFangLin2023LowerBoundPL}.} Revisit the common fixed point problem~\eqref{CF-P} in Example~\ref{exam:CFP} and recall our choice of $h(x) = \dist^2(x,F)$ and $g(x) = \max_{i\in [m]}\|x-T_i(x)\|^2$. \revise{This setting also covers EB conditions for convex feasibility problems: when each $T_i$ is the projection onto a convex set $C_i$, one has $\Fix T_i=C_i$, and the goal is to find a point in the intersection $\bigcap_{i=1}^m C_i$. In this context, our GUEB condition coincides with the metric subregularity of the residual mapping $I-T$ for $0$ on $\B$ with gauge $\varphi$ \cite[Definition 2.5]{LukeThaoTam2018Quantitative}. Under certain growth assumptions on $\phi_\B$, Assumption~\ref{assmp:Global-EB} reduces to the joint Karamata regularity of the operators $\{T_i\}_{i=1}^m$~\cite{Liu2024ConcreteCR}. In the convex feasibility
setting, this specializes to the consistent error-bound condition
introduced in~\cite{liu2024convergence}. Choosing \(\phi_{\B}\) to be a power-type modulus further recovers H\"olderian error bounds~\cite{borwein2014analysis,borwein2017convergence,liu2024convergence,Liu2024ConcreteCR}. When \(\phi_{\B}\) is linear,
Assumption~\ref{assmp:Global-EB} reduces to linear regularity for
intersections of convex sets~\cite{BauschkeBorwein1993VonNeumann,
bauschke1996projection,
BauschkeBorweinLewis1997CyclicProjections,
BauschkeBorweinLi1999StrongCHIP,
BauschkeBorweinTseng2000BoundedLinearRegularity,
SongZang2006BoundedLinearRegularity,
Kruger2006RegularityCollectionsSets} and to its operator-theoretic counterpart for common fixed point problems~\cite{Cegielski2012IterativeMethods,
BauschkeNollPhan2015AveragedOperators,
borwein2017convergence,
CegielskiReichZalas2018RegularSequences,hermer2019random}.} 


\revise{\begin{theorem}[Existence of $\phi_\B$]\label{thm:GUEB-existence}
   Let $\B \subseteq \mathbb{R}^d$ be a compact set and $g,h : \B\to \R$ be two continuous functions with $ g \ge 0$ and $ h^* \coloneqq \inf_{x \in \B} h(x) > -\infty $. Suppose that $g(x) = 0$ implies $h(x) = h^*$. Then there exists a function $\phi_\B$: $ \mathbb{R}_+ \to\mathbb{R}_+$ such that the following statements hold.
   \begin{itemize}
       \item[(i)] For any $x\in\B$,
       \begin{equation*}
           \phi_\B(h(x) - h^*)\le g(x).
       \end{equation*}

       \item[(ii)] The function $\phi_\B$ is non-decreasing and strictly positive on $\mathbb{R}_{++}$ .
   \end{itemize}
\end{theorem}

\begin{proof}
Replacing $h$ by $h-h^*$, we may assume without loss of generality that $h^*=0$, and hence $h\geq0$ on $\B$. Let \(\Delta\coloneqq\max_{x\in\B}h(x)\), which is finite since \(\B\) is compact and \(h\) is continuous. If \(\Delta=0\), then \(h\equiv0\) on \(\B\), and the choice
\(\phi_\B(t)=t\) for all \(t\in\mathbb R_+\) satisfies the conclusions. Suppose that \(\Delta>0\). For \(t\in(0,\Delta]\), define
\[
    \Omega_t\coloneqq\{x\in\B:h(x)\ge t\},
    \qquad
    \widehat{\phi}_\B(t)\coloneqq\min_{x\in\Omega_t}g(x).
\]
The set \(\Omega_t\) is nonempty and compact, so
\(\widehat{\phi}_\B(t)\) is well defined. Moreover, if
\(0<t_1\le t_2\le\Delta\), then \(\Omega_{t_2}\subseteq\Omega_{t_1}\), and
hence \(\widehat{\phi}_\B(t_1)\le\widehat{\phi}_\B(t_2)\). Thus
\(\widehat{\phi}_\B\) is non-decreasing. For every \(t\in(0,\Delta]\), one also has
\(\widehat{\phi}_\B(t)>0\). Indeed, \(h(x)\ge t>0\) implies \(g(x)>0\) by
the assumption \(g(x)=0\Rightarrow h(x)=0\), and therefore the continuous
function \(g\) has a strictly positive minimum on the compact set
\(\Omega_t\). Define \(\phi_\B:\mathbb R_+\to\mathbb R_+\) by
\[
    \phi_\B(t)\coloneqq
    \begin{cases}
        0, & t=0,\\
        \widehat{\phi}_\B(t), & 0<t\le\Delta,\\
        \widehat{\phi}_\B(\Delta), & t>\Delta.
    \end{cases}
\]
Then \(\phi_\B\) is non-decreasing on \(\mathbb R_+\) and strictly positive
on \(\mathbb R_{++}\), proving \textup{(ii)}. Finally, fix \(x\in\B\). If \(h(x)=0\), then
\(\phi_\B(h(x))=0\le g(x)\). If \(h(x)>0\), then
\(x\in\Omega_{h(x)}\), and hence $\phi_\B(h(x))
    =
    \widehat{\phi}_\B(h(x))
    \le g(x)$.
This proves \textup{(i)}.
\end{proof}}

\revise{
Theorem~\ref{thm:GUEB-existence} is closely related in spirit to the consistent error-bound framework of~\cite{liu2024convergence}, which is developed for convex feasibility problems, with $h$ being the distance to the feasible set and $g$ being a feasibility residual.
Theorem~\ref{thm:GUEB-existence} generalizes the construction in~\cite{liu2024convergence} to a general pair $(h,g)$ satisfying the zero-set condition.


}


\revise{We next introduce a local version of the unified error-bound condition. More precisely, for every $\bar{x}\in F$, the error bound is assumed to hold on a neighborhood of $\bar{x}$ intersected with $S$, with an error-bound modulus that may depend on the base point~$\bar{x}$.} We denote by $F = g^{-1}(\{0\})\cap \B$ the target set. 


\begin{assumption}[Local error bound]\label{assmp:local-EB}
    For any $\bar{x}\in F:= g^{-1}(\{0\})\cap \B$, there exists a function $\phi_{\bar{x}}: [0,\eta_{\bar{x}}) \rightarrow \mathbb{R}_+$, where  $\eta_{\bar{x}} > 0$ is some constant, such that the following hold. 
    \begin{enumerate}[label=\textup{\textrm{(\roman*)}}]
        \item\label{ass_l_EB_1} There exist $\delta_{\bar{x}}>0$ such that
        \[
        \phi_{\bar{x}}(h(x)-h(\bar{x})) \leq g(x) \quad \forall x\in \{x\in\mathbb{R}^d :h(\bar{x})<h(x)<h(\bar{x})+\eta_{\bar{x}}\}\cap \mathbb{B}(\bar{x},\delta_{\bar{x}})\cap \B.
        \]
        \item\label{ass_l_EB_2} The function $\phi_{\bar{x}}$ is \revise{non-decreasing and strictly positive on $(0,\eta_{\bar{x}})$ with $\phi_{\bar x}(0) = 0$.}
    \end{enumerate}
\end{assumption}

\revise{
The existence of a local modulus $\phi_{\bar{x}}$ can be established by the same argument as in Theorem~\ref{thm:GUEB-existence}, provided that, in a neighborhood of $\bar{x}$, $h(x)=h(\bar{x})$ whenever $g(x)=0$.
}

\revise{For the smooth unconstrained minimization problem in Example~\ref{exam:SUM}, Assumption~\ref{assmp:local-EB} reduces, at each critical point $\bar{x}\in F$, to a local KL-type condition whose desingularizing function is determined by the local modulus $\phi_{\bar{x}}$. This condition applies to a broad class of objectives. Notably, functions definable in an o-minimal structure are known to satisfy the KL property~\cite{Kurdyka1998OnGO,Bolte2007Clarke,BolteDaniilidisLewis2007Lojasiewicz,Absil2005Convergence,Attouch2010Proximal,Bolte2014Proximal}. When the underlying o-minimal structure is polynomially bounded, the desingularizing function can moreover be chosen to be of power type, yielding a KL exponent~\cite{Kurdyka1998OnGO,Loi2016Lojasiewicz}. This setting includes, in particular, semialgebraic and globally subanalytic functions. 
In the common fixed point setting of Example~\ref{exam:CFP}, Assumption~\ref{assmp:local-EB} instead specializes to the metric $\varphi$-subregularity \cite{kruger2016nonlinear} of the residual mapping $I-T$ at $(\bar{x},0)$, where $\bar{x}\in \Fix T$, $\phi_{\bar x}(\cdot)=(\varphi^{-1}(\sqrt{\cdot}))^2$. When \(\phi_{\bar{x}}\) is a power-type modulus, this reduces to
H\"older metric subregularity of \(I-T\) at
$(\bar{x},0)$~\cite{LiMordukhovich2012Holder,MordukhovichOuyang2015HigherOrderMetricSubregularity,Kruger2015ErrorBA}, where $\bar{x}\in \Fix T$. When $\phi_{\bar x}$ is linear, it further reduces to the standard metric subregularity~\cite{ApetriiDureaStrugariu2013Subregularity,
DontchevRockafellar2004RegularityConditioning,
DontchevRockafellar2009ImplicitFunctions,
Gfrerer2011FirstSecondOrder,
Gfrerer2013DirectionalMetricRegularity,
Gfrerer2013DirectionalMetricSubregularity,
HenrionJourani2002CalmnessConvexConstraints,
HenrionJouraniOutrata2002CalmnessMultifunctions,
HenrionOutrata2001SubdifferentialCalmness,
IoffeOutrata2008MetricCalmnessQualification,
Leventhal2009MetricSubregularityPPA,
NgaiTinh2015MetricSubregularityMultifunctions,
Penot2010ErrorBoundsCalmness,
ZhengNg2007MetricSubregularityConvexGE,
ZhengNg2010MetricSubregularityCalmnessNonconvexGE,
ZhengNg2012MetricSubregularityProximalGE,
ZhengOuyang2011MetricSubregularityCompositeConvexGE,liang2016convergence,LukeTeboulleThao2020Necessary}, which, for convex feasibility problems, is precisely subtransversality of the collection at \(\bar{x}\in C\)~\cite{kruger2017subtransversality,kruger2018set}.}

In the following lemma, we show that the local error bound condition in Assumption~\ref{assmp:local-EB} can be uniformized in the sense of \cite[Lemma~6]{Bolte2014Proximal}.

\begin{lemma}[Uniformized error bound]\label{lem:uniformization}
    Suppose that Assumption~\ref{assmp:local-EB} holds. Define $V = h(F)$. Then, for any $v\in h(F)$ and compact set $C\subseteq h^{-1}(v)\cap F$, there exist $\delta_{C},\eta_{C} > 0$ and a finite set $\{\bar{x}_j\}_{j=1,\dots,l}\subseteq C$ such that
        \[
        \phi_{C}(h(x)-v) \leq g(x) \quad \forall x\in \{x\in\mathbb{R}^d :v<h(x)<v+\eta_{C},~\dist(x,C)<\delta_C\}\cap \B,
        \]
    where $\phi_C: [0,\eta_C) \to \mathbb{R}_+ $ is defined by $ \phi_C = \min_{j=1,\dots,l} \phi_{\bar{x}_j}$ with $\phi_{\bar{x}}$ defined in Assumption~\ref{assmp:local-EB}. 
\end{lemma}
\begin{proof}
    By compactness, $C$ can be covered by a finite number of open balls with centers $\bar{x}_1, \dots, \bar{x}_l \in  C$ and the corresponding radii $\delta_{\bar{x}_1},\dots, \delta_{\bar{x}_l}$ given in Assumption~\ref{assmp:local-EB}\ref{ass_l_EB_1}. Hence, there exists a constant $\delta_C>0$ such that the $\delta_C$-extension $\{x\in \mathbb{R}^d :\dist(x,C)<\delta_C\}$ can be covered by the same collection of open balls. The proof is then completed by taking $\eta_{C} = \min_{j=1,\dots,l} \eta_{\bar{x}_j}$, and $\phi_C = \min_{j=1,\dots,l} \phi_{\bar{x}_j}$.
\end{proof}

\section{\revise{Deterministic Benchmark Rates}}
\label{subsec:mda-benchmark}
\revise{
Our next task is to establish upper and lower bounds on the convergence rates of MDAs, which serve as a natural benchmark for the convergence rates of MRAs.


\subsection{Upper bounds under global UEB}

From now on, for simplicity, we denote the stepsize partial sum as
\[
    A_k:=\sum_{i=0}^{k-1}\alpha_i.
\]
The following result gives the deterministic benchmark used throughout the
paper.

\begin{theorem}[Convergence rates of MDA under GUEB]
\label{thm:mda-gueb-benchmark}
Suppose that Assumptions~\ref{assmp:basic_det},
\ref{assmp:step-size}, and~\ref{assmp:Global-EB} hold, and that $h(x^0)-h^*<\tau$.  Fix any $t_0\in(0,\tau)$ in Definition~\ref{def:inverse-smoothing}.  If $h(x^0)=h^*$, then the trajectory
is stationary.  Otherwise, the following statements hold.

\begin{enumerate}[label=\textup{\textrm{(\roman*)}}]
    \item If $1/\phi_{\B}$ is integrable near $0$, then there exists $K<\infty$
    such that
    \[
        x^k=x^K\in\argmin_{x\in\B}h(x),
        \qquad k\ge K.
    \]

    \item If $1/\phi_{\B}$ is not integrable near $0$, then, for every $k\ge0$,
    \begin{equation}\label{eq:mda-gueb-h-upper}
        h(x^k)-h^*
        \le
        \Psi_{\B}^{-1}\!\left(
            \Psi_{\B}(h(x^0)-h^*)+\cona\conb A_k
        \right).
    \end{equation}

    \item If $1/\phi_{\B}$ is not integrable near $0$ and
    $1/\sqrt{\phi_{\B}}$ is integrable near $0$, then $\{x^k\}$ converges to
    some $x^*\in\argmin_{x\in\B}h(x)$ and, for every $k\ge0$,
    \begin{equation}\label{eq:mda-gueb-tail-upper}
        \sum_{j=k}^{\infty}\|x^{j+1}-x^j\|
        \le
        \frac{1}{\cona\sqrt{\conb}}\,
        \Phi_{\B}(h(x^k)-h^*).
    \end{equation}
    Consequently,
    \begin{equation}\label{eq:mda-gueb-x-upper}
        \|x^k-x^*\|
        \le
        \frac{1}{\cona\sqrt{\conb}}\,
        \Phi_{\B}\!\left(
            \Psi_{\B}^{-1}\!\left(
                \Psi_{\B}(h(x^0)-h^*)+\cona\conb A_k
            \right)
        \right).
    \end{equation}
\end{enumerate}
\end{theorem}

\begin{proof}
Set $r_k:=h(x^k)-h^*$.  If $r_k=0$ for some $k$, then the
sufficient-decrease condition and the definition of $h^*$ give
\[
    h^*
    \le h(x^{k+1})
    \le h^*-\frac{\cona}{\alpha_k}\|x^{k+1}-x^k\|^2.
\]
Hence $x^{k+1}=x^k$, and the trajectory remains constant thereafter. Thus, the following argument is understood up to the first such index, and all the asserted estimates hold trivially afterward. Suppose now that $r_k>0$. By \eqref{eq:basic_1}, \eqref{eq:det-bound}, and the global UEB condition,
\begin{equation}\label{eq:mda-gueb-recursion}
    r_{k+1}
    \le
    r_k-\frac{\cona}{\alpha_k}\|x^{k+1}-x^k\|^2 
    \le
    r_k-\cona\conb\,\alpha_k g(x^k) 
    \le
    r_k-\cona\conb\,\alpha_k\phi_{\B}(r_k).
\end{equation}
Thus $\{r_k\}$ is non-increasing. If $r_k\downarrow r_\infty>0$, then summing
\eqref{eq:mda-gueb-recursion} yields
\[
    r_0
    \ge
    \cona\conb\,\phi_{\B}(r_\infty)
    \sum_{k=0}^{\infty}\alpha_k
    =+\infty,
\]
a contradiction.  Hence $r_k\downarrow0$ unless finite termination has already occurred.
Since $\phi_{\B}$ is non-decreasing, applying Lemma~\ref{lem:rate-preliminary}\ref{rate-preliminary1} to \eqref{eq:mda-gueb-recursion} implies $\Psi_{\B}(r_{k+1})-\Psi_{\B}(r_k)\ge  \cona\conb\,\alpha_k$.
Consequently,
\begin{equation}\label{eq:mda-gueb-psi-growth}
    \Psi_{\B}(r_k)
    \ge
    \Psi_{\B}(r_0)+\cona\conb A_k.
\end{equation}
If $1/\phi_{\B}$ is integrable near $0$, then $\Psi_{\B}(0)<+\infty$, and
\eqref{eq:mda-gueb-psi-growth} cannot hold with $r_k>0$ once
\[
    A_k
    \ge
    \frac{\Psi_{\B}(0)-\Psi_{\B}(r_0)}{\cona\conb}.
\]
This proves finite termination.  If $1/\phi_{\B}$ is not integrable near $0$,
then $\Psi_{\B}(t)\to+\infty$ as $t\downarrow0$.  Since $\Psi_{\B}$ is
decreasing, inversion of \eqref{eq:mda-gueb-psi-growth} gives
\eqref{eq:mda-gueb-h-upper}. It remains to estimate the trajectory length.  The deterministic lower bound
on the step and the global UEB condition imply
\[
    \|x^{k+1}-x^k\|
    \ge
    \sqrt{\conb}\,\alpha_k\sqrt{g(x^k)}
    \ge
    \sqrt{\conb}\,\alpha_k\sqrt{\phi_{\B}(r_k)}.
\]
Combining this inequality with sufficient decrease gives
\[
    r_k-r_{k+1}
    \ge
    \cona\sqrt{\conb}\,
    \|x^{k+1}-x^k\|\sqrt{\phi_{\B}(r_k)}.
\]
Therefore, by Lemma~\ref{lem:rate-preliminary}~\ref{rate-preliminary2}, we have
\begin{align*}
    \Phi_{\B}(r_k)-\Phi_{\B}(r_{k+1})
    \ge
    \cona\sqrt{\conb}\,\|x^{k+1}-x^k\|.
\end{align*}
Summing from $j=k$ to $N$ and letting $N\to\infty$ proves
\eqref{eq:mda-gueb-tail-upper}.  In particular, $\{x^k\}$ is Cauchy.  Since
$\B$ is closed, it converges to some $x^*\in\B$; continuity of $h$ and
$r_k\to0$ then give $h(x^*)=h^*$.  Finally,
\[
    \|x^k-x^*\|
    \le
    \sum_{j=k}^{\infty}\|x^{j+1}-x^j\|,
\]
and combining \eqref{eq:mda-gueb-tail-upper} with
\eqref{eq:mda-gueb-h-upper} proves \eqref{eq:mda-gueb-x-upper}.
\end{proof}

\subsection{A matching MDA instance}
\label{subsec:mda-matching-instance}

We next show that the rates in Theorem~\ref{thm:mda-gueb-benchmark} cannot be improved in general. More precisely, fixed any admissible constants $c_1,c_2,c_3 >0$ and error-bound modulus $\phi_\B$, we construct a
one-dimensional MDA instance for which the lower bounds on objective-value and iterate
match the corresponding upper bounds.

\begin{theorem}[Lower bounds for MDA]
\label{thm:mda-matching-instance}
Fix arbitrary constants $\cona >0$ and $\conc \geq \conb> 0$.
Let $\phi_{\B}\in C^1([0,\eta])$, for some $\eta > 0 $, be a non-decreasing function satisfying that
\[
    \phi_{\B}(0)=0
    \qquad\text{and}\qquad
    \phi_{\B}(t)>0
    \quad \forall t\in (0, \eta].
\]
Depending on whether $1/\sqrt{\phi_{\B}}$ is integrable near $0$, two one-dimensional MDA instances with the following properties exist. Both instances satisfy Assumption~\ref{assmp:basic_det} with the prescribed constants
$(\cona,\conb,\conc)$, Assumption~\ref{assmp:Global-EB} with modulus
$\phi_{\B}$ and $\tau=\eta$, and
Assumption~\ref{assmp:step-size}, and that for every fixed $\gamma>1$, there exists
$K_\gamma<+\infty$ such that
\begin{equation}\label{eq:mda-match-objective-lower}
    h(x^k)-h^*
    \geq
    \Psi_{\B}^{-1}\left(
        \Psi_{\B}\bigl(h(x^0)-h^*\bigr)
        +\gamma\cona\conb A_k
    \right)
    \qquad \forall k\geq K_\gamma.
\end{equation}

\begin{enumerate}[label=\textup{(\roman*)}]
    \item If $1/\sqrt{\phi_{\B}}$ is integrable near $0$, the instance satisfies further that
    \[
        h^*=0,
        \qquad
        \argmin_{x\in\B}h(x)=\{0\},
        \qquad
        x^k\to x^* = 0.
    \]
    Furthermore, for
    every fixed $\sigma\in(0,1)$, there exists $K_\sigma<+\infty$ such
    that
    \begin{equation}\label{eq:mda-match-iterate-lower}
        \|x^k-x^*\|
        \geq
        \frac{\sigma}{\cona\sqrt{\conb}}
        \Phi_{\B}\bigl(h(x^k)-h^*\bigr)
        \qquad \forall k\geq K_\sigma.
    \end{equation}
    Consequently, for all
    $k\geq\max\{K_\gamma,K_\sigma\}$,
    \begin{equation}\label{eq:mda-match-iterate-lower2}
        \|x^k-x^*\|
        \geq
        \frac{\sigma}{\cona\sqrt{\conb}}
        \Phi_{\B}\left(
            \Psi_{\B}^{-1}\left(
                \Psi_{\B}\bigl(h(x^0)-h^*\bigr)
                +\gamma\cona\conb A_k
            \right)
        \right).
    \end{equation}

    \item If $1/\sqrt{\phi_{\B}}$ is not integrable near $0$, the instance satisfies further that
    \[
        h^*=0,
        \qquad
        \argmin_{x\in\B}h(x)=\emptyset,
        \qquad
        x^k\to+\infty.
    \]
\end{enumerate}
\end{theorem}

\begin{proof}
Let $M:=\max_{0\leq t\leq\eta}\phi_{\B}'(t)$. Since $\phi_{\B}(0)=0$ and $\phi_{\B}$ is non-decreasing and positive on $(0,\eta]$, we have $M>0$ and $\phi_{\B}(t)\leq Mt$ on $[0,\eta]$. Hence $1/\phi_{\B}$ is not integrable near $0$. Choose $a>0$ such that $\cona\conb Ma<1$ and set $\alpha_k:=a/(k+1)$. Then Assumption~\ref{assmp:step-size} holds, $\alpha_k\to0$, and $A_k\to+\infty$. Fix $r_0\in(0,\eta)$ and define
\[
r_{k+1}:=r_k-\cona\conb\alpha_k\phi_{\B}(r_k),\qquad d_k:=\sqrt{\conb}\alpha_k\sqrt{\phi_{\B}(r_k)}.
\]
Since $r_k\geq r_{k+1}\geq(1-\cona\conb M\alpha_k)r_k>0$, the sequence $\{r_k\}$ is well defined and decreasing. If $r_k\downarrow r_\infty>0$, then $r_0-r_{N+1}\geq\cona\conb\phi_{\B}(r_\infty)\sum_{k=0}^{N}\alpha_k\to+\infty$, a contradiction. Hence $r_k\downarrow0$. For $s\in[r_{k+1},r_k]$, monotonicity and the definition of $M$ give $\phi_{\B}(s)\leq\phi_{\B}(r_k)$ and $\phi_{\B}(r_k)-\phi_{\B}(s)\leq M(r_k-r_{k+1})=\cona\conb M\alpha_k\phi_{\B}(r_k)$. Thus
\begin{equation}\label{eq:mda-match-phi-comparison}
(1-\cona\conb M\alpha_k)\phi_{\B}(r_k)\leq\phi_{\B}(s)\leq\phi_{\B}(r_k).
\end{equation}
Since $r_k-r_{k+1}=\cona\conb\alpha_k\phi_{\B}(r_k)$, \eqref{eq:mda-match-phi-comparison} yields
\begin{equation}\label{eq:mda-match-psi-increment}
\cona\conb\alpha_k\leq\Psi_{\B}(r_{k+1})-\Psi_{\B}(r_k)\leq\frac{\cona\conb\alpha_k}{1-\cona\conb M\alpha_k},
\end{equation}
and
\begin{equation}\label{eq:mda-match-phi-length}
\cona\sqrt{\conb}\,d_k\leq\int_{r_{k+1}}^{r_k}\frac{ds}{\sqrt{\phi_{\B}(s)}}\leq\frac{\cona\sqrt{\conb}}{\sqrt{1-\cona\conb M\alpha_k}}\,d_k.
\end{equation}

Suppose first that $1/\sqrt{\phi_{\B}}$ is integrable near $0$. Summing the first inequality in \eqref{eq:mda-match-phi-length} gives $\sum_{k=0}^{\infty}d_k\leq\Phi_{\B}(r_0)/(\cona\sqrt{\conb})<+\infty$. Define $x^k:=\sum_{j=k}^{\infty}d_j$ and the instance by
\[
\B:=\{0\}\cup\{x^k:k\geq0\},\qquad h(0)=g(0):=0,\qquad h(x^k):=r_k,\qquad g(x^k):=\phi_{\B}(r_k).
\]
Then $x^k\to0$, the functions $h$ and $g$ are continuous on the closed set $\B$, and $h^*=0$ with $\argmin_{x\in\B}h(x)=\{0\}$. 

Suppose next that $1/\sqrt{\phi_{\B}}$ is not integrable near $0$. Define $x^0:=0$, $x^k:=\sum_{j=0}^{k-1}d_j$ for $k\geq1$, and the instance by
\[
\B:=\{x^k:k\geq0\},\qquad h(x^k):=r_k,\qquad g(x^k):=\phi_{\B}(r_k).
\]
Summing the second inequality in \eqref{eq:mda-match-phi-length} gives
\[
\sum_{j=0}^{N}d_j\geq\frac{\sqrt{1-\cona\conb Ma}}{\cona\sqrt{\conb}}\int_{r_{N+1}}^{r_0}\frac{ds}{\sqrt{\phi_{\B}(s)}}\longrightarrow+\infty.
\]
Hence $x^k\to+\infty$, so $\B$ is closed and discrete, while $h^*=0$ and $\argmin_{x\in\B}h(x)=\emptyset$. In either construction, $\|x^{k+1}-x^k\|=d_k$ and $g(x^k)=\phi_{\B}(r_k)=\phi_{\B}(h(x^k)-h^*)$, so GUEB holds with equality. Moreover, $h(x^k)-h(x^{k+1})=r_k-r_{k+1}=\cona d_k^2/\alpha_k$ and $d_k^2=\conb\alpha_k^2g(x^k)\leq\conc\alpha_k^2g(x^k)$. Thus Assumption~\ref{assmp:basic_det} holds with the prescribed constants $(\cona,\conb,\conc)$. Summing \eqref{eq:mda-match-psi-increment} gives $\Psi_{\B}(r_k)\leq\Psi_{\B}(r_0)+\cona\conb\sum_{j=0}^{k-1}\alpha_j/(1-\cona\conb M\alpha_j)$. Since $\sum_j\alpha_j^2<+\infty$, the sum equals $A_k+O(1)$. Therefore, for every $\gamma>1$, there exists $K_\gamma<+\infty$ such that $\Psi_{\B}(r_k)\leq\Psi_{\B}(r_0)+\gamma\cona\conb A_k$ for all $k\geq K_\gamma$. Since $\Psi_{\B}$ is decreasing, $r_k\geq\Psi_{\B}^{-1}(\Psi_{\B}(r_0)+\gamma\cona\conb A_k)$, proving \eqref{eq:mda-match-objective-lower}. 

Finally, in the integrable case, the second inequality in \eqref{eq:mda-match-phi-length} gives $d_k\geq\sqrt{1-\cona\conb M\alpha_k}\,[\Phi_{\B}(r_k)-\Phi_{\B}(r_{k+1})]/(\cona\sqrt{\conb})$. Given $\sigma\in(0,1)$, choose $K_\sigma$ such that $\sqrt{1-\cona\conb M\alpha_k}\geq\sigma$ for all $k\geq K_\sigma$. Then $\|x^k-x^*\|=\sum_{j=k}^{\infty}d_j\geq\sigma\Phi_{\B}(r_k)/(\cona\sqrt{\conb})$, proving \eqref{eq:mda-match-iterate-lower}. Combining this with \eqref{eq:mda-match-objective-lower} and the monotonicity of $\Phi_{\B}$ proves \eqref{eq:mda-match-iterate-lower2}.
\end{proof}

Theorem~\ref{thm:mda-matching-instance} shows that Theorem~\ref{thm:mda-gueb-benchmark} is sharp both qualitatively and quantitatively. Indeed, the constructed MDA instance attains the same convergence profiles as in Theorem~\ref{thm:mda-gueb-benchmark}: the objective gap and the iterate error evolve
according to $ \Psi_{\B}^{-1}(\Theta(A_k))$ and  $\Theta(\Phi_{\B}(\Psi_{\B}^{-1}(\Theta(A_k))))$, respectively.
Therefore, these rate profiles are not artifacts of the analysis; they are
achievable within the abstract MDA--UEB regime. 

Moreover, the constants appearing in the upper bounds are also sharp. More
precisely, the arbitrary constants $\gamma>1$ and $\sigma\in(0,1)$ in
Theorem~\ref{thm:mda-matching-instance} are purely numerical constants,
independent of the constants $c_1,c_2,c_3$, the error-bound modulus $\phi_\B$, and the stepsizes $\alpha_k$.
Consequently, by letting $\gamma\downarrow1$ and $\sigma\uparrow1$, the
constructed instances approach the coefficients in the upper
bounds. Hence the coefficient $\cona\conb$ in the objective-value bound and
the coefficient $1/(\cona\sqrt{\conb})$ in the iterate estimate cannot be
uniformly improved over the MDA class under the stated assumptions.




\subsection{Extension to local UEB conditions}
\label{subsec:mda-local-matching}

The local UEB extension follows by applying the same argument to the tail of a convergent trajectory. Suppose that $x^k\to x^*$. Since $h(x^k)\downarrow h(x^*)$,
there exists $K$ such that for all $k\ge K$, the tail trajectory lies in the
neighborhood where the local UEB condition at $x^*$ holds. Therefore,
Theorem~\ref{thm:mda-gueb-benchmark} applies to the tail with the substitutions
\[
    h^*\leftarrow h(x^*),
    \qquad
    \phi_{\B}\leftarrow\phi_{x^*},
    \qquad
    \Psi_{\B}\leftarrow\Psi_{x^*},
    \qquad
    \Phi_{\B}\leftarrow\Phi_{x^*},
    \qquad
    x^0\leftarrow x^K,
\]
and
\[
    A_k\leftarrow A_{K,k}:=\sum_{i=K}^{k-1}\alpha_i .
\]
The matching construction in Theorem~\ref{thm:mda-matching-instance} also satisfies
the local UEB assumption in the finite-length regime. Indeed, in this case the
constructed instance has
\[
    \B=\{0\}\cup\{x^k:k\ge0\},
    \qquad
    F=\{x\in\B:g(x)=0\}=\{0\},
\]
and
\[
    g(x)
    =
    \phi_{\B}(h(x)-h^*)
    =
    \phi_{\B}(h(x)-h(0)) 
    \qquad
    \forall x\in\B .
\]
Hence the local UEB condition holds at the only point $x^*=0\in F$
for the same modulus $\phi_{\B}$ with equality. Therefore, the
trajectory-wise lower estimates established in Theorem~\ref{thm:mda-matching-instance}
also provide matching lower estimates for the local UEB theory.

}

\section{\revise{Existing Methods and Their Limitations}}
\label{sec:exist-methods}
\revise{
The goal of this work is to establish sharp convergence rates for MRAs under unified error-bound conditions. Specifically,
under GUEB and LUEB conditions, we derive non-asymptotic in-expectation and/or asymptotic almost-sure rates for the corresponding $h$-gaps and iterate errors.
Before doing so, we explain in this section why existing techniques for deterministic or stochastic algorithms are inadequate. We focus exclusively on results derived from the abstract assumptions, Assumption~\ref{assmp:basic} or Assumption~\ref{assmp:basic_det}, together with GUEB and LUEB conditions
in Assumption~\ref{assmp:Global-EB} and Assumption~\ref{assmp:local-EB}, without
exploiting any additional structure of specific problems or algorithms.}

\subsection{\revise{Existing analysis for deterministic algorithms}}
\label{sec:exist-det}
\revise{
We first review standard techniques for analyzing monotone deterministic algorithms (MDA).
As specified in Assumption~\ref{assmp:basic_det}, an MDA sequence satisfies the relations:
\begin{empheq}[left=\text{(Strong Descent Condition)}\quad\empheqlbrace]{align}
h(x^k)-h(x^{k+1})
&\geq \frac{\cona}{\alpha_k}\,\|x^k-x^{k+1}\|^2,
\tag{SDC1}\label{eq:SDC1} \\
\|x^k-x^{k+1}\|^2
&\geq \conb\alpha_k^2 g(x^k).
\tag{SDC2}\label{eq:SDC2}
\end{empheq}
Combining these two inequalities gives the fundamental descent estimate
\[
    h(x^k)-h(x^{k+1})
    \geq
    \cona\conb\alpha_k g(x^k).
\]
Let
\[
    h^\infty:=\lim_{k\to\infty}h(x^k).
\]
Under the GUEB, the residual measure $g(x^k)$ can be
controlled by the gap $h(x^k)-h^\infty$, which yields a
recursion for the objective sequence, solving which leads to the
standard convergence rate results for the progress metric $h$. Under the LUEB, the same argument applies after the sequence enters the error-bound neighborhood, yielding a recursion of the same form. This approach is classical in deterministic algorithmic analyses based
on error bounds, PL conditions, and KL inequalities; see, for example,
\cite{luo1993error,karimi2016linear,NecoaraNesterovGlineur2019Linear,
DrusvyatskiyLewis2018ErrorBounds,Bolte2017Errora} for optimization problems,
\cite{BauschkeBorwein1993VonNeumann,
bauschke1996projection,
BauschkeBorweinLewis1997CyclicProjections,
BauschkeBorweinLi1999StrongCHIP,liu2024convergence} for convex feasibility problems, and
\cite{Cegielski2012IterativeMethods,
BauschkeNollPhan2015AveragedOperators,
borwein2017convergence,
LewisLukeMalick2009LocalLinear,
LukeThaoTam2018Quantitative,
LukeTeboulleThao2020Necessary,Liu2024ConcreteCR} for common fixed-point problems.

The role of inequality~\eqref{eq:SDC2} is even more important in
the analysis of iterate convergence. Combining \eqref{eq:SDC1} with the
square-root form of \eqref{eq:SDC2} yields
\[
    h(x^k)-h(x^{k+1})
    \geq
    \cona\sqrt{\conb}\,
    \sqrt{g(x^k)}\,
    \|x^{k+1}-x^k\|.
\]
Together with the error-bound condition relating the progress metrics $g(x^k)$ and $h(x^k) - h^\infty$, this
inequality controls the step length through the decrease of the
sequence $h(x^k)$ and the geometry encoded in the error-bound modulus $\phi_\B$. This is the key
mechanism behind finite-length arguments and iterate-rate estimates in
deterministic KL-based analyses
\cite{Absil2005Convergence,Attouch2010Proximal,Attouch2013Convergence,
FrankelGarrigosPeypouquet2015Splitting,Bolte2014Proximal}. A telescoping argument
then converts the value-rate estimate into a rate for the iterates.

These deterministic arguments, however, rely critically on inequality~\eqref{eq:SDC2}, which is missing from MRAs. Instead, MRAs have only the conditional expectation estimate
\[
    \mathbb{E}
    [\|x^{k+1}-x^k\|^2\mid\mathcal F_k]
    \geq
    \conb\alpha_k^2 g(x^k),
\]
which does not provide direct control of individual realizations of the step
lengths. Consequently, the classical deterministic arguments based on strong
descent cannot be directly extended to randomized algorithms.}

\subsection{\revise{Existing analysis for related stochastic algorithms}}

\revise{

}

\subsubsection{\revise{Existing in-expectation rates of $h(x^k)$ under global error bounds}}
\revise{
A standard way to exploit an error-bound condition in stochastic analysis is to
condition on the filtration $\mathcal{F}_k$. In our setting, combining the sufficient
decrease condition, the conditional expected-progress bound, and GUEB yields
\begin{equation}\label{eq:condexp-descent}
    \Exp\!\left[h(x^{k+1})-h^*\,\middle|\,\mathcal{F}_k\right]
    \le
    h(x^k)-h^*
    -
    \cona\conb\alpha_k
    \phi_{\B}\!\left(h(x^k)-h^*\right).
\end{equation}
Taking expectations gives
\[
    \Exp[h(x^{k+1})-h^*]
    \le
    \Exp[h(x^k)-h^*]
    -
    \cona\conb\alpha_k
    \Exp\!\left[
        \phi_{\B}\!\left(h(x^k)-h^*\right)
    \right].
\]
If $\phi_{\B}$ is convex, Jensen's inequality then leads to a deterministic recursion
\begin{equation}\label{eq:exp-descent}
    \Exp[h(x^{k+1})-h^*]
    \le
    \Exp[h(x^k)-h^*]
    -
    \cona\conb\alpha_k
    \phi_{\B}\!\left(\Exp[h(x^k)-h^*]\right),
\end{equation}
from which an in-expectation convergence rate for the objective gap follows directly.

Conditional-expectation descent recursions of this type are standard in
stochastic quasi-Fejér and almost-supermartingale analyses
\cite{RobbinsSiegmund1971,CombettesPesquet2015}. Closely related arguments are also
widely used in the analysis of randomized projection and randomized Kaczmarz
methods \cite{Nedic,wang2016stochastic,Necoara2019,Strohmer2009Randomized,
LeventhalLewis2010,GowerRichtarik2015}, randomized fixed point iteration \cite{hermer2019random}, randomized coordinate-descent methods
\cite{Nesterov2012Efficiency,beck2013convergence,Richtarik2014Iteration,
Lu2015Complexity}, and stochastic gradient methods under strong-growth noise
conditions \cite{Schmidt2013Fast}, typically in conjunction with classical
error-bound or Polyak--Łojasiewicz-type conditions. When expressed in the GUEB
framework, the moduli used in these rate analyses are typically linear or convex
power-type functions, for which Jensen's inequality is applicable. For completeness, we summarize how these techniques apply within our MRA framework under the GUEB condition with convex error-bound moduli.}

\begin{theorem}\label{thm:globalEB-exph-cvx}
Suppose that Assumptions~\ref{assmp:basic}, \ref{assmp:step-size}, and~\ref{assmp:Global-EB} hold and that $\phi_{\B}$ is convex. \revise{Then $1/\phi_{\B}$ is not integrable near $0$.} \revise{If $h(x^0)=h^*$, then the trajectory is stationary almost surely. Otherwise, for any $k\geq1$,}
\begin{equation}\label{eq:globalEB-exph-cvx-0}
\Exp[h(x^k)-h^*]\leq\Psi_{\B}^{-1}\left(\Psi_{\B}(h(x^0)-h^*)+\cona\conb\revise{A_k}\right).
\end{equation}
\end{theorem}

\begin{proof}
\revise{Fix $\delta\in(0,\tau)$. Since $\phi_{\B}$ is convex and $\phi_{\B}(0)=0$, $\phi_{\B}(t)\leq t\phi_{\B}(\delta)/\delta$ for all $t\in(0,\delta]$, so $1/\phi_{\B}$ is not integrable near $0$. If $h(x^0)=h^*$, Assumption~\ref{assmp:basic}\ref{basic_1} implies that the trajectory is stationary almost surely.} \revise{Suppose that $h(x^0)>h^*$ and set $r_k:=\Exp[h(x^k)-h^*]$.} For any $k\geq1$ such that $r_k>0$, \eqref{eq:exp-descent} implies that $r_i>0$ for all $i\leq k$. Applying Lemma~\ref{lem:rate-preliminary}\ref{rate-preliminary1} to \eqref{eq:exp-descent} and summing from $i=0$ to $k-1$ yields
\begin{equation}\label{eq:globalEB-exph-cvx-1}
\Psi_{\B}(r_k)\geq\Psi_{\B}(r_0)+\cona\conb\revise{A_k}.
\end{equation}
Since $\Psi_{\B}$ is strictly decreasing and $r_0=h(x^0)-h^*$, inversion gives \eqref{eq:globalEB-exph-cvx-0}. If $r_k=0$, \eqref{eq:globalEB-exph-cvx-0} holds trivially. This completes the proof.
\end{proof}

\subsubsection{\revise{Non-convex moduli and limitations of convexification}} \label{subsubsection:non-convex_moduli}

\revise{
When the modulus $\phi_{\B}$ is non-convex, the pivotal implication from \eqref{eq:exp-descent} to \eqref{eq:condexp-descent} is not valid in general.
The following  optimization example shows that a non-convex error-bound modulus can arise naturally.

\begin{example}[GUEB with a nonconvex modulus $\phi_\B$]\label{ex:nonconvex-phi}
Let $\phi_{\B}:[0,\eta]\to\mathbb{R}_{+}$ be any continuous non-decreasing function
satisfying $\phi_{\B}(0)=0 $ and $\phi_{\B}(t)>0$ for $t>0$,
where $1/\sqrt{\phi_{\B}(s)}$ is integrable near $0$. Define
\[
    Q(t):=\int_t^{\eta}\frac{ds}{\sqrt{\phi_{\B}(s)}},
\]
and let $f:=Q^{-1}$ on the range of $Q$. Then $f'(x)=-\sqrt{\phi_{\B}(f(x))}$ whenever $f$ is differentiable. Consider the one-dimensional
optimization problem
\[
    \min_{x\in S} h(x),
\]
where $h(x):=f(x)$, and define the stationarity measure $g(x):=|h'(x)|^2 $.
By construction,
\[
    g(x)
    =
    |f'(x)|^2
    =
    \phi_{\B}(f(x))
    =
    \phi_{\B}(h(x)-h^*),
\]
where $h^*=0$ is the infimum value of $h$. The last display thus shows that the GUEB holds with equality.
Therefore, any admissible error-bound modulus $\phi_{\B}$ can be realized by a concrete
one-dimensional optimization problem, regardless of its convexity.
\end{example}

For a problem with a nonconvex error-bound modulus, a tempting idea is to replace it by a convex minorant and then apply the preceding Jensen's inequality-based techniques. This approach, however, has several limitations. First, constructing a useful convex
minorant is generally problem-dependent and must be carried out on a case-by-case basis. Second, without a uniform quantitative comparison between the original modulus and its convex minorant, the resulting non-asymptotic rate may be substantially weaker. Third, a nontrivial convex minorant may not exist on an unbounded domain. For example, suppose that $\B=\R$ and that $\phi_{\B}$ is asymptotically flat in the sense that $\phi_{\B}(t)=o(t)$ as $t\to+\infty$.
If there exists a convex non-decreasing function $\widehat{\phi}$ satisfying $\widehat{\phi}(0)=0$, $0\le \widehat{\phi}(t)\le\phi_{\B}(t)$ and $\widehat{\phi}(t_0)>0$ for some $t_0>0$, then convexity would imply
\[
    \phi_\B (t) \ge \widehat{\phi}(t)
    \ge
    \frac{\widehat{\phi}(t_0)}{t_0}\,t
    \qquad \forall t\ge t_0,
\]
contradicting $\phi_{\B}(t)=o(t)$. Hence the only convex minorant is the trivial function $\widehat{\phi}\equiv0$. Therefore, convexification is not guaranteed to yield a rate-preserving convex modulus. 
}

\subsubsection{\revise{Existing in-expectation rates of $x^k$ under global error bounds}}
\revise{
Once an in-expectation rate of the gap $h(x^k)-h^*$ is available, we can transform it to that of iterates combining the usual deterministic KL analysis tools with concavity of the desingularization function $\Phi_\B$ and Jensen's inequality, see also \cite[Lemma 4.2]{Fest2026Stochastic} for this type of argument.}

\begin{proposition}\label{prop:iter-global-exp}
    Suppose that Assumptions~\ref{assmp:basic} and~\ref{assmp:Global-EB} hold and that $1/\sqrt{\phi_{\B}}$ is integrable near $0$.
    \revise{Then there exists a random variable $x^*$ such that $x^k$ converges to $x^*$ both almost surely and in $L^1$, and}
    \begin{equation}\label{eq:ite-rate-exp}
        \Exp[\norm{x^k-x^*}] \leq \frac{\sqrt{\conc}}{\cona \conb } \Phi_{\B}(\Exp[h(x^k)-h^*])\quad \forall k\ge 0 .
    \end{equation}
\end{proposition}
\begin{proof}
    \revise{First, Assumptions~\ref{assmp:basic} and~\ref{assmp:Global-EB}\ref{ass_g_E_3}, together with inequality~\eqref{eq:basic_1}, imply that
    \[
        0\leq h(x^k)-h^*
        \leq h(x^0)-h^*<\tau
        \quad\forall k\geq 0.
    \]
    In particular, all the random variables appearing in the conditional expectations below are integrable.} \revise{Taking the conditional expectation in inequality~\eqref{eq:basic_1} and using inequality~\eqref{eq:expbound}, we obtain}
    \[
    \Exp[h(x^{k+1})-h^*\mid \cF_k] \leq h(x^k)-h^* - \cona \conb  \alpha_k g(x^k).
    \]
    By Assumption~\ref{assmp:Global-EB} and inequality~\eqref{eq:basic_4}, we see that 
    \[
    \sqrt{g(x^k)} \geq \frac{1}{\sqrt{\conc}\alpha_k} \norm{x^{k+1}-x^k}\quad \text{and} \quad\sqrt{g(x^k)}\geq \sqrt{\phi_{\B}(h(x^k)-h^*)}.
    \]
    Combining these inequalities, we have
    \[
    \Exp[h(x^{k+1})-h^*\mid \cF_k] \leq h(x^k)-h^* -   \frac{\cona \conb }{\sqrt{\conc}} \norm{x^{k+1}-x^k} \sqrt{\phi_{\B}(h(x^k)-h^*) }.
    \] 
    \revise{Let $A_k:=\{h(x^k)-h^*>0\}\in\cF_k$. Fix $\omega\in A_k$ outside a null set on which the preceding conditional descent inequality may fail. Applying Lemma~\ref{lem:rate-preliminary}\ref{rate-preliminary2} pointwise to the scalar quantities
    \[
        a=h(x^k(\omega))-h^*,\qquad
        \widetilde{a}
        =
        \Exp[h(x^{k+1})-h^*\mid\cF_k](\omega),
    \]
    and
    \[
        \Delta
        =
        \frac{\cona\conb}{\sqrt{\conc}}
        \norm{x^{k+1}(\omega)-x^k(\omega)}
    \]
    yields the inequality below at $\omega$. On the complement $A_k^c$, inequality~\eqref{eq:basic_1} and the nonnegativity of $h(x^{k+1}(\omega))-h^*$ imply that
    \[
        \norm{x^{k+1}(\omega)-x^k(\omega)}=0
        \quad\text{and}\quad
        h(x^{k+1}(\omega))-h^*=0.
    \]
    Since $A_k^c\in\cF_k$, the property of conditional expectation gives
    \[
        \mathbf{1}_{A_k^c}
        \Exp[h(x^{k+1})-h^*\mid\cF_k] 
        =
        \Exp\!\left[
            \mathbf{1}_{A_k^c}(h(x^{k+1})-h^*)
            \,\middle|\,\cF_k
        \right]
        =0.
    \]
    Therefore, the following inequality holds almost surely:}
    \[
    \norm{x^{k+1}-x^k} \leq \frac{\sqrt{\conc}}{\cona \conb } (\Phi_{\B}(h(x^k)-h^*) - \Phi_{\B}(\Exp[h(x^{k+1})-h^*\mid \cF_k])).
    \]
    Using the concavity of $ \Phi_{\B} $, we obtain $$\Phi_{\B}(\Exp[h(x^{k+1})-h^*\mid \cF_k]) \geq \Exp[\Phi_{\B}(h(x^{k+1})-h^*)\mid \cF_k]\quad\text{a.s.}.$$
    Thus, 
    \begin{equation}\label{ineq:prop:iter-global-exp-proof-1}
        \Exp[\norm{x^{k+1}-x^k}] \leq \frac{\sqrt{\conc}}{\cona \conb } (\Exp[\Phi_{\B}(h(x^k)-h^*)] - \Exp[\Phi_{\B}(h(x^{k+1})-h^*)]),
    \end{equation}
    which, together with Assumption~\ref{assmp:Global-EB}\ref{ass_g_E_3} and  the continuity of $\Phi_{\B}$, yields 
    $\sum_{k=0}^\infty \Exp[\norm{x^{k+1}-x^k}] < +\infty$. Therefore, there exists a random variable $ x^* $ such that $ x^k \to x^* $ almost surely and in $ L^1 $. Furthermore, 
    \begin{equation*}\label{eq:ite-rate-sure}
        \Exp[\norm{x^k-x^*}] \leq  \frac{\sqrt{\conc}}{\cona \conb }\Exp[\Phi_{\B}(h(x^k)-h^*)]\leq \frac{\sqrt{\conc}}{\cona \conb } \Phi_{\B}(\Exp[h(x^k)-h^*]),
    \end{equation*}
    where the first inequality follows from a telescoping trick on~\eqref{ineq:prop:iter-global-exp-proof-1} the second inequality follows directly from the concavity of $\Phi_{\B}$.
\end{proof}


\subsubsection{\revise{Almost-Sure convergence rates of $h(x^k)$ and $x^k$}}
\revise{
Almost-sure rates are also important, as they describe the convergence behavior along almost every realized sample path. To the best of our knowledge, no existing result directly addresses such rates in our setting. The closest result is \cite[Lemma~3.1]{weissmann2025almost}, which studies a recursion similar to \eqref{eq:condexp-descent} but with an additional additive error term. Its analysis is restricted to square-summable and polynomially decaying stepsizes. The resulting bounds are not sharp when the additive error is absent, which is the regime considered here. Nevertheless, tools developed in related contexts can be adapted to our setting. The following elementary lemma, which extends the block-maximal argument of \cite[Proposition~I.1(ii)]{omiya2026randomized}, identifies the almost-sure rate guarantees that can be deduced by combining a supermartingale property with an in-expectation rate. In the sequel, we use $Y_k$ to denote $h(x^k)-h^*$. However, the conclusions apply more generally to any nonnegative random sequence satisfying the assumed properties.}

\revise{\begin{lemma}[Expectation rate to almost-sure little-\(o\) rate]
\label{lem:exp-to-as-little-o}
    Let \(\{Y_k\}_{k\geq0}\) be a nonnegative supermartingale with respect to the filtration
    \(\{\cF_k\}_{k\geq0}\). Let \(\{r_k\}_{k\geq0}\) be a deterministic
    non-increasing sequence such that \(r_k>0\) and \(r_k\downarrow0\). Suppose
    that $\Exp[Y_k]\leq r_k$ for all $k\ge 0$.
    Let \(L:[1,\infty)\to(0,\infty)\) be non-decreasing and satisfy
    \[
        \sum_{n=1}^{\infty}\frac{1}{L(n)}<+\infty .
    \]
    Then, 
    \[
    Y_k=o\!\left(
        r_kL\!\left(\log\frac{1}{r_k}\right)
    \right)
    \qquad\text{almost surely}.
\]
\end{lemma}

\begin{proof}
    Since \(r_k\downarrow0\), it suffices to consider all sufficiently large
    \(k\), and we may assume \(r_k\leq \exp(-1)\). For \(n\geq1\), define
    \[
        B_n\coloneqq
        \left\{
            k:\ \exp(-(n+1))<r_k\leq \exp(-n)
        \right\}.
    \]
    Each $B_n$ is finite because $r_k\downarrow 0$. Empty blocks are ignored. If \(B_n\neq\emptyset\), let \(m_n=\min B_n\)
    and set \(M_n=\max_{k\in B_n}Y_k\). For every \(\lambda>0\), Doob's
    maximal inequality for nonnegative supermartingales gives
    \[
        \mathbb P\left(
            \sup_{j\geq m_n}Y_j\geq\lambda
        \right)
        \leq
        \frac{\Exp[Y_{m_n}]}{\lambda}.
    \]
    Hence, for every \(\varepsilon>0\),
    \[
    \begin{aligned}
        \mathbb P\left(
            M_n\geq \varepsilon \exp(-n)L(n)
        \right)
        &\leq
        \frac{\Exp[Y_{m_n}]}{\varepsilon \exp(-n)L(n)}
        \leq
        \frac{r_{m_n}}{\varepsilon \exp(-n)L(n)}
        \leq
        \frac{1}{\varepsilon L(n)} ,
    \end{aligned}
    \]
    where the last inequality uses \(m_n\in B_n\), so \(r_{m_n}\leq \exp(-n)\).
    Since \(\sum_n1/L(n)<+\infty\), the Borel--Cantelli lemma implies that,
    for every fixed \(\varepsilon>0\), almost surely
    \(M_n<\varepsilon \exp(-n)L(n)\) for all sufficiently large nonempty blocks. Applying this conclusion with $\varepsilon=1/q$ for every
    $q\in\mathbb N$ and taking the intersection of the resulting
    probability-one events, we obtain
    \[
        \frac{M_n}{\exp(-n)L(n)}
        \longrightarrow0
        \qquad\text{almost surely}
    \]
    along the nonempty blocks. 

    Now let \(k\in B_n\). Then \(\exp(-(n+1))<r_k\leq \exp(-n)\), and hence
    \(\exp(-n)<er_k\). Moreover,
    \(n\leq\log(1/r_k)<n+1\). Since \(L\) is non-decreasing,
    \(L(n)\leq L(\log(1/r_k))\). Therefore,
    \[
        \frac{Y_k}{r_kL\!\left(\log\frac{1}{r_k}\right)}
        \leq
        e\,\frac{M_n}{\exp(-n)L(n)}
        \longrightarrow0
        \qquad\text{almost surely}.
    \]
    Since the block index $n$ tends to infinity as $k\rightarrow \infty$, the right-hand side converges to zero almost surely. This proves the claim.
\end{proof}}


\revise{The following example shows that an analysis relying solely on the conditional descent property~\eqref{eq:condexp-descent} cannot, in general, eliminate the additional loss in the almost-sure convergence rates. We will show later that this loss can be avoided by exploiting structures of MRAs.}

\revise{\begin{example}[Failure of endpoint almost-sure rates under conditional drift]
\label{exam:drift-no-endpoint-rate}
Consider a nonnegative adapted sequence satisfying a conditional descent inequality of the form~\eqref{eq:condexp-descent} with a power-type modulus $\phi_{\B}(t)\sim t^p$:
\[
    \Exp[Y_{k+1}\mid\cF_k]
    \leq
    Y_k-aY_k^p,
\]
where $a>0$, $p>1$, and $\alpha_k\equiv1$. Writing
$u_k:=\Exp[Y_k]$, taking expectations and applying Jensen's inequality give
\[
    u_{k+1}\leq u_k-au_k^p.
\]
A standard comparison argument therefore yields
\[
    \Exp[Y_k]
    \leq
    \left(u_0^{1-p}+a(p-1)k\right)^{-\frac1{p-1}}
    =
    O\left(k^{-\frac1{p-1}}\right).
\]
We next show that, even if $\{Y_k\}$ is pathwise non-increasing, the corresponding endpoint almost-sure rate $Y_k=O\left(k^{-\frac1{p-1}}\right)$ does not follow in general.

Let $y_j=2^{-j}y_0$ for $j\geq0$, where $y_0>0$ is sufficiently small that
$2ay_0^{p-1}\leq1/2$. Define a Markov chain $\{Y_k\}_{k\geq0}$ with
$Y_0=y_0$ as follows: whenever $Y_k=y_j$,
\[
    Y_{k+1}
    =
    \begin{cases}
        y_{j+1},&
        \text{with probability }2ay_j^{p-1},\\[1ex]
        y_j,&
        \text{with probability }1-2ay_j^{p-1}.
    \end{cases}
\]
Then
\[
\begin{aligned}
    \Exp[Y_{k+1}\mid\cF_k]
    &=
    (1-2ay_j^{p-1})y_j
    +
    2ay_j^{p-1}y_{j+1}\\
    &=
    y_j-ay_j^p
    =
    Y_k-aY_k^p.
\end{aligned}
\]
Moreover, $Y_{k+1}\leq Y_k$ almost surely. For each $j\geq0$, define the entrance time
\[
    \tau_j:=\inf\{k\geq0:Y_k=y_j\}
\]
and the holding time at level $y_j$ by
\[
    W_j:=\inf\{m\geq1:Y_{\tau_j+m}=y_{j+1}\}.
\]
Then $\{W_j\}_{j\geq0}$ are independent geometric random variables with
success probabilities $\theta_j:=2ay_j^{p-1}$. For $j\geq2$, set
\[
    m_j:=\left\lceil\frac{\log j}{\theta_j}\right\rceil,
    \qquad
    A_j:=\{W_j\geq m_j\}.
\]
Using $(1-x)^t\geq\exp(-tx/(1-x))$, we obtain
\[
    \mathbb P(A_j)=
    (1-\theta_j)^{m_j-1}\geq
    (1-\theta_j)^{\frac{\log j}{\theta_j}}\geq
    \exp\left(-\frac{\log j}{1-\theta_j}\right)=
    \frac1j
    \exp\left(-\frac{\theta_j\log j}{1-\theta_j}\right).
\]
Since $\theta_j\log j\to0$, there exists $c>0$ such that $\mathbb P(A_j)\geq\frac{c}{j}$ for all sufficiently large $j$. Hence
\[
    \sum_{j=2}^{\infty}\mathbb P(A_j)=+\infty.
\]
By independence and the second Borel--Cantelli lemma, $\mathbb P(A_j\ \mathrm{i.o.})=1$. Let $K_j:=\tau_j+m_j-1.$ On $A_j$, the chain remains at level $y_j$ until time $K_j$, so
$Y_{K_j}=y_j$. Moreover, since $\tau_j\geq j\geq1$,
\[
    K_j
    \geq
    \frac{\log j}{\theta_j}
    =
    \frac{\log j}{2a}y_j^{-(p-1)}.
\]
Consequently, on $A_j$,
\[
    K_j^{\frac1{p-1}}Y_{K_j}
    \geq
    \left(\frac{\log j}{2a}\right)^{\frac1{p-1}}
    \longrightarrow+\infty.
\]
Since $A_j$ occurs infinitely often almost surely,
\[
    \limsup_{k\to\infty}
    k^{\frac1{p-1}}Y_k
    =
    +\infty
    \qquad\text{almost surely}.
\]
\end{example}}

\revise{To the best of our knowledge, there is no systematic approach for deriving convergence rates of iterates based solely on the information provided by the MRA framework under the GUEB condition. In certain special settings, additional problem-specific structures enable a direct analysis of iterate convergence. For example, in common fixed point problems involving families of firmly nonexpansive operators, the operator structure provides stronger information beyond Assumption~\ref{assmp:basic}, allowing iterate rates to be established directly; see Section~\ref{section:RKMM}. However, such additional structures are generally unavailable in most optimization problems considered later in this paper.}

\subsubsection{\revise{Existing methods under local error bounds}}\label{sec:exist-LUEB}

\revise{The convergence analysis of algorithms with MRA-type structures under local
error-bound conditions has been studied primarily in the optimization
literature, especially for randomized coordinate descent (RCD) and its
variants under local Kurdyka--{\L}ojasiewicz (KL) conditions. For instance,
\cite{chouzenoux2016block, Xu2017Globally, latafat2022block,
latafat2022bregman} establish convergence rates for RCD-type methods
under local KL assumptions, but rely on the essential cyclicity assumption,
which requires every coordinate to be selected at least once within any
fixed number of consecutive iterations. As mentioned in the introduction,
this assumption is violated almost surely under the i.i.d.
sampling scheme commonly used in RCD variants.

More recently, local KL analyses without essential cyclicity have been
developed in~\cite{chorobura2023random, li2025smoothing}. Nevertheless, these results are based on a conditional descent lemma and establish only asymptotic convergence rates for the expected objective gap, up to an additive error of order $\sqrt{\delta}$, after an iteration threshold $k_\delta$, where $k_\delta\to\infty$ as $\delta\to0$. These results therefore do not provide explicit convergence rates. In another direction,
\cite{Ptracu2013EfficientRC} establishes asymptotic linear convergence of
the expected objective value under a local Lipschitzian error bound.
However, the analysis requires the time at which the iterates enter the
local error-bound neighborhood to be deterministic, namely dependent only
on the initial point rather than the random realization. This requirement
is automatically satisfied for global error bounds but is generally
unavailable under genuinely local assumptions.

These limitations also raise a more fundamental question about the role of expectation-based analysis under local error-bound conditions. Since local error bounds are only valid after the iterates enter a neighborhood of the solution set, convergence estimates in this setting are inherently asymptotic rather than nonasymptotic. Consequently, the usual advantage of in-expectation analysis---providing explicit average-case performance guarantees at a prescribed iteration---is largely diminished. In the asymptotic setting, almost-sure convergence rates become more informative, as they characterize the behavior along almost every realized trajectory rather than merely the average behavior. This motivates the sample-path analysis aimed at establishing almost-sure convergence rates for both objective values and iterates; see Section~\ref{sec:localEB}.

}

\section{\revise{Main Results}}

\subsection{\revise{A new analytical framework based on an almost-sure descent lemma}}

\revise{To overcome these limitations, we develop a new analytical framework for MRAs based on the following almost-sure descent lemma, which allows us to separate the stochastic components from the deterministic descent mechanism.}

\revise{\begin{lemma}\label{lem:pre-stochastic}
    Suppose that Assumption~\ref{assmp:basic} holds. Then, for any $k\ge 0$, there exists a $\cF_{k+1}$-measurable random variable $z_k$ such that
    \begin{equation}\label{eq:Random-Descent}
    h(x^{k+1})\le h(x^k) - \cona  \alpha_k\, z_k\,  g(x^k),
    \end{equation}
    where $0\le z_k \le c_3$ and $\Exp[z_k|\cF_k]\ge c_2$ almost surely.
\end{lemma}}

\revise{\begin{proof}
Define
\[
z_k = \begin{cases}
    \frac{\norm{x^{k+1}-x^k}^2}{\alpha_k^2 g(x^k)}, & \text{if } g(x^k) > 0,\\
    \conb, & \text{if } g(x^k) = 0.
\end{cases}
\]
By construction, $z_k$ is a nonnegative $\cF_{k+1}$-measurable random variable. Moreover, since \eqref{eq:basic_4} implies that $\norm{x^{k+1}-x^k}=0$ whenever $g(x^k)=0$, the identity $\norm{x^{k+1}-x^k}^2 = z_k\, \alpha_k^2 \, g(x^k)$ holds almost surely. It also follows from \eqref{eq:basic_4} that $z_k \leq \conc$ almost surely. Combining the preceding identity with Assumption~\ref{assmp:basic}\ref{basic_1}, we obtain \eqref{eq:Random-Descent}. It remains to verify the conditional lower bound on $z_k$. Since $g(x^k)$ is $\cF_k$-measurable, we have almost surely that
\[
\begin{aligned}
\Exp[z_k|\cF_k] 
= \frac{\Exp[\norm{x^{k+1}-x^k}^2\mid\cF_k]}{\alpha_k^2 g(x^k)}\mathbf{1}_{\{g(x^k)>0\}} + \conb\mathbf{1}_{\{g(x^k) = 0\}} 
\ge \conb \left(\mathbf{1}_{\{g(x^k)>0\}}  + \mathbf{1}_{\{g(x^k) = 0\}}\right) = \conb,
\end{aligned}
\]
where the inequality follows from \eqref{eq:expbound}. This completes the proof.
\end{proof}}

\revise{Lemma~\ref{lem:pre-stochastic} suggests a two-stage route to the convergence analysis. First, along each realization, we regard the variables $\{z_k\}$ as a fixed parameter sequence and derive pathwise convergence bounds in terms of their weighted partial sums. We then control these random bounds using the special random structure of $\{z_k\}$. For simplicity, we illustrate this strategy under the GUEB condition. Lemma~\ref{lem:pre-stochastic} and Assumption~\ref{assmp:Global-EB} yield
\begin{equation}\label{eq:Descent-Global}
    h(x^{k+1})-h^*
    \leq
    h(x^k)-h^*
    -\cona\alpha_k z_k\,
    \phi_{\B}\bigl(h(x^k)-h^*\bigr).
\end{equation}
Applying Lemma~\ref{lem:rate-preliminary}\ref{rate-preliminary1} to
\eqref{eq:Descent-Global}, we obtain
\begin{equation}\label{eq:Descent-Global1}
    \Psi_{\B}\bigl(h(x^{k+1})-h^*\bigr)
    \geq
    \Psi_{\B}\bigl(h(x^k)-h^*\bigr)
    +\cona\alpha_k z_k\,
    \mathbf{1}_{\{h(x^k)>h^*\}}.
\end{equation}
Thus, pathwise convergence is governed by the accumulated random descent
$\sum_{i<k}\alpha_i z_i$. In the common regime where $1/\phi_{\B}$ is
not integrable near $0$, this observation yields the following lemma.}

\begin{lemma}\label{lem:random-convergence}
    Suppose that Assumptions~\ref{assmp:basic} and~\ref{assmp:Global-EB} hold and ${1}/{\phi_{\B}}$ is not integrable near $0$. Then,
    \begin{equation}\label{eq:random-convergence}
    h(x^k)-h^* \leq \Psi_{\B}^{-1}\left(\Psi_{\B}(h(x^0)-h^*)+\revise{\cona}\sum_{i=0}^{k-1} \alpha_i z_i\right)\quad \forall k \ge 1.
    \end{equation}
\end{lemma}
\begin{proof}
    Given any $\omega \in \Omega$ and integer $k >0$, if $h(x^{k-1}(\omega)) = h^*$, then \eqref{eq:random-convergence} holds trivially. Otherwise, if $h(x^{k-1}(\omega)) > h^*$, it follows that $h(x^i(\omega)) > h^*$ for all $i \leq k - 1$. By inequality~\eqref{eq:Descent-Global1}, we have that for all $0\leq i\leq k-1$,
    \[
    \Psi_{\B}(h(x^{i+1}(\omega))-h^*) \geq \Psi_{\B}(h(x^i(\omega))-h^*) + \revise{\cona} \alpha_i z_i(\omega).
    \]
    Summing both sides of this inequality from $i = 0$ to $k - 1$ yields
    \[
    \Psi_{\B}(h(x^{k}(\omega))-h^*) \geq \Psi_{\B}(h(x^0(\omega))-h^*) + \revise{\cona} \sum_{i=0}^{k-1} \alpha_i z_i(\omega),
    \]
    which implies \eqref{eq:random-convergence}.
\end{proof}

\revise{Using GUEB as a representative setting, our framework consists of four steps:
\begin{enumerate}[label=(\arabic*)]
    \item Derive the almost-sure descent inequality~\eqref{eq:Random-Descent}.
    \item Obtain the pathwise bound~\eqref{eq:random-convergence} for $h(x^k)$ in terms of the realized sequence $\{z_i\}_{i\geq0}$.
    \item\label{framework-step3} Establish in-expectation and almost-sure upper bounds for the random variable
    \[
        \Psi_{\B}^{-1}\left(
            \Psi_{\B}\bigl(h(x^0)-h^*\bigr)
            +\cona\sum_{i=0}^{k-1}\alpha_i z_i
        \right).
    \]
    \item\label{framework-step4} Transfer the resulting rates for $h(x^k)-h^*$ to rates for the iterates $\{x^k\}$.
\end{enumerate}
The first two steps have been completed above, while Steps~\ref{framework-step3} and~\ref{framework-step4} are developed in Section~\ref{sec:globalEB}. A key advantage of our approach is that it largely separates the deterministic error-bound analysis from the probabilistic control of the accumulated random descent. As a result, the distinction between the global and local error-bound settings becomes substantially narrower than in approaches centered on conditional-expectation recursions. Although the LUEB case still requires additional localization arguments, most of the tools developed for almost-sure rates under GUEB remain applicable.}

\subsection{Convergence analysis under global error bounds}\label{sec:globalEB}

We first present the case where ${1}/{\phi_{\B}}$ is integrable near $0$.

\begin{theorem}\label{thm:global-finite-terminate}
    Suppose that Assumptions \ref{assmp:basic}, \ref{assmp:step-size} and \ref{assmp:Global-EB} hold and that ${1}/{\phi_{\B}}$ is integrable near $0$. Then, there exist an almost surely finite random index $K_1$ and a random variable $x^*$ such that almost surely, $ x^*\in \arg\min_{x \in \B} h(x)$ and $x^k = x^*$ for all $k \geq K_1$.
\end{theorem}

\begin{proof}
Taking conditional expectations on both sides of \eqref{eq:Descent-Global1} yields
\[
\Exp[\Psi_{\B}(h(x^{k+1})-h^*)\mid \cF_k] \geq \Psi_{\B}(h(x^k)-h^*) + \revise{\cona} \alpha_k\, \Exp[z_k\mid \cF_k] \mathbf{1}_{\{h(x^k)>h^*\}}.
\]
Since $ \mathbb{E}[z_k \mid \mathcal{F}_k] \geq \revise{\conb} $ almost surely, it follows that
\[
\Exp[\Psi_{\B}(h(x^{k+1})-h^*)] \geq \Exp[\Psi_{\B}(h(x^k)-h^*)] + \revise{\cona\conb} \alpha_k\, \Exp[\mathbf{1}_{\{h(x^k)>h^*\}}].
\]
 Since $ \mathbb{E}[\Psi_{\B}(h(x^k) - h^*)] $ is non-decreasing with respect to $ k $ and $ \Psi_{\B}(t) \leq \Psi_{\B}(0) < +\infty $ for all $ t \in [0, \tau) $, it follows that
    \[
    \Exp\left[ \sum_{k=0}^\infty \alpha_k \mathbf{1}_{\{h(x^k)>h^*\}} \right] = \sum_{k=0}^\infty \alpha_k\, \Exp \left[ \mathbf{1}_{\{h(x^k)>h^*\}} \right] \leq 
    \revise{\frac{\Psi_{\B}(0)-\Psi_{\B}(h(x^0)-h^*)}{\cona\conb}}  < +\infty,
    \]
    which implies that $ \sum_{k=0}^\infty \alpha_k \mathbf{1}_{\{h(x^k) > h^*\}} < \infty $ almost surely. Combined with $ \sum_k \alpha_k = +\infty $ and the non-increasing property of $ h(x^k) $, we conclude that $ h(x^k) $ reaches $ h^* $ in a finite number of iterations almost surely. According to Assumption~\ref{assmp:basic}\ref{basic_1}, once $h(x^k)$ attains $h^*$, the algorithm terminates in finitely many iterations\footnote{Given any $\omega \in \Omega$, we say that the sequence $\{x^k(\omega)\}_k$ terminates in finitely many iterations if there exists some $K(\omega) \in \mathbb{N}$ such that $x^k(\omega) = x^K(\omega)$ for all $k\geq K(\omega)$.} starting from iteration $k$. This completes the proof.
\end{proof}

We now consider the case where ${1}/{\phi_{\B}}$ is not integrable near $0$. 
\begin{theorem}\label{th_general_exp_rate1}
    Suppose that Assumptions \ref{assmp:basic}, \ref{assmp:step-size} and \ref{assmp:Global-EB} hold and ${1}/{\phi_{\B}}$ is not integrable near $0$. \revise{Denote $\mathfrak{h}:=\frac{1-\log 2}{2}$.}
     \begin{enumerate}[label=\textup{\textrm{(\roman*)}}]
        \item\label{Th_exp_1} For all $k \ge 1$, 
        \begin{equation}\label{eq:hxk-rate-exp}
            \Exp[h(x^k)-h^*] \le \exp\! \left(\!\revise{-\mathfrak{h}\frac{c_2}{c_3\bar\alpha}A_k}\right) (h(x^0)-h^*) + \Psi_{\B}^{-1}\!\!\prt{\! \Psi_{\B}(h(x^0) \! - \! h^*)+\revise{\frac{\cona \conb}{2} A_k} \!}.
        \end{equation}
        Suppose additionally that $1/\sqrt{\phi_{\B}}$ is integrable near $0$. Then, there exists a random variable $x^*$ such that $x^k$ converges to $x^*$ both almost surely and in $L^1$, that $x^* \in \argmin_{x\in \B} h(x)$ almost surely, and that 
        \begin{equation}\label{eq:ite-rate-exp_Th3}
            \Exp[\norm{x^k-x^*}] \leq \frac{\sqrt{\conc}}{\cona \conb }\Phi_{\B}(\Exp[h(x^k)-h^*])\quad\forall k\ge 0 .
        \end{equation}
        \item\label{Th_exp_2} There exists an almost surely finite random index $K_2$ such that 
        \begin{equation}\label{eq:global-hxk-asrate}
            h(x^{k})-h^* \leq \Psi_{\B}^{-1} \left( \revise{\frac{\cona \conb}{2} A_k} \right) \quad \forall k \ge K_2\  \as.
        \end{equation}
        Suppose additionally that $1/\sqrt{\phi_{\B}}$ is integrable near $0$. Then, for the limit $x^*$ established in \ref{Th_exp_1}, there exists an almost surely finite random index $K_3$ such that
        \begin{equation}\label{eq:global-ite-asrate}
            \norm{x^k-x^*} \leq \revise{\frac{\sqrt{2}}{\cona\sqrt{\conb}}} \Phi_{\B} \left( \Psi_{\B}^{-1} \left( \revise{\frac{\cona \conb}{2} A_k} \right) \right) \quad \forall k \ge K_3\  \as.
        \end{equation}
        \end{enumerate}
\end{theorem}

\revise{In most applications, $\phi_{\B}(t)=O(t)$ as $t\downarrow0$, which implies that the second, EB-term in~\eqref{eq:hxk-rate-exp} converges at most exponentially in $A_k$ and hence that the first, exponential term does not qualitatively alter the convergence rate.
Compared with the existing techniques summarized in Theorem~\ref{thm:globalEB-exph-cvx}, Theorem~\ref{th_general_exp_rate1}\ref{Th_exp_1} completely removes the convexity assumption on $\phi_{\B}$, whose restrictiveness is illustrated in Section~\ref{subsubsection:non-convex_moduli}. Moreover, Theorem~\ref{th_general_exp_rate1}\ref{Th_exp_2} eliminates the additional factor $L(\log (r_k^{-1}))$ in Lemma~\ref{lem:exp-to-as-little-o}.} The proof of Theorem~\ref{th_general_exp_rate1} is based on \revise{three analytical tools}, Propositions~\ref{prop:exp-est}-\ref{prop:iter-global-as}. The first two facilitates the establishment of in-expectation and almost-sure convergence rates of $h(x^k)$; while the third transfers the \revise{almost sure} convergence rates of $h(x^k)$ to $x^k$. 

\subsubsection{An apparatus for in-expectation convergence rate analysis}
To estimate $\Exp[h(x^k)-h^*]$, by Lemma~\ref{lem:random-convergence}, it suffices to estimate
\[
\Exp \left[\Psi_{\B}^{-1}\left(\Psi_{\B}(h(x^0)-h^*)+\revise{\cona}\sum_{i=0}^{k-1} \alpha_i z_i \right) \right].
\]
\revise{Note that $\Psi_{\B}^{-1}$ is decreasing and convex. Indeed, by Lemma~\ref{lem:functions}\ref{lem_func_1}, $\Psi_{\B}$ is decreasing and convex, and the inverse of a decreasing convex function is again decreasing and convex on its range.} Then, the above task is abstracted as follows.

\begin{question}\label{q:global-Exp}
    Consider a random sequence $\{X_k\}_{k \geq 0}$ such that $X_k$ is $\cF_{k+1}$-measurable, \revise{$0\le X_k \le b$ and $\Exp[X_k\mid \cF_k] \ge a$ for all $k\ge 0$, where $0<a\le b$}. Let $\{\alpha_k\}_{k \geq 0}$ be a sequence of positive coefficients. Define the partial sum $ S_k = \sum_{i=0}^{k-1} \alpha_i X_i $, and let $\chi: \R_+ \to \R$ be a given non-increasing \revise{convex} function. Can we upper bound $ \Exp[\chi(S_k)] $ for a given $ k $?
\end{question}

\revise{Although the random variables $\{X_k\}_{k\ge0}$ may be dependent, their boundedness and conditional mean lower bound allow us to compare $\Exp[\chi(S_k)]$ with the corresponding expectation of a weighted sum generated by an independent Bernoulli sequence.}

\revise{\begin{lemma}\label{lem:convex-Bernoulli-comparison}
    Suppose that the setting in Question~\ref{q:global-Exp} holds. Let $\{Y_k\}_{k\ge0}$ be an i.i.d. Bernoulli sequence with $Y_k\sim \operatorname{Bernoulli}(p)$ with $p= a/b$.
    Define
    \[
        T_k := b\sum_{i=0}^{k-1}\alpha_iY_i .
    \]
    Then, for any convex non-increasing function $\chi:\R_+\to\R$,
    \[
        \Exp[\chi(S_k)]\le \Exp[\chi(T_k)]
        \qquad \forall k\ge0.
    \]
\end{lemma}}

\revise{\begin{proof}
    We prove the claim by induction on $k$. The case $k=0$ is trivial, since
    $S_0=T_0=0$. Suppose that the claim holds for some $k\ge0$ for every
    convex non-increasing function $\chi:\R_+\to\R$. We now prove it for $k+1$.
    Since $S_k$ is $\cF_k$-measurable, and since $0\le X_k\le b$, the convexity of $\chi$ gives
    \[
        \chi(S_k+\alpha_kX_k)
        \le
        \chi(S_k)
        +
        \frac{X_k}{b}
        \bigl(\chi(S_k+\alpha_k b)-\chi(S_k)\bigr).
    \]
    Taking conditional expectation, and using the monotonicity of $\chi$, we obtain
    \[
    \begin{aligned}
        & \Exp[\chi(S_{k+1})\mid \cF_k] \le
        \chi(S_k)
        +
        \frac{\Exp[X_k\mid\cF_k]}{b}
        \bigl(\chi(S_k+\alpha_k b)-\chi(S_k)\bigr)  \\
        \le &
        \chi(S_k)
        +
        \frac{a}{b}
        \bigl(\chi(S_k+\alpha_k b)-\chi(S_k)\bigr)  
        =
        (1-p)\chi(S_k)+p\chi(S_k+\alpha_k b).
    \end{aligned}
    \]
    Taking expectation on both sides yields $\Exp[\chi(S_{k+1})]
        \le
        (1-p)\Exp[\chi(S_k)]
        +
        p\Exp[\chi(S_k+\alpha_k b)]$.

    Next, the functions $s\mapsto \chi(s)$ and $s\mapsto \chi(s+\alpha_k b)$ are both convex and non-increasing on $\R_+$. Hence, by the induction hypothesis, $\Exp[\chi(S_k)]\le \Exp[\chi(T_k)]$ and $\Exp[\chi(S_k+\alpha_k b)] \le \Exp[\chi(T_k+\alpha_k b)]$.
    Consequently,
    \[
        \Exp[\chi(S_{k+1})]
        \le
        (1-p)\Exp[\chi(T_k)]
        +
        p\Exp[\chi(T_k+\alpha_k b)].
    \]

    Since $Y_k$ is independent of $Y_0,\ldots,Y_{k-1}$ and satisfies
    $Y_k\sim\operatorname{Bernoulli}(p)$, we have
    \[
    \begin{aligned}
        \Exp[\chi(T_{k+1})]
        =
        \Exp[\chi(T_k+b\alpha_kY_k)] 
        =
        (1-p)\Exp[\chi(T_k)]
        +
        p\Exp[\chi(T_k+\alpha_k b)].
    \end{aligned}
    \]
    Combining the last two displays gives $\Exp[\chi(S_{k+1})]\le \Exp[\chi(T_{k+1})]$.
    This completes the induction and hence the proof.
\end{proof}}

\revise{
\begin{proposition}\label{prop:exp-est}
    Consider Question~\ref{q:global-Exp}. Suppose that the sequence $\{\alpha_k\}_{k\ge 0}$ is upper bounded by some constant $\bar{\alpha}>0$. Denote $\mathfrak{h}:=\frac{1-\log 2}{2}$.
    Then, for any convex non-increasing function $\chi:\R_+\to\R_+$,
    \begin{equation}\label{eq:exp-est-eq1}
        \Exp[\chi(S_k)]
        \le
        \chi(0)\exp\left(
            -\mathfrak{h}\frac{a}{b\bar\alpha}A_k
        \right)
        +
        \chi\left(\frac{a}{2}A_k\right)
        \qquad \forall k\ge 1.
    \end{equation}
\end{proposition}}

\revise{
\begin{proof}
    By Lemma~\ref{lem:convex-Bernoulli-comparison}, we have $\Exp[\chi(S_k)]\le \Exp[\chi(T_k)]$,
    where
    \[
        T_k=b\sum_{i=0}^{k-1}\alpha_iY_i,
        \qquad
        Y_i\sim \operatorname{Bernoulli}(p)\ \text{i.i.d.},
        \qquad
        p=\frac{a}{b}.
    \]
    Notice that $\Exp[T_k]=bpA_k=aA_k$.
    Since $\chi$ is non-increasing and nonnegative, we obtain
    \[
    \begin{aligned}
        \Exp[\chi(T_k)]
        &=
        \Exp[\chi(T_k)\mathbf 1_{\{T_k\le aA_k/2\}}]
        +
        \Exp[\chi(T_k)\mathbf 1_{\{T_k>aA_k/2\}}]  \\
        &\le
        \chi(0)\mathbb{P}\left(T_k\le \frac{a}{2}A_k\right)
        +
        \chi\left(\frac{a}{2}A_k\right).
    \end{aligned}
    \]
    It remains to bound the lower-tail probability. Define
    \[
        W_k:=\sum_{i=0}^{k-1}\frac{\alpha_i}{\bar\alpha}Y_i .
    \]
    Then, $T_k=b\bar\alpha W_k$, and
    \[
        \Exp[W_k]=\frac{p}{\bar\alpha}A_k
        =
        \frac{a}{b\bar\alpha}A_k.
    \]
    Hence,
    \[
        T_k\le \frac{a}{2}A_k
        \quad\Longleftrightarrow\quad
        W_k\le \frac12\Exp[W_k].
    \]
    For any $\lambda>0$, Markov's inequality gives
    \[
    \begin{aligned}
        \mathbb{P}\left(W_k\le \frac12\Exp[W_k]\right)
        &\le
        \exp\left(\frac{\lambda}{2}\Exp[W_k]\right)
        \Exp[e^{-\lambda W_k}]  \\
        &=
        \exp\left(\frac{\lambda}{2}\Exp[W_k]\right)
        \prod_{i=0}^{k-1}
        \left(
            1-p+p\exp\left(-\lambda\frac{\alpha_i}{\bar\alpha}\right)
        \right).
    \end{aligned}
    \]
    Using $1-u\le e^{-u}$ and the concavity of $t\mapsto 1-e^{-\lambda t}$ on $[0,1]$, we have
    \[
    \begin{aligned}
        & \prod_{i=0}^{k-1}
        \left(
            1-p+p\exp\left(-\lambda\frac{\alpha_i}{\bar\alpha}\right)
        \right)
        \le
        \exp\left[
            -p\sum_{i=0}^{k-1}
            \left(
                1-\exp\left(-\lambda\frac{\alpha_i}{\bar\alpha}\right)
            \right)
        \right]  \\
        \le& \exp\left[
            -p(1-e^{-\lambda})
            \sum_{i=0}^{k-1}\frac{\alpha_i}{\bar\alpha}
        \right]  
        =
        \exp\left[
            -(1-e^{-\lambda})\Exp[W_k]
        \right].
    \end{aligned}
    \]
    Therefore,
    \[
        \mathbb{P}\left(W_k\le \frac12\Exp[W_k]\right)
        \le
        \exp\left[
            \left(\frac{\lambda}{2}-(1-e^{-\lambda})\right)\Exp[W_k]
        \right].
    \]
    Choosing $\lambda=\log 2$, we obtain
    \[
        \frac{\lambda}{2}-(1-e^{-\lambda})
        =
        \frac{\log 2}{2}-\frac12
        =
        -\frac{1-\log 2}{2}
        =
        -\mathfrak{h}.
    \]
    Thus,
    \[
        \mathbb{P}\left(T_k\le \frac{a}{2}A_k\right)
        \le
        \exp\left(
            -\mathfrak{h}\frac{a}{b\bar\alpha}A_k
        \right).
    \]
    Combining this estimate with the preceding bound on $\Exp[\chi(T_k)]$ gives
    \[
        \Exp[\chi(S_k)]
        \le
        \chi(0)\exp\left(
            -\mathfrak{h}\frac{a}{b\bar\alpha}A_k
        \right)
        +
        \chi\left(\frac{a}{2}A_k\right),
    \]
    which proves the desired result.
\end{proof}
}

\subsubsection{An apparatus for almost-sure convergence rate analysis}
Based on \eqref{eq:random-convergence}, establishing an almost-sure upper bound for $ h(x^{k}) $ reduces to deriving an almost sure lower bound for $ \sum_{i=0}^{k-1}\alpha_i z_i $. This task is abstracted as follows. 

\begin{question}\label{q:2}
\revise{Consider the same random sequence $\{X_k\}_{k \geq 0}$ defined in Question~\ref{q:global-Exp}.} Given a positive sequence $\{\alpha_k\}_{k\ge 0}$, can we obtain an almost-sure asymptotic lower bound on $\sum_{i=0}^{k-1} \alpha_i X_i$?
\end{question}

We answer this question in the proposition below, whose proof closely follows that of~\cite[Theorem 1]{wu2012asymptotics}.

\begin{proposition}\label{prop:as-bound}
    Consider Question~\ref{q:2}. Suppose that sequence $\{\alpha_k\}_{k\ge 0}$ satisfies Assumption~\ref{assmp:step-size}. Then,
    \[
    \liminf_{k\rightarrow \infty}\frac{\sum_{j=0}^{k-1} \alpha_j X_j}{\sum_{j=0}^{k-1} \alpha_j} \geq \revise{a}\quad \text{ a.s.}.
    \]
\end{proposition}
\begin{proof}
    \revise{Recall that $A_{k} = \sum_{i=0}^{k-1} \alpha_i$ for $k\ge 1$}. Define $Z_k = X_k-\Exp[X_k|\cF_{k}]$ for $k\ge 0$. Then, $\frac{\alpha_k}{\revise{A_{k+1}}} Z_k$ is a martingale difference and $|Z_k|\leq \revise{b}$ almost surely. Define $U_k = \sum_{j = 0}^{k-1} \frac{\alpha_j}{\revise{A_{j+1}}} Z_j$ for $k\geq 1$. Then, $U_k$ is a martingale. Observing that
    \[
    \Exp\left[\left(\frac{\alpha_j}{\revise{A_{j+1}}} Z_j\right)^2 \Bigg|\cF_{j}\right] \leq \left(\frac{\alpha_j \revise{b}}{\revise{A_{j+1}}}\right)^2\quad \text{a.s.},
    \]
    \revise{For $k\geq1$, by Assumption~\ref{assmp:step-size}, we have
    \[
    \sum_{k=1}^{\infty}
    \left(\frac{\alpha_k}{A_{k+1}}\right)^2
    \leq \bar{\alpha} \sum_{k=1}^{\infty}\frac{\alpha_k}{A_{k+1}^2}
    \leq \bar{\alpha}\sum_{k=1}^{\infty}\int_{A_k}^{A_{k+1}}\frac{dt}{t^2}
    = \bar{\alpha}\int_{A_1}^{\infty}\frac{dt}{t^2}
    =\frac{\bar{\alpha}}{\alpha_0}<+\infty.
    \]}
    Then $\sum_{j=0}^\infty \Exp[(\frac{\alpha_j}{\revise{A_{j+1}}} Z_j)^2|\cF_{j}]<\infty$ almost surely. By \cite[Theorem 2.17]{hall2014martingale}, $U_k$ converges almost surely. Furthermore, since $0 < A_1 \leq A_2 \leq \cdots$ and $\revise{A_k} \to +\infty$, an application of Kronecker's lemma \cite[Lemma~IV.3.2]{shiryaev1996probability} yields
    \[
    \revise{\frac{1}{A_k} \sum_{j=0}^{k-1} \alpha_j Z_j \rightarrow 0} \quad \text{a.s.,}
    \]
    or equivalently,
    \[
    \sum_{j=0}^{k-1} \frac{\alpha_j X_j}{A_k} - \sum_{j=0}^{k-1} \frac{\alpha_j \Exp[X_j|\cF_{j}]}{A_k} \rightarrow 0 \quad \text{a.s..}
    \]
    \revise{Noting that $\Exp[X_j|\cF_{j}] \geq a$ for all $j\ge 0$ almost surely, the proof is then completed.}
\end{proof}

\subsubsection{Transferring almost sure rates of $h(x^k)$ to $x^k$ with the weak descent condition}

\revise{The aim of this part is to devise a mechanism for establishing the almost-sure convergence rate of $x^k$ from that of $h(x^k)$. This task is considerably more challenging than transferring expected rates, as in Proposition~\ref{prop:iter-global-exp}. As discussed in Section~\ref{sec:exist-det}, existing methods achieve such a transfer only under the strong descent conditions \eqref{eq:SDC1} and \eqref{eq:SDC2}. In our setting, however, \eqref{eq:SDC2} is no longer available. This raises a natural question: is it possible to obtain a rate transformation using only the weak descent condition \eqref{eq:SDC1}? We answer this question affirmatively by showing that a certain rate transformation remains possible under \eqref{eq:SDC1} alone.} Our starting point is Assumption~\ref{assmp:basic}\ref{basic_1}, which yields
\begin{equation}\label{eq:basic_1_variant}
\norm{x^{k+1}-x^k} \leq \sqrt{\frac{\alpha_k}{\cona}(h(x^k)-h(x^{k+1}))} .
\end{equation}
If $ h(x^k) - h^* $ converges linearly, then the desired transfer can be obtained through the simple relaxation
\[
\norm{x^{k+1}-x^k} \leq \sqrt{\frac{\bar{\alpha}}{\cona}(h(x^k)-h^*)} .
\]
This approach has been used in \cite{luo1993error, Tseng2009BlockCoordinate, Tseng2010Coordinate}. However, linear convergence of $h(x^k)-h^*$ typically holds only when $\phi_{\B}$ is linear. \revise{For nonlinear $\phi_{\B}$, directly replacing $h(x^{k+1})$ with $h^*$ yields a bound that is too coarse and may even blow up. To overcome this limitation, we adopt a more refined treatment of the right-hand side of \eqref{eq:basic_1_variant} and derive a trajectory-based rate-transfer result applicable to nonlinear $\phi_{\B}$.}

\begin{proposition}\label{prop:iter-global-as}
    Suppose that Assumptions~\ref{assmp:basic},~\ref{assmp:step-size} and~\ref{assmp:Global-EB} hold, and that \revise{$1/\phi_{\B}$ is not integrable near $0$, whereas $1/\sqrt{\phi_{\B}}$ is integrable near $0$.} Given a fixed sample path $\omega$, if there exists $\epsilon > 0$ and an integer $K\ge 1$ such that $\epsilon A_k$ lies in the domain of $\Psi_\B^{-1}$ for all $k\ge K$ and
    \begin{equation}
        \label{ineq:almost-sure-transfer-condition}
        h(x^k(\omega))-h^* \leq \Psi_{\B}^{-1}\left(\epsilon \revise{A_{k+1}}\right) \quad \forall k\geq K,
    \end{equation}
    then $\{x^k (\omega)\}_{k\ge 0}$ converges to some $x^*(\omega) \in \argmin_{x\in \B} h(x)$ and
    \[
    \norm{x^{k}(\omega)-x^*(\omega)} \leq \frac{1}{\sqrt{\cona\epsilon}} \Phi_{\B}\left(\Psi_{\B}^{-1}\left(\epsilon \revise{A_k}\right)\right) \quad \forall k\geq K.
    \]
\end{proposition}

\begin{proof}
    For any positive, non-increasing sequence $\{\beta_k\}_{k\ge 0}$,  Assumption~\ref{assmp:basic}\ref{basic_1} implies that
    \[
        \norm{x^{k+1}(\omega)-x^k(\omega)} \leq \sqrt{\frac{\alpha_k}{\cona} [h(x^k(\omega)) - h(x^{k+1}(\omega))]} \leq \frac{\alpha_k \beta_k}{2\cona} + \frac{h(x^k(\omega))-h(x^{k+1}(\omega))}{2\beta_k}.
    \] 
\revise{Let $J_{\B}:=\Psi_{\B}((0,\tau))$ and define
$L=\Phi_{\B}\circ\Psi_{\B}^{-1}$ on $J_{\B}$. Since $\Psi_{\B}$ is continuous and strictly decreasing on $(0,\tau)$, $J_{\B}$ is an open interval and $\Psi_{\B}^{-1}$, and hence $L$, is continuous on $J_{\B}$. Fix $t\in J_{\B}$ and set $r=\Psi_{\B}^{-1}(t)$. Since $\Psi_{\B}^{-1}$ is decreasing, approaching $t$ from the left corresponds to approaching $r$ from the right. Thus,
\[
\begin{aligned}
L'_-(t)
&=\lim_{s\downarrow r}
\frac{\Phi_{\B}(s)-\Phi_{\B}(r)}
     {\Psi_{\B}(s)-\Psi_{\B}(r)}
=\frac{\Phi'_{\B,+}(r)}{\Psi'_{\B,+}(r)}
=-\sqrt{\phi_{\B}(r+)}.
\end{aligned}
\]
Similarly, $L'_+(t) = -\sqrt{\phi_{\B}(r-)}$. Since $\Psi_{\B}^{-1}$ is decreasing and $\phi_{\B}$ is non-decreasing, $L'_-$ is non-decreasing on $J_{\B}$. Therefore, $L$ is convex on $J_{\B}$ by
\cite[Chapter~I, Theorem~5.3.1]{hiriart2013convex}. Moreover, since
$\phi_{\B}(r-)\leq\phi_{\B}(r)\leq\phi_{\B}(r+)$,
\[
-\sqrt{\phi_{\B}(\Psi_{\B}^{-1}(t))}
\in[L'_-(t),L'_+(t)]
=\partial L(t),
\]
where the last equality follows from
\cite[Theorem~24.1]{Rockafellar1970}.}
    \revise{If $h(x^k(\omega))=h^*$ for some $k\geq K$, let $K_*:=\min\{k\geq K:h(x^k(\omega))=h^*\}$; otherwise, set $K_*:=+\infty$. Assumption~\ref{assmp:basic}\ref{basic_1} implies that $x^k(\omega)=x^{K_*}(\omega)$ for all $k\geq K_*$ whenever $K_*<+\infty$. For $K\leq k<K_*$, choose $\beta_k = \sqrt{\cona\epsilon\phi_{\B}(h(x^k(\omega))-h^*)}$. This sequence is positive and non-increasing for $K\leq k<K_*$. Applying the subgradient inequality for $L$ at $\epsilon A_{k+1}$ gives
\[
    L(\epsilon A_k)-L(\epsilon A_{k+1})
    \geq
    -\sqrt{\phi_{\B}(\Psi_{\B}^{-1}(\epsilon A_{k+1}))}
    \bigl(\epsilon A_k-\epsilon A_{k+1}\bigr)
    =
    \epsilon\alpha_k
    \sqrt{\phi_{\B}(\Psi_{\B}^{-1}(\epsilon A_{k+1}))}.
\] 
Moreover, for all $K\leq k<K_*$, \eqref{ineq:almost-sure-transfer-condition} gives $\beta_k\leq\sqrt{\cona\epsilon\phi_{\B}(\Psi_{\B}^{-1}(\epsilon A_{k+1}))}$. Consequently,
\[
    \alpha_k\beta_k
    \leq
    \sqrt{\frac{\cona}{\epsilon}}
    \bigl(
        L(\epsilon A_k)-L(\epsilon A_{k+1})
    \bigr).
\]
Meanwhile, Lemma~\ref{lem:rate-preliminary}\ref{rate-preliminary2}, applied with $a=h(x^k(\omega))-h^*$ and $\widetilde a=h(x^{k+1}(\omega))-h^*$, yields
\[
    \beta_k^{-1}
    \bigl(h(x^k(\omega))-h(x^{k+1}(\omega))\bigr)
    \leq
    \frac{\Phi_{\B}(h(x^k(\omega))-h^*)-\Phi_{\B}(h(x^{k+1}(\omega))-h^*)}{\sqrt{\cona\epsilon}}.
\]
Combining the preceding estimates, summing from $i=k$ to $n<K_*$, we obtain
\[
\begin{aligned}
    & \sum_{i=k}^n \norm{x^{i+1}(\omega)-x^i(\omega)}
    \leq \frac{1}{2}\sum_{{i=k}}^n \srt{\frac{\alpha_i \beta_i}{\cona} + \beta_i^{-1} [h(x^i(\omega))-h(x^{i+1}(\omega))]}\\
    \leq & \frac{1}{2\sqrt{\cona\epsilon}}
    \left[L(\epsilon A_k)+\Phi_{\B}(h(x^k(\omega))-h^*)\right]
    \leq \frac{1}{\sqrt{\cona\epsilon}}
    \Phi_{\B}(\Psi_{\B}^{-1}(\epsilon A_k)),
\end{aligned}
\]
where the last inequality follows from \eqref{ineq:almost-sure-transfer-condition} and the monotonicity of $\Phi_{\B}$ and $\Psi_{\B}^{-1}$. If $K_*=+\infty$, let $n\to\infty$; otherwise, take $n=K_*-1$ and use $x^i(\omega)=x^{K_*}(\omega)$ for all $i\geq K_*$, yielding the same bound for $\sum_{i=k}^{\infty}\norm{x^{i+1}(\omega)-x^i(\omega)}$ in both cases. The same bound holds trivially for $k\geq K_*$ when $K_*<+\infty$. Thus, $\{x^k(\omega)\}_{k\ge 0}$ converges to some $x^*(\omega)\in\argmin_{x\in\B}h(x)$, and the desired inequality follows from the preceding tail bound.}
This completes the proof.
\end{proof}

\revise{
    Proposition~\ref{prop:iter-global-as} may be of independent interest. It provides a unified approach to analyzing KL-type convergence rates for both gradient-descent-type and proximal-gradient-type algorithms: one first establishes a convergence rate for the function values and then transfers it to the iterates using only the weak descent condition. A detailed development of this application is beyond the scope of the present paper and is therefore omitted.
}

\subsubsection{Proof of Theorem~\ref{th_general_exp_rate1}}
We begin with assertion~\ref{Th_exp_1}. 
By inequality~\eqref{eq:random-convergence}, applying Proposition~\ref{prop:exp-est} with $\chi(t)= \Psi_{\B}^{-1}(\Psi_{\B}(h(x^0)-h^*)+\revise{\cona} t)$ and $X_k = z_k$, we obtain inequality~\eqref{eq:hxk-rate-exp}. The existence of the limit $x^*$ and inequality~\eqref{eq:ite-rate-exp_Th3} follow directly from Proposition~\ref{prop:iter-global-exp}. It remains to prove that $x^*\in \argmin_{x\in \B} h(x)$ almost surely. From \eqref{eq:hxk-rate-exp} and Lemma~\ref{lem:functions}\ref{lem_func_1}, we have $ \mathbb{E}[h(x^k)-h^*]\to 0$, which implies $h(x^k)\to h^*$ in $L^1$. Therefore, by \cite[Theorem~4.3.3]{junghenn2018principles}, there exists a subsequence $ \{k_n\}_n $ such that $ h(x^{k_n}) \to h^* $ almost surely, which in turn implies that $x^*\in \argmin_{x\in \B} h(x)$. 

Next, we prove assertion~\ref{Th_exp_2}. Define the event
\[
E = \left\{\omega: \liminf_{k\rightarrow \infty}\frac{\sum_{i=0}^{k-1} \alpha_i z_i}{\sum_{i=0}^{k-1} \alpha_i} \geq \revise{c_2}\right\}. 
\]
Under Assumption~\ref{assmp:step-size}, applying Proposition~\ref{prop:as-bound} with $X_k = z_k$ yields $\mathbb{P}(E) = 1$. Define
\[
K_2 = \min\left\{N:\Psi_{\B}(h(x^0)-h^*)+\revise{\cona}\sum_{i=0}^{k-1} \alpha_i z_i \geq \revise{\frac{3\cona \conb}{4}}\revise{A_k},~~\forall~k\geq N \right\}. 
\]
Then, $K_2$ is finite on $E$. By the definition of $K_2$, Lemma~\ref{lem:random-convergence} and the monotonicity of $\Psi_\B^{-1}$, 
we obtain \eqref{eq:global-hxk-asrate}.
Moreover, by Assumption~\ref{assmp:step-size}, there exists a deterministic $N$ such that
\[
\revise{\frac{3}{4}}\revise{A_k} \geq \revise{\frac{1}{2}} \revise{A_{k+1}},\quad \forall~k\geq N.
\]
Define $K_3 = \max\{K_2, N\}$, which is also finite on $E$. By inequality~\eqref{eq:global-hxk-asrate}, we have on $E$ that
\[
    h(x^k)-h^*  \leq \Psi_{\B}^{-1} \left( \revise{\frac{\cona \conb}{2}} \revise{A_{k+1}} \right) \quad \forall  k\geq K_3.
\]
Applying Proposition~\ref{prop:iter-global-as} with $\epsilon = \revise{\frac{\cona \conb}{2}}$, we obtain \eqref{eq:global-ite-asrate}.

\subsection{Convergence analysis under local error bounds}\label{sec:localEB}
Our next goal is to analyze the algorithms in Assumption~\ref{assmp:basic} under a local error bound condition defined below. From now on, we denote by $F = g^{-1}(\{0\})\cap \B$ the target set.

The following two events will be crucial to our analysis in this section:
\begin{equation}\label{eq:events}
E_0  = \left\{\omega: g(x^k)\rightarrow 0,\  \liminf_{k\rightarrow \infty}\frac{\sum_{i=1}^k \alpha_i z_i}{\sum_{i=1}^k \alpha_i} \geq \revise{c_2}\right\}\quad\text{and}\quad E_1 = \{\omega: x^k\text{ is bounded}\}.
\end{equation}
We now present our main convergence rate result in the local error bound setting. 

\begin{theorem}\label{thm:local-EB}
    Suppose that Assumptions~\ref{assmp:basic},~\ref{assmp:step-size} and~\ref{assmp:local-EB} hold and that $1/\sqrt{\phi_{\bar{x}}}$ is integrable near $0$ for all $\bar{x}\in F$. Then, there exists a random variable $x^*$ such that $x^*\in F$ and $x^k\rightarrow x^*$ on $E_0\cap E_1$. Moreover, the following hold on $E_0\cap E_1$.
    \begin{enumerate}[label=\textup{\textrm{(\roman*)}}]
        \item If $1/{\phi_{x^*(\omega)}}$ is integrable near $0$, then the sequence $\{x^k(\omega)\}_{k\ge 0}$ terminates in finitely many iterations.
        \item \label{cor:localEB-ii} If $1/{\phi_{x^*(\omega)}}$ is not integrable near $0$, then there exists a positive integer $K_4(\omega)$ such that for all $k\geq K_4(\omega)$,
        \[
        h(x^{k}(\omega))-h^\infty(\omega) \leq \Psi_{x^*(\omega)}^{-1} \left( \revise{\frac{\cona \conb}{2}} \revise{A_k} \right)
        \]
        and 
        \[
        \norm{x^k-x^*} \leq \revise{\frac{\sqrt{2}}{\cona\sqrt{\conb}}}\Phi_{x^*(\omega)} \left( \Psi_{x^*(\omega)}^{-1} \left( \revise{\frac{\cona \conb}{2}} \revise{A_k} \right) \right).
        \]
    \end{enumerate}
\end{theorem}

\begin{proof}
    For any $\omega\in E_0\cap E_1$, denote by $\mathcal{A}(\omega)$ the set of limit points of $\{x^k(\omega)\}_{k\ge 0}$.
    Then, $\mathcal{A}(\omega)$ is compact,   $\mathcal{A}(\omega) \subseteq F$ by the continuity of $g$,  and $h(x) = h^\infty(\omega)\in h(F)$ for all $x\in \mathcal{A}(\omega)$. By Lemma~\ref{lem:uniformization}, there exist positive constants $\delta_{\mathcal{A}(\omega)}$, $\eta_{\mathcal{A}(\omega)}$ and a finite set $\{\bar{x}_j\}_{j=1,\dots,l}\subseteq {\mathcal{A}(\omega)}$ such that
        \[
        \phi_{\mathcal{A}(\omega)}(h(x)-h^\infty(\omega)) \leq g(x)
        \]
    for any $x\in S$ satisfying $h^\infty(\omega)<h(x)<h^\infty(\omega)+\eta_{\mathcal{A}(\omega)}$ and $\dist(x,{\mathcal{A}(\omega)})<\delta_{\mathcal{A}(\omega)}$, where $\phi_{\mathcal{A}(\omega)} = \min_{j=1,\dots,l} \phi_{\bar{x}_j}$ is defined on $[0,\eta_{\mathcal{A}(\omega)})$. Since $1/\sqrt{\phi_{\bar{x}_j}}$ is integrable near $0$ for all $j \in [l]$, it follows that $1/\sqrt{\phi_{\mathcal{A}(\omega)}}$ is also integrable near $0$. Define
    \[
    N_1(\omega) = \min\{N: h^\infty(\omega)\leq h(x^k(\omega))<h^\infty(\omega)+\eta_{\mathcal{A}(\omega)},\ \dist(x^k(\omega),{\mathcal{A}(\omega)})<\delta_{\mathcal{A}(\omega)} \ \forall k\geq N\}.
    \]
    Then, $N_1(\omega)<\infty$. For any $k\geq N_1(\omega)$, by Lemma~\ref{lem:rate-preliminary}\ref{rate-preliminary1} and Lemma~\ref{lem:pre-stochastic}, we have
    \[
    \Psi_{\mathcal{A}(\omega)}(h(x^{k+1}(\omega))-h^\infty(\omega))\geq \Psi_{\mathcal{A}(\omega)}(h(x^{k}(\omega))-h^\infty(\omega)) + \revise{\cona}\alpha_k z_k(\omega),
    \]
    whenever $h(x^{k}(\omega))>h^\infty(\omega)$. Since $\alpha_k z_k$ is not summable on $E_0$, if $1/\phi_{\mathcal{A}(\omega)}$ is integrable near $0$, then the algorithm terminates in finitely many iterations. If $1/\phi_{\mathcal{A}(\omega)}$ is not integrable near $0$, following similar arguments as in the proof of Lemma~\ref{lem:random-convergence} but replacing the global error bound in Assumption~\ref{assmp:Global-EB} by the uniformized error bound in Lemma~\ref{lem:uniformization}, we obtain 
    \[
    h(x^{k}(\omega))-h^\infty(\omega) \leq \Psi^{-1}_{\mathcal{A}(\omega)} \!\! \left(\! \Psi_{\mathcal{A}(\omega)}(h(x^{N_1(\omega)}(\omega))-h^\infty(\omega)) + \revise{\cona}\sum_{i = N_1(\omega)}^{k-1} \alpha_i z_i(\omega) \!\right) \ \forall k>N_1(\omega).
    \]
    Since $\Psi$ is non-negative and $\omega\in E_0$, there exists a integer $N_2(\omega) > N_1(\omega) $ such that
    \[
    \Psi_{\mathcal{A}(\omega)}(h(x^{N_1(\omega)}(\omega))-h^\infty(\omega)) + \revise{\cona}\sum_{i = N_1 (\omega) }^{k-1} \alpha_i z_i \geq \revise{\frac{3\cona\conb}{4}} \revise{A_k}\geq \revise{\frac{\cona \conb}{2}} \revise{A_{k+1}} \quad\forall k\geq N_2(\omega).
    \]
    Thus, for all $k\geq N_2(\omega)$, we have
    \[
    h(x^{k}(\omega))-h^\infty(\omega) \leq \Psi_{\mathcal{A}(\omega)}^{-1} \left( \revise{\frac{3\cona \conb }{4}} \revise{A_k} \right)\leq \Psi_{\mathcal{A}(\omega)}^{-1} \left( \revise{\frac{\cona \conb}{2}} \revise{A_{k+1}} \right) .
    \]
    Following similar arguments as in the proof of Proposition~\ref{prop:iter-global-as} but replacing the global error bound in Assumption~\ref{assmp:Global-EB} by the uniformized error bound in Lemma~\ref{lem:uniformization}, $\{x^k(\omega)\}_{k\ge 0}$ converges to some $x^*(\omega)\in F$ such that 
    \[
    \norm{x^k(\omega)-x^*(\omega)} \leq \revise{\frac{\sqrt{2}}{\cona\sqrt{\conb}}} \Phi_{\mathcal{A}(\omega)} \left( \Psi_{\mathcal{A}(\omega)}^{-1} \left( \revise{\frac{\cona \conb }{2}} \revise{A_k} \right) \right) \quad \forall k \ge N_2 (\omega).
    \]
    In particular, $\mathcal{A}(\omega) = \{x^*(\omega)\}$. Substituting $\mathcal{A}(\omega)$ with $x^*(\omega)$ in the previous argument and setting $K_4(\omega) = N_2(\omega)$ yields the desired inequality. Finally, extending the definition of $x^* $ by setting $x^*(\omega) = 0$ for $\omega\not\in E_0\cap E_1$ completes the proof.
\end{proof}

The following result is immediate from Theorem~\ref{thm:local-EB} and Lemma~\ref{lem:g^k_vanishes}.

\begin{corollary}\label{cor:localEB}
    Suppose that Assumptions~\ref{assmp:basic},~\ref{assmp:step-size} and~\ref{assmp:local-EB} hold with a compact $ \B$, \revise{that $\Gamma\circ g$ is uniformly continuous on $S$ for some strictly increasing function $\Gamma:\R_+\rightarrow \R$}, and that $ 1/\sqrt{\phi_{\bar{x}}} $ is integrable near $0$ for all $ \bar{x} \in F $. Then, there exists a random variable $x^*$ such that $x^*\in F$ and $x^k\rightarrow x^*$ almost surely. Moreover, $\mathbb{P}(E_0) = 1$ and the following hold on $E_0$.
    \begin{enumerate}[label=\textup{\textrm{(\roman*)}}]
        \item If $1/{\phi_{x^*(\omega)}}$ is integrable near $0$, then the sequence $\{x^k(\omega)\}_{k\ge 0}$ terminates in finitely many iterations.
        \item If $1/{\phi_{x^*(\omega)}}$ is not integrable near $0$, then there exists a positive integer $K_4(\omega)$ such that for all $k\geq K_4(\omega)$,
        \[
        h(x^{k}(\omega))-h^\infty(\omega) \leq \Psi_{x^*(\omega)}^{-1} \left( \revise{\frac{\cona \conb}{2}} \revise{A_k} \right)
        \]
        and 
        \[
        \norm{x^k (\omega) - x^*(\omega) } \leq \revise{\frac{\sqrt{2}}{\cona\sqrt{\conb}}} \Phi_{x^*(\omega)} \left( \Psi_{x^*(\omega)}^{-1} \left( \revise{\frac{\cona \conb}{2}} \revise{A_k} \right) \right) .
        \]
    \end{enumerate}
\end{corollary}
\begin{proof}
    Note that $E_1 = \Omega$ since $\B$ is bounded and that $\mathbb{P}(E_0) = 1$ by Lemma~\ref{lem:g^k_vanishes} and Proposition~\ref{prop:as-bound}. The conclusions then follow directly from Theorem~\ref{thm:local-EB}.
\end{proof}

\section{Applications}\label{sec:applications}

\subsection{\revise{Generalized randomized subspace descent}}
\label{section_matrix_sketched_gd}
\revise{

We apply our framework to generalized randomized subspace descent (GRSD) for smooth unconstrained minimization problems~\eqref{pro_min}. GRSD encompasses several popular randomized algorithms, including randomized subspace descent (RSD)~\cite{frongillo2015convergence}, randomized fast subspace descent (RFASD)~\cite{chen2020randomized}, stochastic subspace descent (SSD)~\cite{kozak2021stochastic}, randomized block coordinate descent (RBCD), and randomized coordinate descent (RCD).

Let $\{B_k\}_{k\geq0}$ be a sequence of random symmetric positive semidefinite matrices such that $B_k$ is $\cF_{k+1}$-measurable. The GRSD update reads
\begin{equation}\label{eq:matrix-sketched-GD}
    x^{k+1}
    =
    x^k-\alpha_k B_k\nabla f(x^k)
    \qquad \forall k\geq0. 
\end{equation}
Let \(\mathscr B\subseteq\mathbb S_+^d\) denote the collection of possible realizations of $B_k$. We say that \(f\) is upper \(L_B\)-smooth along \(B\in\mathscr B\), for some \(L_B>0\), if
\begin{equation*}\label{eq:matrix-directional-smoothness}
    f(x+Bu)
    \leq
    f(x)
    +
    \langle\nabla f(x),Bu\rangle
    +
    \frac{L_B}{2}\langle u,Bu\rangle
    \qquad
    \forall x,u\in\mathbb R^d.
\end{equation*}
By replacing each $B\in\mathscr B$ with $B/L_B$, we may assume without loss of generality that $f$ is upper $1$-smooth along every $B\in\mathscr B$, i.e.,
\begin{equation}\label{eq:normalized-matrix-directional-smoothness}
    f(x+Bu)
    \leq
    f(x)
    +
    \langle\nabla f(x),Bu\rangle
    +
    \frac{1}{2}\langle u,Bu\rangle
    \qquad
    \forall x,u\in\mathbb R^d.
\end{equation}
Taking $B=B_k$ and
$u=-\alpha_k\nabla f(x^k)$ in
\eqref{eq:normalized-matrix-directional-smoothness} then yields
\begin{equation}\label{eq:matrix-sketch-descent}
\begin{aligned}
    f(x^{k+1})
    &\leq
    f(x^k)
    -
    \alpha_k
    \left(1-\frac{\alpha_k}{2}\right)
    \langle\nabla f(x^k),B_k\nabla f(x^k)\rangle .
\end{aligned}
\end{equation}
Consequently, it suffices to impose the deterministic stepsize condition
\begin{equation}\label{eq:matrix-sketched-stepsize}
    0<\alpha_k\leq\bar\alpha<2
    \qquad
    \forall k\geq0.
\end{equation}
For this application, we set
\[
    \B
    :=
    \{x\in\mathbb R^d:f(x)\leq f(x^0)\}, \qquad h(x):= f(x), \qquad g(x):= \|\nabla f(x)\|^2.
\]
We further suppose that there exist constants $M,\nu>0$ such that
\begin{equation}\label{eq:matrix-sketch-assumption}
    0\preceq B_k\preceq MI,
    \qquad
    \Exp[B_k^2\mid\cF_k]\succeq\nu I
    \quad\as
    \qquad
    \forall k\geq0.
\end{equation}

We now present several concrete examples of GRSD methods and explain how the required conditions hold.

\begin{example}
For the GRSD method~\eqref{eq:matrix-sketched-GD}, the support of $B_k$ could either be discrete or continuous.
Under the discrete support setting, suppose that $B_k$ takes values in the finite
family $\mathscr B=\{\mathsf B_i\}_{i=1}^m$.
This setting includes the randomized fast subspace descent (RFASD) method
of~\cite{chen2020randomized}, where
\[
    \mathsf B_i
    =
    \frac{1}{L_{A,i}}I_iA_i^{-1}I_i^\top.
\]
Here, $I_i$ embeds the selected subspace into the ambient space,
$I_i^\top$ is its adjoint, and $A_i^{-1}$ is the local
preconditioner. The subspace-dependent normalization $1/L_{A,i}$ is
absorbed into $\mathsf B_i$, so the RFASD update is recovered from
\eqref{eq:matrix-sketched-GD} by taking $\alpha_k=1$. Because the support is finite, the upper bound in
\eqref{eq:matrix-sketch-assumption} holds automatically with $M=\max_{i\in[m]}\|\mathsf B_i\|$. Moreover, suppose that every $B_k$ is sampled with uniformly
positive conditional probability:
\[
    \mathbb P(B_k=\mathsf B_i\mid\cF_k)\geq\rho>0
    \qquad
    \forall i\in[m],\ k\geq0,
\]
and that their ranges collectively span the ambient space.
Then
\[
    \Exp[B_k^2\mid\cF_k]
    =
    \sum_{i=1}^m
    \mathbb P(B_k=\mathsf B_i\mid\cF_k)\mathsf B_i^2
    \succeq
    \rho\sum_{i=1}^m\mathsf B_i^2
    \succeq
    \nu I,
\]
where $\nu:= \rho\lambda_{\min}\!\left(\sum_{i=1}^m\mathsf B_i^2\right)>0$ provided that $\sum_{i=1}^m\range(\mathsf B_i) = \R^d$. 

Many classical randomized subspace methods correspond to the
special case
\[
    \mathsf B_i=\frac{1}{L_i}P_{V_i},
\]
where $P_{V_i}$ is the orthogonal projector onto $V_i$. A finite family
of arbitrary, possibly overlapping, subspaces gives the projection-based
randomized subspace descent method of~\cite{frongillo2015convergence}.
Choosing the $V_i$ to be coordinate blocks recovers RBCD,
while taking $V_i=\mathrm{span}\{e_i\}$ recovers RCD. 

A particularly interesting example under the continuous support setting is the
stochastic subspace descent method of~\cite{kozak2021stochastic}:
\[
    x^{k+1}
    =
    x^k-\beta P_kP_k^\top\nabla f(x^k),
\]
where $P_k\in\mathbb R^{d\times\ell}$ satisfies $P_k^\top P_k=\frac{d}{\ell}I$ and $\Exp[P_kP_k^\top]=I$.
If $\nabla f$ is globally $\lambda$-Lipschitz continuous, this method is
recovered by setting $B_k=\frac{\ell}{\lambda d}P_kP_k^\top$ and $\alpha_k=\frac{\lambda d}{\ell}\beta$.
Since $(P_kP_k^\top)^2=(d/\ell)P_kP_k^\top$, it follows that $0\preceq B_k\preceq\frac{1}{\lambda}I$. If $P_k$ is independent of $\cF_k$, as in the i.i.d.\ setting of
\cite{kozak2021stochastic}, then
\[
    \Exp[B_k^2\mid\cF_k]
    =
    \frac{\ell^2}{\lambda^2d^2}
    \Exp[(P_kP_k^\top)^2]
    =
    \frac{\ell}{\lambda^2d}
    \Exp[P_kP_k^\top]
    =
    \frac{\ell}{\lambda^2d}I.
\]
Hence,~\eqref{eq:matrix-sketch-assumption} holds with $M=\frac{1}{\lambda}$ and $\nu=\frac{\ell}{\lambda^2d}$.
Scaled Haar-distributed orthonormal frames provide a standard
sampling scheme satisfying these identities. More generally, allowing
$H_k$ in $B_k=U_kH_kU_k^\top$ to be non-scalar gives continuously
sampled preconditioned subspace descent.
\end{example}

\begin{proposition}\label{prop:matrix-sketched-GD}
    Consider the GRSD method~\eqref{eq:matrix-sketched-GD}. Suppose that conditions~\eqref{eq:normalized-matrix-directional-smoothness},~\eqref{eq:matrix-sketched-stepsize} and~\eqref{eq:matrix-sketch-assumption} hold. Then Assumption~\ref{assmp:basic} holds on $\B$ with $\cona=(1-\frac{\bar\alpha}{2})/M$, $\conb=\nu$, $\conc=M^2$, and the same deterministic stepsizes $\{\alpha_k\}_{k\geq0}$ as in~\eqref{eq:matrix-sketched-GD}.
\end{proposition}

\begin{proof}
    Let $s^k:=x^{k+1}-x^k=-\alpha_k B_k\nabla f(x^k)$.
    Since $0\preceq B_k\preceq MI$, we have
    $B_k^2\preceq MB_k$. Consequently,
    \[
        \|s^k\|^2
        =
        \alpha_k^2
        \left\langle
            \nabla f(x^k),
            B_k^2\nabla f(x^k)
        \right\rangle\\
        \leq
        M\alpha_k^2
        \left\langle
            \nabla f(x^k),
            B_k\nabla f(x^k)
        \right\rangle.
    \]
    Combining this inequality with inequality~\eqref{eq:matrix-sketch-descent} and using $\alpha_k\leq\bar\alpha$ yields
    \[
        f(x^{k+1})
        \leq
        f(x^k)
        -
        \frac{1-\frac{\bar\alpha}{2}}{M\alpha_k}
        \|x^{k+1}-x^k\|^2.
    \]
    Thus,~\eqref{eq:basic_1} holds with $\cona=\frac{1-\frac{\bar\alpha}{2}}{M}$. In particular, $\{f(x^k)\}_{k\geq0}$ is non-increasing, and hence $x^k\in\B$ for all $k\geq0$. Next, since $x^k$ and $\nabla f(x^k)$ are $\cF_k$-measurable, condition~\eqref{eq:matrix-sketch-assumption} gives
    \[
    \begin{aligned}
        \Exp\!\left[
            \|x^{k+1}-x^k\|^2
            \,\middle|\,
            \cF_k
        \right]
        =
        \alpha_k^2
        \left\langle
            \nabla f(x^k),
            \Exp[B_k^2\mid\cF_k]
            \nabla f(x^k)
        \right\rangle
        \geq
        \nu\alpha_k^2
        \|\nabla f(x^k)\|^2
        =
        \nu\alpha_k^2g(x^k).
    \end{aligned}
    \]
    Therefore,~\eqref{eq:expbound} holds with $\conb=\nu$. Finally, the bound $B_k\preceq MI$ implies
    \[
        \|x^{k+1}-x^k\|^2
        =
        \alpha_k^2
        \|B_k\nabla f(x^k)\|^2
        \leq
        M^2\alpha_k^2
        \|\nabla f(x^k)\|^2
        =
        M^2\alpha_k^2g(x^k),
    \]
    which verifies~\eqref{eq:basic_4} with $\conc=M^2$. The proof is complete.
\end{proof}

We next analyze the convergence rates of GRSD by relating the Kurdyka--Łojasiewicz property to our UEB conditions. We first recall the KL property.

\begin{definition}[Kurdyka--Łojasiewicz property]
\label{def:local-KL}
    A function $f\in C^1(\mathbb R^d)$ is said to satisfy the
    Kurdyka--Łojasiewicz\/ (KL) property at
    $\bar{x}\in\mathbb R^d$ if there exist
    $\delta_{\bar{x}},\eta_{\bar{x}}\in(0,+\infty]$ and a continuous,
    concave function $\varphi_{\bar{x}}:[0,\eta_{\bar{x}})\to \R_+$ such that $\varphi_{\bar{x}}(0)=0$, $\varphi_{\bar{x}}$ is continuously
    differentiable on $(0,\eta_{\bar{x}})$, and
    $\varphi_{\bar{x}}'>0$ on $(0,\eta_{\bar{x}})$, with
    \begin{equation}\label{eq:local-KL}
        \varphi_{\bar{x}}'(f(x)-f(\bar{x}))
        \|\nabla f(x)\|
        \geq 1
    \end{equation}
    whenever $x\in\mathbb B(\bar{x},\delta_{\bar{x}})$ and $0<f(x)-f(\bar{x})<\eta_{\bar{x}}$. The function $\varphi_{\bar{x}}$ is called a desingularization function
    of $f$ at $\bar{x}$. We say that $f$ has KL exponent
    $\kappa\in [0,1)$ at $\bar{x}$ if one may choose
    $\varphi_{\bar{x}}(t)=\bar c\,t^{1-\kappa}$ for some $\bar c>0$.
\end{definition}

The following global result is the smooth finite-valued specialization of
the globalization theorem for convex definable functions
in~\cite{Bolte2017Errora}.

\begin{proposition}[Global KL property for convex definable functions]
\label{prop:global-KL-convex-definable}
    Let $f\in C^1(\mathbb R^d)$ be convex, definable, and coercive. Set $X^*:=\argmin f$ and $f^*:=\min f$. Then there exists a
    continuous, concave function
    $\varphi:\R_+\to\R_+$ such that $\varphi(0)=0$,
    $\varphi$ is continuously differentiable on $(0,+\infty)$, and
    $\varphi'>0$ on $(0,+\infty)$, with
    \begin{equation}\label{eq:global-KL-convex-definable}
        \varphi'(f(x)-f^*)\|\nabla f(x)\|
        \geq 1
        \qquad
        \forall x\in\mathbb R^d\setminus X^*.
    \end{equation}
\end{proposition}

For general, possibly nonconvex, definable functions, one retains the local
KL property at every point; see~\cite{Bolte2007Clarke}.

\begin{proposition}[Local KL property for definable functions]
\label{prop:local-KL-definable}
    Let $f\in C^1(\mathbb R^d)$ be definable. Then $f$ satisfies the KL
    property of Definition~\ref{def:local-KL} at every
    $\bar{x}\in\mathbb R^d$.
\end{proposition}

The following proposition establishes the correspondence between KL inequalities and UEBs. The proof is straightforward and thus omitted.

\begin{proposition}[From KL inequalities to unified error bounds]
\label{prop:KL-to-UEB}
    The following statements hold.
    \begin{enumerate}[label=\textup{(\roman*)}]
        \item\label{prop:KL-to-UEB-global}
        Suppose that $f$ admits a global desingularization function
        $\varphi$ satisfying~\eqref{eq:global-KL-convex-definable}. Define
        $\phi:\mathbb R_+\to\mathbb R_+$ by
        \[
            \phi(0):=0,
            \qquad
            \phi(t):=\frac{1}{(\varphi'(t))^2}
            \quad\text{for }t>0.
        \]
        Then Assumption~\ref{assmp:Global-EB} holds with the modulus $\phi$.

        \item\label{prop:KL-to-UEB-local}
        Suppose that $f$ satisfies the local KL property at every critical point
        $\bar{x}$ with desingularization function
        $\varphi_{\bar{x}}$. Define
        $\phi_{\bar{x}}:\mathbb R_+\to\mathbb R_+$ by $\phi_{\bar{x}}(0):=0$ and
        \[
            \phi_{\bar{x}}(t)
            :=
            \frac{1}{(\varphi_{\bar{x}}'(t))^2}
            \quad
            \text{for }t\in(0,\eta_{\bar{x}}).
        \]
        Then Assumption~\ref{assmp:local-EB} holds with modulus $\phi_{\bar{x}}$ at any critical point $\bar{x}\in \B$.
    \end{enumerate}
\end{proposition}


In particular, Proposition~\ref{prop:global-KL-convex-definable} yields a
GUEB for every convex definable coercive function, whereas Proposition~\ref{prop:local-KL-definable} yields an LUEB for every general definable function. By the definitions of $\phi$ and $\phi_{\bar{x}}$, we have $1/\sqrt{\phi} = \varphi'$ and $1/ \sqrt{\phi_{\bar{x}}} = \varphi_{\bar{x}}'$,
so both functions are integrable near zero because the corresponding desingularization functions are bounded around zero. In contrast, the following theorem shows that the squared derivative of a KL desingularization function is not integrable near zero under the stated regularity assumptions. Consequently, whenever the theorem applies, both the global and local inverse smoothing functions diverge to $+\infty$ as $t\downarrow0$. Recall that \(\nabla f\) is said to be locally Lipschitz continuous if, for every \(\bar{x}\in\R^d\), there exist a neighborhood \(U\) of \(\bar{x}\) and a constant \(L>0\) such that
\[
    \|\nabla f(x)-\nabla f(y)\|
    \leq
    L\|x-y\|
    \qquad
    \forall x,y\in U.
\]


\begin{theorem}\label{Th_KL_general}
Suppose $f\in C^1(\R^d)$ is a definable function with a locally Lipschitz continuous gradient $\nabla f$, and let $\bar{x}$ be a critical point of $f$ that is not a local maximizer. If $f$ satisfies the KL property at $\bar{x}$ with a desingularization function $\varphi$, then, after possibly shrinking the domain of $\varphi$, the function
\[
    \phi_{\bar{x}}(t):=\frac{1}{(\varphi'(t))^2}
\]
satisfies $\phi_{\bar{x}}(t)=O(t)$ as  $t\downarrow0$.
Consequently, $1/\phi_{\bar x} = (\varphi')^2$ is not integrable near $0$. In particular, $f$ cannot satisfy the KL property with an exponent $\kappa \in [0,1/2)$ at $\bar{x}$.
\end{theorem}

\begin{proof}
Without loss of generality, we assume $\bar{x}=0$ and $f(0)=0$. Let $\eta, r>0$ be such that
\begin{equation}\label{eq4-1-gen}
    \varphi'(f(x))\|\nabla f(x)\|\geq1
    \qquad
    \forall x\in\mathbb B(0,r)
    \text{ with }0<f(x)<\eta.
\end{equation}
For each $t\in[0,r)$, define
\[
    M(t):=\max_{\|x\|\le t} f(x),
    \qquad
    H(t):=\argmax\{f(x):\|x\|\le t\}.
\]
Because \(0\) is not a local maximizer, \(M(t)>0\) for every sufficiently small \(t>0\), while continuity of \(f\) gives \(M(t)\to0\) as \(t\downarrow0\). Meanwhile, the function $M$ is definable. With
\[
    W(x,t):=-f(x)+\iota_{\{\|x\|\le t\}}(x),
    \qquad
    Q(t):=\inf_{x\in\R^d}W(x,t),
\]
we have $M(t)=-Q(t)$, and the definability follows from the standard definability of marginal functions; see, e.g., \cite[Remark 2 and Lemma 1]{Kurdyka1998OnGO}. Moreover, $\operatorname{gph}(H)=(W-Q)^{-1}(0)$ is definable by \cite[Chapter 1, Lemma 2.3(iii)]{van1998tame}. Hence, by definable selection~\cite[Theorem~4.5]{Dries1996GeometricCA}, we may select a definable curve \(z:(0,r)\to\R^d\) such that \(z(t)\in H(t)\) for every \(t\in(0,r)\). Applying the monotonicity lemma~\cite[Theorem~4.1]{Dries1996GeometricCA} to the finitely many coordinate functions of \(z\), and taking a common subinterval, we may assume that \(z\in C^1((0,r))\). After reducing \(r>0\) further if necessary, the same lemma ensures that \(M\) is \(C^1\) and strictly increasing, with \(M(t)<\eta\) on \((0,r)\).

We first claim that $\|z(t)\|=t$ and $\nabla f(z(t))\neq 0$ for all $t\in(0,r)$.
Since $M(t)\in (0,\eta)$ for $t\in (0,r)$, \eqref{eq4-1-gen} gives $\nabla f(z(t))\neq 0$. If $z(t)$ belonged to the interior of $\overline{\mathbb B}(0,t)$, then the first-order necessary condition for maximizing $f$ over $\overline{\mathbb B}(0,t)$ would give $\nabla f(z(t))=0$, a contradiction. Thus $\|z(t)\|=t$. By the optimality condition for maximizing $f$ over $\overline{\mathbb B}(0,t)$, there exists $\lambda(t)>0$ such that $\nabla f(z(t))=\lambda(t)z(t)$,
which implies that $\frac{\nabla f(z(t))}{\|\nabla f(z(t))\|} = \frac{z(t)}{\|z(t)\|}$. Since $\|z(t)\|=t$, differentiating $\|z(t)\|^2=t^2$ gives
\[
    \left\langle \frac{z(t)}{\|z(t)\|},z'(t)\right\rangle=1.
\]
Therefore,
\begin{equation}\label{eq:Mprime-gradient}
    M'(t)=\langle \nabla f(z(t)),z'(t)\rangle = \|\nabla f(z(t))\|,
    \qquad \forall t\in(0,r).
\end{equation}
Next, since $\nabla f$ is locally Lipschitz continuous and $\nabla f(0)=0$, after reducing $r$ if necessary, there exists $L>0$ such that $\nabla f$ is $L$-Lipschitz continuous on $\mathbb B(0,r)$. Hence, by the descent lemma,
\[
    M(t)=f(z(t))
    \le
    f(0)+\frac{L}{2}\|z(t)\|^2
    =
    \frac{L}{2}t^2.
\]
Define $u(t):=\sqrt{M(t)}$. Then $u$ is positive, definable, and $C^1$ on $(0,r)$, satisfying $0<u(t)\le \sqrt{L/2}\,t$.
Since \(u\) is strictly increasing on \((0,r)\), we have \(u'\geq0\) on \((0,r)\). By the monotonicity lemma, after reducing \(r>0\) if necessary, we may assume that \(u'\) is monotone and that \(2t\) remains in its original domain for every \(t\in(0,r)\).

Suppose first that \(u'\) is non-decreasing. Then
\[
    u(2t)-u(t)
    =
    \int_t^{2t}u'(s)\,ds
    \geq t u'(t).
\]
Using \(u(2t)\leq 2\sqrt{L/2}\,t\), we obtain $u'(t)\leq \sqrt{2L}$. If \(u'\) is non-increasing, then
\[
    u(t)-u(t/2)
    =
    \int_{t/2}^{t}u'(s)\,ds
    \geq\frac{t}{2}u'(t),
\]
and \(u(t)\leq\sqrt{L/2}\,t\) again gives $u'(t)\leq \sqrt{2L}$. Therefore, in either case, $0\leq u'(t)\leq\sqrt{2L}$ for any $t\in(0,r)$.
Consequently, $M'(t)=2u(t)u'(t)\le 2\sqrt{2L} \sqrt{M(t)}$.
Combining this with \eqref{eq:Mprime-gradient}, for any $t\in (0,r)$, we have
\begin{equation}\label{eq:gradient-square-value}
    \|\nabla f(z(t))\|^2
    =
    (M'(t))^2
    \le
    8L M(t) = 
    8L f(z(t)).
\end{equation}

Since \(M\) is continuous and strictly increasing with \(M(t)\to0\), for every sufficiently small \(s>0\), there exists \(t_s>0\) such that $M(t_s)=f(z(t_s))=s$.
Applying \eqref{eq4-1-gen} at \(z(t_s)\) and then using
\eqref{eq:gradient-square-value}, we obtain
\[
    \frac{1}{(\varphi'(s))^2}
    \leq
    \|\nabla f(z(t_s))\|^2
    \leq 8Ls.
\]
Therefore, for sufficiently small $s>0$,
\[
    \phi_{\bar{x}}(s)
    :=
    \frac{1}{(\varphi'(s))^2}
    \le 8Ls.
\]
Consequently, $1/\phi_{\bar{x}} = (\varphi'(s))^2\geq\frac{1}{8Ls}$ is not integrable near \(0\). Finally, if
\(\varphi(s)=\bar c\,s^{1-\kappa}\), then
\[
    \frac{1}{(\varphi'(s))^2}
    =
    \frac{s^{2\kappa}}
         {\bar c^2(1-\kappa)^2}.
\]
The estimate \(1/(\varphi'(s))^2=O(s)\) requires \(2\kappa\geq1\), and hence
\(\kappa\geq1/2\).
\end{proof}

Some remarks are in order.
First, although both constructions use definable selection, our curve $z(\cdot)$ is different from the classical talweg, see, e.g., \cite{bolte2008characterizations}. A talweg selects the points where the gradient norm is approximately minimal. In contrast, our curve traces the maximal growth of the function value with respect to the distance from \(\bar x\). Thus, the two constructions are based on distinct—and, in a sense, opposite—variational principles. Second, it is proved in \cite[Theorem~7]{bento2025convergence} that the sum of a globally Lipschitz smooth function and a convex function cannot have a KL exponent $\kappa\in (0,\frac{1}{2})$ at a local minimizer in the interior of its domain. Our Theorem~\ref{Th_KL_general} does not require such a sum structure but considers a general locally Lipschitz smooth function. It is advantageous in two senses: it covers general desingularization functions and does not require the target point to be a local minimizer. The former advantage allows us to work with general definable functions but not just semi-algebraic or subanalytic functions; while the latter accommodates nonconvex functions, where limit points are not guaranteed to be local minimizers.

\begin{theorem}[Global rates for convex definable objectives]
\label{thm:matrix-sketched-GD-global}
    Consider the GRSD method~\eqref{eq:matrix-sketched-GD} with stepsizes $\alpha_k$ satisfying Assumption~\ref{assmp:step-size}. Suppose that $f$ is coercive, convex, definable, and has a locally Lipschitz continuous gradient, and that conditions~\eqref{eq:normalized-matrix-directional-smoothness},~\eqref{eq:matrix-sketched-stepsize} and~\eqref{eq:matrix-sketch-assumption} hold. Let $\varphi$ be the global KL desingularization function given by
    Proposition~\ref{prop:global-KL-convex-definable}, and define $\phi(0):=0$,
    \[
        \phi(t):=\frac{1}{(\varphi'(t))^2}
        \quad\text{and}\quad
        \psi(t):=\int_t^{t_0}(\varphi'(s))^2\,\mathrm ds\qquad\forall t>0.
    \]
    Then $1/\sqrt{\phi}$ is integrable near zero, whereas $1/\phi$ is not
    integrable near zero. Consequently, all conclusions of
    Theorem~\ref{th_general_exp_rate1} hold
    with functions $\phi_{\B}=\phi$, $\Psi_{\B}=\psi$, $\Phi_{\B}=\varphi$, and constants $\cona=(1-\frac{\bar\alpha}{2})/M$, $\conb=\nu$, $\conc=M^2$.
\end{theorem}

\begin{proof}
    Coercivity and descent imply that the initial level set $\B$ is compact.
    Propositions~\ref{prop:global-KL-convex-definable}
    and~\ref{prop:KL-to-UEB}\ref{prop:KL-to-UEB-global} yield the required
    GUEB with modulus $\phi=1/(\varphi')^2$. Since $1/\sqrt{\phi}=\varphi'$, the function $1/\sqrt{\phi}$ is integrable near zero and $\Phi_{\B}=\varphi$. To establish the opposite conclusion for
    $1/\phi$, choose $\bar{x}\in\operatorname{bdry}(\argmin f)$. Such a point exists because $\argmin f$ is nonempty and compact.
    Moreover, $\bar{x}$ is not a local maximizer: otherwise, since
    $f(\bar{x})=f^*$, a neighborhood of $\bar{x}$ would be contained in
    $\argmin f$, contradicting
    $\bar{x}\in\operatorname{bdry}(\argmin f)$. The global function
    $\varphi$ is also a local KL desingularization function at $\bar{x}$.
    Hence, Theorem~\ref{Th_KL_general} gives that ${1}/{\phi}=(\varphi')^2$ is not integrable near $0$, and therefore
    $\Psi_{\B}=\psi$. The result now follows directly from
    Theorem~\ref{th_general_exp_rate1}.
\end{proof}

\begin{theorem}[Local convergence and rates for definable objectives]
\label{thm:matrix-sketched-GD-local}
    Consider the GRSD method~\eqref{eq:matrix-sketched-GD} with stepsizes $\alpha_k$ satisfying Assumption~\ref{assmp:step-size}. Suppose that \(f\) is coercive, definable and has a locally Lipschitz continuous gradient, and that conditions~\eqref{eq:normalized-matrix-directional-smoothness},
    \eqref{eq:matrix-sketched-stepsize}, and
    \eqref{eq:matrix-sketch-assumption} hold. For each critical point \(\bar{x}\), let
    \(\varphi_{\bar{x}}:[0,\eta_{\bar{x}})\to\mathbb R_+\) be a local KL
    desingularization function given by
    Proposition~\ref{prop:local-KL-definable}. Fix
    \(t_{0,\bar{x}}\in(0,\eta_{\bar{x}})\), and define $\phi_{\bar{x}}(0):=0$,
    \[
        \phi_{\bar{x}}(t)
        :=
        \frac{1}{(\varphi_{\bar{x}}'(t))^2}
        \quad\text{and}\quad
        \psi_{\bar{x}}(t)
        :=
        \int_t^{t_{0,\bar{x}}}
        (\varphi_{\bar{x}}'(s))^2\,\mathrm ds\quad\forall t>0.
    \]
    Then \(1/\sqrt{\phi_{\bar{x}}}\) is integrable near zero for every
    critical point \(\bar{x}\). Moreover, if \(\bar{x}\) is not a local
    maximizer, then \(1/\phi_{\bar{x}}\) is not integrable near zero. Consequently, all conclusions of Corollary~\ref{cor:localEB} hold with functions $\Psi_{\bar{x}}=\psi_{\bar{x}}$, $\Phi_{\bar{x}}=\varphi_{\bar{x}}$, and constants $\cona=\frac{1-\frac{\bar\alpha}{2}}{M}$, $\conb=\nu$, $\conc=M^2$. In particular, the iterates converge almost surely to a critical point \(x^*\). On every such convergent path,\footnote{The dependence on the sample path $\omega\in \Omega$ is suppressed in case of no potential confusion.} either \(x^*\) is a local maximizer and the sequence terminates at \(x^*\) after finitely many iterations, or \(x^*\) is not a local maximizer and the rate estimates in Corollary~\ref{cor:localEB}\ref{cor:localEB-ii} hold.
\end{theorem}

\begin{proof}
    Coercivity makes the initial level set $\B$ compact, while
    \eqref{eq:matrix-sketch-descent} ensures that all iterates remain in
    $\B$. By
    \eqref{eq:normalized-matrix-directional-smoothness},
    \eqref{eq:matrix-sketched-stepsize}, and
    \eqref{eq:matrix-sketch-assumption},
    Proposition~\ref{prop:matrix-sketched-GD} verifies
    Assumption~\ref{assmp:basic} with the stated constants.
    Propositions~\ref{prop:local-KL-definable}
    and~\ref{prop:KL-to-UEB}\ref{prop:KL-to-UEB-local} then yield the
    required LUEB at every critical point. Note that $1/\sqrt{\phi_{\bar{x}}}=\varphi_{\bar{x}}'$ is integrable near zero and, by
    Definition~\ref{def:inverse-smoothing}, $\Phi_{\bar{x}}=\varphi_{\bar{x}}$ and $\Psi_{\bar{x}}=\psi_{\bar{x}}$.
    Corollary~\ref{cor:localEB} therefore applies and yields almost-sure
    convergence to a critical point $x^*$. Fix a sample path on which $x^k\to x^*$. If $x^*$ is not a local
    maximizer, Theorem~\ref{Th_KL_general} shows that ${1}/{\phi_{x^*}}=(\varphi_{x^*}')^2$ is not integrable near zero. Thus,
    Corollary~\ref{cor:localEB}\ref{cor:localEB-ii} gives the stated rates.
    If $x^*$ is a local maximizer, then, for all sufficiently large $k$,
    local maximality and $f(x^k)\downarrow f(x^*)$ imply
    $f(x^k)=f(x^*)$. The sufficient-descent inequality
    \eqref{eq:basic_1} then forces the sequence to be constant, and hence
    equal to $x^*$, from some finite index onward.
\end{proof}

\begin{example}[Rates under a power-type KL inequality]
\label{exam:matrix-sketched-GD-subanalytic}
    Suppose that the conditions of
    Theorem~\ref{thm:matrix-sketched-GD-local} hold and that \(f\) is
    subanalytic. Fix a sample path on which \(x^k\to x^*\). If \(x^*\) is a
    local maximizer, the sequence terminates at \(x^*\) after finitely many
    iterations. Otherwise, \(f\) has a KL exponent
    \(\kappa\in[1/2,1)\) at \(x^*\), with the desingularization function $\varphi_{x^*}(t)=\bar c\,t^{1-\kappa}$ for some \(\bar c>0\). Setting \(r:=\bar c(1-\kappa)\), one may take $\phi_{x^*}(t)=r^{-2}t^{2\kappa}$, $\Phi_{x^*}(t)=\bar c\,t^{1-\kappa}$,
    and
    \[
        \Psi_{x^*}(t)
        =
        \begin{cases}
            r^2\log(t_0/t),
            & \kappa=\frac12,\\[1mm]
            \displaystyle
            \frac{r^2}{2\kappa-1}
            \left(
                t^{-(2\kappa-1)}
                -
                t_0^{-(2\kappa-1)}
            \right),
            & \kappa\in(\frac12,1).
        \end{cases}
    \]
    Consequently, if \(\kappa=\frac12\), then
    \[
        f(x^k)-f(x^*)
        =
        O\!\left(\exp(-\Theta(A_k))\right)
        \quad\text{and}\quad
        \|x^k-x^*\|
        =
        O\!\left(\exp(-\Theta(A_k))\right).
    \]
    If \(\kappa\in(\frac12,1)\), then
    \[
        f(x^k)-f(x^*)
        =
        O\!\left(A_k^{-\frac{1}{2\kappa-1}}\right)\quad\text{and}\quad
        \|x^k-x^*\|
        =
        O\!\left(
            A_k^{-\frac{1-\kappa}{2\kappa-1}}
        \right).
    \]
    These estimates hold pathwise for all sufficiently large \(k\), and
    hence almost surely. These asymptotic rates coincide with those established for deterministic KL descent methods; see, for example, \cite[Theorems~4]{FrankelGarrigosPeypouquet2015Splitting} and \cite[Theorem 2]{bento2025convergence}.
\end{example}

}

\subsection{\revise{Randomized block proximal-gradient descent}}
\label{section:smooth-random-bpg}
\revise{
Consider the smooth composite optimization problem
\begin{equation}\label{eq:smooth-composite-min}
    \min_{x\in\mathbb R^d}
    H(x):=F(x)+ {R(x)},
\end{equation}
where $x=(x_1,\dots,x_m)\in\mathbb R^{d_1}\times\cdots\times\mathbb R^{d_m}$, \(F\in C^1(\mathbb R^d)\) has an accessible gradient, and
{\(R\in C^1(\mathbb R^d)\) is a possibly nonconvex and nonseparable regularizer. For each
\(i\in[m]\), let \(I_i:\mathbb R^{d_i}\to\mathbb R^d\) denote the canonical embedding into the \(i\)th block. For any
\(G\in C^1(\mathbb R^d)\), we denote its restricted gradient along the
\(i\)th block by $\nabla_iG(x):=I_i^\top\nabla G(x)\in\mathbb R^{d_i}$.} Given \(x\in\mathbb R^d\), a block \(i\in[m]\), and a stepsize
\(\gamma>0\), define
\[
    \widehat T_i(x;\gamma)
    :=
    \bigl(
        x_1,\dots,x_{i-1},
        \widehat T_i^b(x;\gamma),
        x_{i+1},\dots,x_m
    \bigr),
\]
where
\[
    \widehat T_i^b(x;\gamma)
    =
    \argmin_{u_i\in\mathbb R^{d_i}}
    \left\{
        {R(x+I_i(u_i-x_i))}
        +
        \langle\nabla_iF(x),u_i-x_i\rangle
        +
        \frac{1}{2\gamma}\|u_i-x_i\|^2
    \right\}.
\]
A randomized block proximal-gradient method takes the form
\begin{equation}\label{eq:smooth-random-bpg-raw}
    x^{k+1}
    =
    \widehat T_{i_k}(x^k;\gamma_k) 
    \qquad \forall k\geq0,
\end{equation}
where \(i_k\in[m]\) is the block sampled at iteration \(k\), and
\(\gamma_k>0\) is the corresponding blockwise stepsize. 
{Under the conditions introduced below}, the subproblem is strongly convex, and hence
\(\widehat T_i^b(x;\gamma)\) is unique. Moreover, the
nonexpansiveness of the proximal map and the continuity of \(\nabla F\)
imply that \(\widehat T_i(\cdot;\gamma)\) is continuous. When \(m=1\),
\eqref{eq:smooth-random-bpg-raw} reduces to proximal gradient descent
with a smooth proximal term and a varying stepsize.

We now introduce the assumptions and notation used in the
analysis. Let $\B
    :=
    \{x\in\mathbb R^d:H(x)\leq H(x^0)\}$,
and suppose that \(H\) is bounded from below on \(\B\). 

{\begin{assumption}[Blockwise regularity]
\label{assmp:smooth-bpg-smoothness}
For each \(i\in[m]\), there exist constants
\(L_i^F,L_i^R>0\) and \(\varrho_i\geq0\) such that the following hold.
\begin{enumerate}[label=\textup{\textrm{(\roman*)}}]
    \item \label{block-smooth-F} The function \(F\) is blockwise upper \(L_i^F\)-smooth:
    \begin{equation}\label{eq:block-smooth-F}
        F(x+I_i u_i)
        \leq
        F(x)
        +
        \langle\nabla_iF(x),u_i\rangle
        +
        \frac{L_i^F}{2}\|u_i\|^2
        \qquad
        \forall x\in\mathbb R^d,\ 
        u_i\in\mathbb R^{d_i}.
    \end{equation}

    \item The function \(R\) is blockwise \(\varrho_i\)-weakly convex; \ie,
    for every fixed \((x_j)_{j\neq i}\), the function
    \[
        u_i
        \longmapsto
        R(x_1,\ldots,x_{i-1},u_i,x_{i+1},\ldots,x_m)
        +
        \frac{\varrho_i}{2}\|u_i\|^2
    \]
    is convex.

    \item The restricted gradient \(\nabla_iR\) is
    \(L_i^R\)-Lipschitz continuous along the \(i\)th block on \(\B\):
    \[
        \|\nabla_iR(x+I_i u_i)-\nabla_iR(x)\|
        \leq
        L_i^R\|u_i\|\quad \forall x,x+I_i u_i\in\B.
    \]
\end{enumerate}
\end{assumption}}


{The upper smoothness condition (Assumption~\ref{assmp:smooth-bpg-smoothness}\ref{block-smooth-F}) is imposed globally for simplicity and can be replaced by local upper smoothness under coercivity using a refined argument.}

For the analysis, we define $\widehat L_i:=L_i^F+\varrho_i$ for $i\in [m]$. We also parameterize the blockwise stepsizes as $\gamma_k=\alpha_k / {\widehat L_i}$,
where \(\{\alpha_k\}_{k\geq0}\) is a deterministic sequence. Accordingly,
for \(\alpha>0\), define
\[
    T_i(x;\alpha)
    :=
    \widehat T_i\left(x;\frac{\alpha}{{\widehat L_i}}\right).
\]
Its updated block is given by
\begin{equation}\label{eq:smooth-bpg-map}
    T_i^b(x;\alpha)
    =
    \argmin_{u_i\in\mathbb R^{d_i}}
    \left\{
        {R(x+I_i(u_i-x_i))}
        +
        \langle\nabla_iF(x),u_i-x_i\rangle
        +
        \frac{{\widehat L_i}}{2\alpha}\|u_i-x_i\|^2
    \right\}.
\end{equation}
Thus, \(1/{\widehat L_i}\) is the block-dependent baseline stepsize, while
\(\alpha\) provides a common scaling across all blocks. With this
parameterization, the method becomes
\begin{equation}\label{eq:smooth-random-bpg}
    x^{k+1}
    =
    T_{i_k}(x^k;\alpha_k)
    \qquad \forall k\geq0.
\end{equation}
We assume that 
\begin{equation}\label{eq:smooth-bpg-stepsize}
    {
    0<\alpha_k\leq\bar\alpha
    <
    \min\left\{
        2,\,
        \min_{i\in[m]}
        \frac{\widehat L_i}{\varrho_i}
    \right\}
    \qquad
    \forall k\geq0,
    }
\end{equation}
{where \(\widehat L_i/\varrho_i:=+\infty\) whenever
\(\varrho_i=0\),} and that the sampling rule satisfies
\begin{equation}\label{eq:smooth-bpg-sampling}
    \mathbb P(i_k=i\mid\cF_k)\geq\rho>0
    \qquad
    \forall i\in[m],\ k\geq0.
\end{equation}
Finally, to place the method within our abstract framework, we define 
\begin{equation}\label{eq:smooth-bpg-constants}
\begin{aligned}
    \cona
    :=
    \left(1-\frac{\bar\alpha}{2}\right)
    \min_{i\in[m]}{\widehat L_i},\quad
    \conb
    :=
    \rho
    \min_{i\in[m]}
    \frac{1}{({\widehat L_i}+\bar\alpha L_i^R)^2},\quad
    \conc
    :=
    \max_{i\in[m]}
    \frac{1}{({\widehat L_i-\bar{\alpha}\varrho_i})^2}.
\end{aligned}
\end{equation}

\begin{proposition}\label{prop:smooth-random-bpg-basic}
    Suppose that Assumption~\ref{assmp:smooth-bpg-smoothness} and conditions~\eqref{eq:smooth-bpg-stepsize}and
    \eqref{eq:smooth-bpg-sampling} hold. Then the iterates generated by~\eqref{eq:smooth-random-bpg} satisfy Assumption~\ref{assmp:basic} on $\B$ with $h = H $ and $g = \|\nabla H \|^2$, the stepsize sequence $\{\alpha_k\}_{k\geq0}$, and
    the constants $c_1, c_2,c_3$ given in~\eqref{eq:smooth-bpg-constants}.
\end{proposition}

\begin{proof}
    Fix $x\in\B$, $i\in[m]$, and $\alpha\in(0,\bar \alpha]$. Set $z_i:=T_i^b(x;\alpha)$ and $d_i:=z_i-x_i$.
    Define
    \[
        Q_{i,\alpha}(u_i)
        :=
        {R\bigl(x+I_i(u_i-x_i)\bigr)}
        +
        \langle\nabla_iF(x),u_i-x_i\rangle
        +
        \frac{{\widehat L_i}}{2\alpha}\|u_i-x_i\|^2.
    \]
    Since $Q_{i,\alpha}$ is {$(\widehat L_i/\alpha-\varrho_i)$}-strongly convex and $z_i$ is
    its minimizer,
    \[
        Q_{i,\alpha}(x_i)
        \geq
        Q_{i,\alpha}(z_i)
        +
        {\frac{1}{2}
    \left(
        \frac{\widehat L_i}{\alpha}-\varrho_i
    \right)}\|d_i\|^2.
    \]
    Expanding this inequality and combining it with
    \eqref{eq:block-smooth-F} gives
    \[
    \begin{aligned}
        H(T_i(x;\alpha))
        \leq
        H(x)
        -
        \frac{{\widehat L_i}(1-\alpha/2)}{\alpha}
        \|T_i(x;\alpha)-x\|^2
        \leq
        H(x)
        -
        \frac{\cona}{\alpha}
        \|T_i(x;\alpha)-x\|^2.
    \end{aligned}
    \]
    Therefore,~\eqref{eq:basic_1} holds with the same stepsize sequence
    $\{\alpha_k\}_{k\geq0}$. In particular,
    $H(x^{k+1})\leq H(x^k)$, and hence $x^k\in\B$ for all $k\geq0$.

    {Denote $y:=T_i(x;\alpha)=x+I_i d_i$. The first-order optimality condition for
\(z_i=T_i^b(x;\alpha)\) is
\[
    0
    =
    \nabla_iF(x)
    +
    \nabla_iR(y)
    +
    \frac{\widehat L_i}{\alpha}d_i.
\]
Since \(\nabla_iH(x)=\nabla_iF(x)+\nabla_iR(x)\), it follows that
\begin{equation}\label{eq:smooth-bpg-optimality-relation}
    \nabla_iH(x)
    =
    -\frac{\widehat L_i}{\alpha}d_i
    +
    \nabla_iR(x)-\nabla_iR(y).
\end{equation}
Because \(x,y\in\B\), the blockwise Lipschitz continuity of
\(\nabla_iR\) gives
\[
    \|\nabla_iH(x)\|
    \leq
    \left(
        \frac{\widehat L_i}{\alpha}+L_i^R
    \right)\|d_i\|.
\]
On the other hand, blockwise \(\varrho_i\)-weak convexity gives
\[
    \langle
        \nabla_iR(y)-\nabla_iR(x),
        d_i
    \rangle
    \geq
    -\varrho_i\|d_i\|^2.
\]
Taking the inner product of
\eqref{eq:smooth-bpg-optimality-relation} with \(-d_i\) yields
\[
    -\langle\nabla_iH(x),d_i\rangle
    \geq
    \left(
        \frac{\widehat L_i}{\alpha}-\varrho_i
    \right)\|d_i\|^2.
\]
Consequently,
\[
    \|d_i\|
    \leq
    \frac{\alpha}{\widehat L_i-\alpha\varrho_i}
    \|\nabla_iH(x)\|.
\]
We have therefore established
\begin{equation}\label{eq:smooth-bpg-step-comparison}
    \frac{\alpha}
    {\widehat L_i+\alpha L_i^R}
    \|\nabla_iH(x)\|
    \leq
    \|T_i(x;\alpha)-x\|
    \leq
    \frac{\alpha}
    {\widehat L_i-\alpha\varrho_i}
    \|\nabla_iH(x)\|.
\end{equation}
}
    
    Using~\eqref{eq:smooth-bpg-step-comparison} and
    \eqref{eq:smooth-bpg-sampling}, we obtain
    \[
    \begin{aligned}
        & \Exp[
            \|x^{k+1}-x^k\|^2
            \mid\cF_k
        ]
        =
        \sum_{i=1}^m
        \mathbb P(i_k=i\mid\cF_k)
        \|T_i(x^k;\alpha_k)-x^k\|^2\\
        \geq &
        \rho\alpha_k^2
        \sum_{i=1}^m
        \frac{\|\nabla_iH(x^k)\|^2}
        {({\widehat L_i}+\bar\alpha L_i^R)^2}
        \geq
        \conb\alpha_k^2
        \|\nabla H(x^k)\|^2,
    \end{aligned}
    \]
    which verifies~\eqref{eq:expbound}. Finally,
    \[
        \|x^{k+1}-x^k\|^2
        \leq
        \frac{\alpha_k^2}{({\widehat L_{i_k}-\alpha_k\varrho_{i_k}})^2}
        \|\nabla_{i_k}H(x^k)\|^2
        \leq
        \conc\alpha_k^2
        \|\nabla H(x^k)\|^2,
    \]
    which verifies~\eqref{eq:basic_4}.
\end{proof}

Since $h=H$ and $g=\|\nabla H\|^2$, the correspondence between KL inequalities and unified error bounds is
almost the same as the last application. If $\varphi$ is a global KL desingularization function of $H$,
then Assumption~\ref{assmp:Global-EB} holds with
\[
    \phi_{\B}(t)
    :=
    \frac{1}{(\varphi'(t))^2}.
\]
Likewise, if $\varphi_{\bar{x}}$ is a local KL desingularization function
at a critical point $\bar{x}$, then Assumption~\ref{assmp:local-EB} holds
with
\[
    \phi_{\bar{x}}(t)
    :=
    \frac{1}{(\varphi_{\bar{x}}'(t))^2}.
\]
The associated inverse smoothing functions are
\[
    \psi(t)
    :=
    \int_t^{t_0}(\varphi'(s))^2\,\mathrm ds,
    \qquad
    \psi_{\bar{x}}(t)
    :=
    \int_t^{t_0}
    (\varphi_{\bar{x}}'(s))^2\,\mathrm ds,
\]
and hence
\[
    \Psi_{\B}=\psi,
    \qquad
    \Phi_{\B}=\varphi,
    \qquad
    \Psi_{\bar{x}}=\psi_{\bar{x}},
    \qquad
    \Phi_{\bar{x}}=\varphi_{\bar{x}}.
\]

\begin{theorem}[Global rates for convex definable objectives]
\label{thm:smooth-random-bpg-global}
    Consider the method~\eqref{eq:smooth-random-bpg} with stepsizes $\alpha_k$ satisfying Assumption~\ref{assmp:step-size}. Suppose that $H$ is
    coercive, convex, definable and has a locally Lipschitz continuous gradient, and that
    Assumption~\ref{assmp:smooth-bpg-smoothness} and
    conditions~\eqref{eq:smooth-bpg-stepsize} and
    \eqref{eq:smooth-bpg-sampling} hold. Let $\varphi$ be the global KL
    desingularization function given by
    Proposition~\ref{prop:global-KL-convex-definable}. Fix $t_0>0$ in
    the domain of $\varphi$, and define $\phi(0):=0$,
    \[
        \phi(t):=\frac{1}{(\varphi'(t))^2}
        \quad\text{and}\quad
        \psi(t):=\int_t^{t_0}(\varphi'(s))^2\,\mathrm ds\quad\forall t> 0.
    \]
    Then $1/\sqrt{\phi}$ is integrable near zero, whereas $1/\phi$ is
    not integrable near zero. Consequently, all conclusions of 
    Theorem~\ref{th_general_exp_rate1} hold with functions $h = H $, $g = \|\nabla H \|^2$, {$\Psi_{\B}=\psi$}, {$\Phi_{\B}=\varphi$}, and constants $c_1,c_2,c_3$ given in~\eqref{eq:smooth-bpg-constants}.
\end{theorem}

\begin{proof}
    Proposition~\ref{prop:smooth-random-bpg-basic} verifies
    Assumption~\ref{assmp:basic} with the choices and constants stated
    above. The remainder of the proof is identical to that of
    Theorem~\ref{thm:matrix-sketched-GD-global}, with $f$ replaced by
    $H$. In particular,
    Propositions~\ref{prop:global-KL-convex-definable}
    and~\ref{prop:KL-to-UEB}\ref{prop:KL-to-UEB-global} yield the GUEB
    with modulus $\phi=1/(\varphi')^2$, while
    Theorem~\ref{Th_KL_general}, applied at a boundary point of
    $\argmin H$, gives the non-integrability of
    $1/\phi=(\varphi')^2$. The result therefore follows from
    Theorem~\ref{th_general_exp_rate1}.
\end{proof}

\begin{theorem}[Local convergence and rates for definable objectives]
\label{thm:smooth-random-bpg-local}
    Consider method~\eqref{eq:smooth-random-bpg} with stepsizes $\alpha_k$ satisfying Assumption~\ref{assmp:step-size}. Suppose that $H$ is
    coercive, definable and has a locally Lipschitz continuous gradient, and that
    Assumption~\ref{assmp:smooth-bpg-smoothness},
    conditions~\eqref{eq:smooth-bpg-stepsize} and
    \eqref{eq:smooth-bpg-sampling} hold.
    For each critical point
    $\bar{x}$, let $\varphi_{\bar{x}}:[0,\eta_{\bar{x}})\to\mathbb R_+$ be a local KL desingularization function given by
    Proposition~\ref{prop:local-KL-definable}. Fix
    $t_{0,\bar{x}}\in(0,\eta_{\bar{x}})$, and define $\phi_{\bar{x}}(0):=0$,
    \[
        \phi_{\bar{x}}(t)
        :=
        \frac{1}{(\varphi_{\bar{x}}'(t))^2}
        \quad\text{and}\quad
        \psi_{\bar{x}}(t)
        :=
        \int_t^{t_{0,\bar{x}}}
        (\varphi_{\bar{x}}'(s))^2\,\mathrm ds \forall t> 0.
    \]
    Then $1/\sqrt{\phi_{\bar{x}}}$ is integrable near zero for every
    critical point $\bar{x}$. Moreover, if $\bar{x}$ is not a local
    maximizer, then $1/\phi_{\bar{x}}$ is not integrable near zero.
    Consequently, all conclusions of Corollary~\ref{cor:localEB} hold with functions $h = H $, $g = \|\nabla H \|^2$, $\Psi_{\bar{x}}=\psi_{\bar{x}}$, $\Phi_{\bar{x}}=\varphi_{\bar{x}}$, and constants $c_1, c_2, c_3$ given in~\eqref{eq:smooth-bpg-constants}.
    In particular, the iterates converge almost surely to a random
    critical point $x^*$ of $H$. On every such convergent path, either $x^*$ is a local maximizer and the
    sequence terminates at $x^*$ after finitely many iterations, or
    $x^*$ is not a local maximizer and the rate estimates in
    Corollary~\ref{cor:localEB}\ref{cor:localEB-ii} hold.
\end{theorem}

\begin{proof}
    Proposition~\ref{prop:smooth-random-bpg-basic} verifies
    Assumption~\ref{assmp:basic} with the choices and constants stated
    above. The remainder of the proof follows exactly as in
    Theorem~\ref{thm:matrix-sketched-GD-local}, with $f$ replaced by
    $H$. Namely,
    Propositions~\ref{prop:local-KL-definable}
    and~\ref{prop:KL-to-UEB}\ref{prop:KL-to-UEB-local} yield the required
    LUEB at every critical point, and
    Corollary~\ref{cor:localEB} gives almost-sure convergence. If the
    limit is not a local maximizer, Theorem~\ref{Th_KL_general} gives
    the non-integrability of $1/\phi_{x^*}$ and hence the stated rates;
    if it is a local maximizer, sufficient descent implies finite
    termination, as in the proof of
    Theorem~\ref{thm:matrix-sketched-GD-local}.
\end{proof}

}

\subsection{Randomized Krasnosel'ski\u{\i}--Mann method}
\label{section:RKMM}

Finally, we analyze RKMM \eqref{alg:RKMM} for solving the common fixed point problem~\eqref{CF-P} using our theoretical framework. To begin, we recall the notion of firm quasi-nonexpansiveness~\cite[Definition~4.1]{bauschke2017convex}.

\begin{definition}[Firm quasi-nonexpansiveness]
Let $B\subseteq \mathbb{R}^d$ be a nonempty subset. Then, $T:B\rightarrow \mathbb{R}^d$ is said to be firmly quasi-nonexpansive if 
    \begin{equation*}
        \|Tx-y\|^2+\|Tx-x\|^2\leq\; \|x-y\|^2\quad \forall y\in \Fix T,\  x\in B. 
    \end{equation*}
\end{definition}

\revise{Assume that, in~\eqref{alg:RKMM}, the relaxation parameters satisfy
\begin{equation}\label{eq:RKMM-stepsize}
    0<\beta_k\leq\bar\beta<2
    \qquad
    \forall k\geq0,
\end{equation}
and that the sampling rule satisfies
\begin{equation}\label{eq:RKMM-sampling}
    \mathbb P(i_k=i\mid\cF_k)\geq\rho>0
    \qquad
    \forall i\in[m],\ k\geq0.
\end{equation}
For this application, we set
\[
    h(x):=\dist^2(x,F),
    \qquad
    g(x):=\max_{i\in[m]}\|T_i(x)-x\|^2.
\]
Let $x^0\notin F$, set $h_0:=h(x^0)$, and let
$p^0:=P_F(x^0)$. We work on the compact Fej\'er region
\begin{equation}\label{eq:RKMM-Fejer-region}
    \B
    :=
    \overline{\mathbb B}
    \bigl(p^0,\dist(x^0,F)\bigr).
\end{equation}

\begin{proposition}\label{prop:RKMM-basic}
    Suppose that $T_1,\dots,T_m$ are firmly quasi-nonexpansive and that
    \eqref{eq:RKMM-stepsize} and~\eqref{eq:RKMM-sampling} hold. Then
    RKMM~\eqref{alg:RKMM} satisfies Assumption~\ref{assmp:basic} on $\B$
    with stepsizes $\alpha_k=\beta_k$ and constants $\cona=2-\bar\beta$, $\conb=\rho$, and $\conc=1$.
    Moreover, $\{x^k\}_{k\geq0}$ is Fej\'er monotone with respect to $F$.
\end{proposition}

\begin{proof}
    By the RKMM update~\eqref{alg:RKMM} and
    \cite[Proposition~4.3]{bauschke2017convex}, for every $z\in F$,
    \[
        \|x^{k+1}-z\|^2
        \leq
        \|x^k-z\|^2
        -
        \frac{2-\beta_k}{\beta_k}
        \|x^{k+1}-x^k\|^2.
    \]
    This directly establishes the Fej\'er monotonicity of
    $\{x^k\}_{k\geq0}$ with respect to $F$. Taking the infimum over
    $z\in F$ further gives
    \[
        \dist^2(x^{k+1},F)
        \leq
        \dist^2(x^k,F)
        -
        \frac{2-\beta_k}{\beta_k}
        \|x^{k+1}-x^k\|^2.
    \]
    Since $0<\beta_k\leq\bar\beta<2$, this yields
    \eqref{eq:basic_1} with $\alpha_k=\beta_k$ and
    $\cona=2-\bar\beta$. The monotonicity of $h(x^k)$ also implies that
    $\{x^k\}_{k\geq0}\subset\B$. Moreover,
    \[
    \begin{aligned}
        \mathbb{E}\!\left[
            \|x^{k+1}-x^k\|^2
            \mid \mathcal{F}_k
        \right]
        &=
        \sum_{i\in[m]}
        \mathbb{P}(i_k=i\mid\mathcal{F}_k)\,
        \beta_k^2\|T_i(x^k)-x^k\|^2 \geq
        \rho\,\beta_k^2
        \max_{i\in[m]}\|T_i(x^k)-x^k\|^2,
    \end{aligned}
    \]
    which verifies~\eqref{eq:expbound} with
    $\alpha_k=\beta_k$ and $\conb=\rho$. Finally,
    \[
        \|x^{k+1}-x^k\|^2
        =
        \beta_k^2\|T_{i_k}(x^k)-x^k\|^2
        \leq
        \beta_k^2
        \max_{i\in[m]}\|T_i(x^k)-x^k\|^2,
    \]
    which verifies~\eqref{eq:basic_4} with
    $\alpha_k=\beta_k$ and $\conc=1$.
\end{proof}


If $T_1,\dots,T_m$ are continuous, since the Fej\'er region $\B$ is compact, $h$ and $g$ are continuous on $\B$. Moreover, $g^{-1}\{0\} = h^{-1}\{0\} = F$.
Hence, Theorem~\ref{thm:GUEB-existence} guarantees the existence of a GUEB modulus $\phi_{\B}$ such that Assumption~\ref{assmp:Global-EB} holds on
$\B$. Firm quasi-nonexpansiveness further imposes the following restriction on
any such modulus.

\begin{proposition}\label{prop:RKMM-nonintegrability}
    Under the assumptions of
    Propositions~\ref{prop:RKMM-basic}, any GUEB modulus $\phi_{\B}$ on $\B$ satisfies $\phi_{\B}(t)\leq t$ for all $t\in(0,h_0]$.
    Consequently, $1/\phi_{\B}$ is not integrable near zero.
\end{proposition}

\begin{proof}
    Firm quasi-nonexpansiveness gives $g(x)\leq h(x)$ for any $x\in\B$.
    For any $t\in(0,h_0]$, let
    \[
        x_t
        :=
        p^0+
        \sqrt{\frac{t}{h_0}}\,(x^0-p^0).
    \]
    Since $F$ is closed and convex, $P_F(x_t)=p^0$, so
    $x_t\in\B$ and $h(x_t)=t$. By the GUEB inequality,
    \[
        \phi_{\B}(t)
        \leq
        g(x_t)
        \leq
        h(x_t)
        =
        t.
    \]
    Hence $1/\phi_{\B}(t)\geq1/t$ on $(0,h_0]$, and thus
    $1/\phi_{\B}$ is not integrable near zero.
\end{proof}


\begin{theorem}[Global convergence rates for RKMM]
\label{thm:RKMM-global}
    Suppose that the assumptions of
    Propositions~\ref{prop:RKMM-basic} and Assumption~\ref{assmp:step-size} hold. Then 
    \[
        \Exp[\dist^2(x^k,F)]
        \leq
        \mathcal \exp\left(
            -\mathfrak h\frac{\rho}{\bar\beta}A_k
        \right)h_0+
        \Psi_{\B}^{-1}\left(
            \Psi_{\B}(h_0)
            +
            \frac{(2-\bar\beta)\rho}{2}A_k
        \right).
    \]
    where $\mathfrak h:=\frac{1-\log 2}{2}$, and there exists an almost surely finite random index $K_1$ such that
    \[
        \dist^2(x^k,F)
        \leq
        \Psi_{\B}^{-1}\left(
            \frac{(2-\bar\beta)\rho}{2}A_k
        \right)
        \qquad
        \forall k\geq K_1
        \quad\as,
    \]
    Moreover, there exists a random variable $x^*\in F$ such that
    $x^k\to x^*$ almost surely and in $L^1$, and
    \begin{equation}\label{eq:RKMM-expected-iterate-rate}
        \Exp[\|x^k-x^*\|]
        \leq2\sqrt{\Exp[\dist^2(x^k,F)]}
        \qquad
        \forall k\geq0,
    \end{equation}
    and there exists an almost surely finite random index $K_2$ such that
    \begin{equation}\label{eq:RKMM-as-iterate-rate}
        \|x^k-x^*\|
        \leq
        2\sqrt{
            \Psi_{\B}^{-1}\left(
                \frac{(2-\bar\beta)\rho}{2}A_k
            \right)
        }
        \qquad
        \forall k\geq K_2
        \quad\as.
    \end{equation}
\end{theorem}

\begin{proof}
    Propositions~\ref{prop:RKMM-basic} and
    \ref{prop:RKMM-nonintegrability} verify the assumptions of
    Theorem~\ref{th_general_exp_rate1} in the non-integrable regime. It remains only to derive the iterate estimates. Since $\{\dist^2(x^k,F)\}_{k\geq0}$ is non-increasing and converges to zero
    almost surely, Fej\'er monotonicity and compactness imply that
    $x^k\to x^*$ for some $x^*\in F$. By \cite[Theorem 3.3(iv)]{BauschkeBorwein1993VonNeumann}, we have
    \begin{equation}\label{eq:Fejer-distance-transfer}
        \|x^k-x^*\|
        \leq
        2\dist(x^k,F)
        =
        2\sqrt{h(x^k)}.
    \end{equation}
    Taking expectations and applying Jensen's inequality proves
    \eqref{eq:RKMM-expected-iterate-rate}; combining
    \eqref{eq:Fejer-distance-transfer} with the almost-sure bound for $h$
    gives~\eqref{eq:RKMM-as-iterate-rate}. 
\end{proof}

For RKMM, if we assume Fej\'er monotonicity, a slightly sharper rate for the iterates can de derived. However, we should emphasize that this does not indicate looseness of our MRA analysis, since it does not assume Fej\'er monotonicity. Furthermore, the derivation still relies on the MRA framework to establish rates for $h(x^k)$, which are then transferred to the iterates using Fej\'er monotonicity.


\begin{example}[H\"olderian error bounds]
    Suppose that $T_1,\dots,T_m$ satisfy the H\"olderian error bound on
    $\B$ with exponent $\theta\in(0,1]$,
    \[
        \dist(x,F)
        \leq
        r\max_{i\in[m]}\|T_i(x)-x\|^\theta
        \qquad
        \forall x\in\B,
    \]
    for some $r>0$. Then one may take $\phi_{\B}(t)=(r^{-2}t)^{1/\theta}$, for which
    \[
        \Psi_{\B}(t)
        =
        \begin{cases}
            r^2\log(t_0/t), & \theta=1,\\[1mm]
            \displaystyle
            \frac{\theta r^{2/\theta}}{1-\theta}
            \left(
                t^{-\frac{1-\theta}{\theta}}
                -
                t_0^{-\frac{1-\theta}{\theta}}
            \right),
            & \theta\in(0,1).
        \end{cases}
    \]
    Consequently, when $\theta=1$, both the feasibility error and the
    iterates converge geometrically, in expectation and almost surely.
    When $\theta\in(0,1)$,
    \[
        \dist^2(x^k,F)
        =
        O\left(A_k^{-\frac{\theta}{1-\theta}}\right)
        \quad\text{and}\quad
        \|x^k-x^*\|
        =
        O\left(
            A_k^{-\frac{\theta}{2(1-\theta)}}
        \right),
    \]
    both in expectation and almost surely. The decay orders obtained above, both in expectation and almost surely, coincide with the corresponding deterministic rates for averaged and quasi-cyclic fixed-point iterations under H\"olderian error bounds \cite{borwein2017convergence,Liu2024ConcreteCR} and, for the linear specialization $\theta=1$ \cite{liang2016convergence}. Meanwhile, the rates for $\Exp[\dist^2(x^k,F)]$ also recover those obtained through the conditional-expectation/Jensen argument described in Section~5.2.1; see \cite[Theorem 3.12]{hu2022quasi}. When $\theta=1$, they specialize to the standard linear expected rates \cite{Nedic,Necoara2019,Strohmer2009Randomized,LeventhalLewis2010,GowerRichtarik2015,hermer2019random,hermer2019random}.
\end{example}
}

\section*{Declarations} 

\subsection*{Funding}
This work is supported in part by the Hong Kong Research Grants Council under the GRF project 17310124.

\subsection*{Competing interests}
The authors have no relevant financial or non-financial interests to disclose.


\appendix
\appendixpage

\section{\revise{Failure of Essential Cyclicity under I.I.D. Sampling}}\label{app:ess-cyclic}
\revise{
\begin{definition}[Essential cyclicity]
Let $\mathcal I=\{1,\ldots,m\}$ be a finite set, and let $(I_k)_{k\ge 0}$ be a sequence taking values in $\mathcal I$. We say that $(I_k)_{k\ge 0}$ is essentially cyclic if there exists an integer $T\ge 1$ such that, for every $k\ge 0$,
\[
\mathcal I \subseteq \{I_k,I_{k+1},\ldots,I_{k+T-1}\}.
\]
Equivalently, every element of $\mathcal I$ appears at least once in every window of $T$ consecutive selections.
\end{definition}

\begin{theorem}[I.i.d. sampling is not essentially cyclic]
Let $\mathcal I=\{1,\ldots,m\}$ with $m\ge 2$, and let $(I_k)_{k\ge 0}$ be an i.i.d. sequence taking values in $\mathcal I$. Then $(I_k)_{k\ge 0}$ is not essentially cyclic almost surely. More precisely,
\[
\mathbb P\left(
\exists\,T\ge 1 \text{ such that }
\mathcal I \subseteq \{I_k,I_{k+1},\ldots,I_{k+T-1}\}
\text{ for all } k\ge 0
\right)=0.
\]
In fact, for every fixed $T\ge 1$, almost surely there exist infinitely many windows of length $T$ that miss at least one element of $\mathcal I$.
\end{theorem}

\begin{proof}
Let $p_i=\mathbb P(I_0=i)$ for $i\in\mathcal I$. Since $m\ge 2$, there exists an element $i_0\in\mathcal I$ such that $p_{i_0}<1$.
Fix $T\ge 1$. For each $n\ge 0$, define the event
\[
E_n^{(T)}
:=
\{I_{nT}\neq i_0,\ I_{nT+1}\neq i_0,\ \ldots,\ I_{nT+T-1}\neq i_0\}.
\]
The events $(E_n^{(T)})_{n\ge 0}$ are independent, because they depend on disjoint groups of i.i.d. random variables. Moreover,
\[
\mathbb P(E_n^{(T)})=(1-p_{i_0})^T>0
\qquad \text{for every } n\ge 0.
\]
Hence
\[
\sum_{n=0}^{\infty}\mathbb P(E_n^{(T)})=\infty.
\]
By the second Borel--Cantelli lemma, $E_n^{(T)}$ occurs infinitely often almost surely. Therefore, almost surely, there are infinitely many windows of length $T$ in which the element $i_0$ does not appear.

Thus, for this fixed $T$, the sequence cannot satisfy the essential cyclicity condition with window length $T$. Since the set of possible values of $T$ is countable, taking the intersection over all $T\ge 1$ shows that, almost surely, no finite $T$ satisfies the essential cyclicity condition. This proves the claim.
\end{proof}
}

\bibliographystyle{abbrv}
\bibliography{refs}

\end{document}